\documentclass[a4paper,11pt,reqno]{amsart}

\usepackage{a4wide}
\usepackage[T1]{fontenc}
\usepackage[ngerman,english]{babel}
\usepackage{amsmath}
\usepackage{amsthm}
\usepackage{amssymb}
\usepackage{amsfonts}
\usepackage{amsxtra}
\usepackage[breaklinks]{hyperref}
\usepackage{color}
\usepackage{enumitem}
\usepackage{tabularx}
\usepackage[all]{xy}
\usepackage{float}
\usepackage{nicefrac}
\usepackage{rotating}
\usepackage{multirow}
\usepackage{booktabs,longtable}

\usepackage{tikz}
\usetikzlibrary{matrix}

\newcommand{\td}{\mathrm{d}}
\newcommand{\Lie}{\textup{Lie}}
\newcommand{\const}{\textup{const}}
\newcommand{\Ad}{\textup{Ad}}
\newcommand{\ad}{\textup{ad}}
\let\Re\relax
\DeclareMathOperator{\Re}{Re}
\DeclareMathOperator{\Det}{Det}
\newcommand{\id}{\textup{id}}
\DeclareMathOperator{\tr}{tr}
\newcommand{\rank}{\textup{rank}}
\newcommand{\cl}{\textup{cl}}
\newcommand{\Hom}{\textup{Hom}}
\DeclareMathOperator{\Tr}{Tr}
\renewcommand{\1}{\mathbf{1}}
\newcommand{\Ind}{\textup{Ind}}

\newcommand{\Sym}{\textup{Sym}}
\newcommand{\Skew}{\textup{Skew}}
\newcommand{\Herm}{\textup{Herm}}
\newcommand{\Str}{\textup{Str}}
\newcommand{\str}{\mathfrak{str}}

\newcommand{\GL}{\textup{GL}}
\newcommand{\gl}{\mathfrak{gl}}

\renewcommand{\sl}{\mathfrak{sl}}

\newcommand{\so}{\mathfrak{so}}

\newcommand{\End}{\textup{End}}
\newcommand{\Kfinite}{{K-\textup{finite}}}

\newcommand{\CC}{\mathbb{C}}

\newcommand{\FF}{\mathbb{F}}
\newcommand{\HH}{\mathbb{H}}

\newcommand{\OO}{\mathbb{O}}

\newcommand{\RR}{\mathbb{R}}

\newcommand{\ZZ}{\mathbb{Z}}

\newcommand{\calE}{\mathcal{E}}
\newcommand{\calF}{\mathcal{F}}
\newcommand{\calK}{\mathcal{K}}
\newcommand{\calO}{\mathcal{O}}

\newcommand{\calS}{\mathcal{S}}

\newcommand{\calB}{\mathcal{B}}
\newcommand{\calH}{\mathcal{H}}
\newcommand{\calP}{\mathcal{P}}
\newcommand{\calU}{\mathcal{U}}

\newcommand{\calV}{\mathcal{V}}

\newcommand{\fraka}{\mathfrak{a}}

\newcommand{\frakg}{\mathfrak{g}}
\newcommand{\frakh}{\mathfrak{h}}
\newcommand{\frakk}{\mathfrak{k}}
\newcommand{\frakl}{\mathfrak{l}}
\newcommand{\frakm}{\mathfrak{m}}
\newcommand{\frakn}{\mathfrak{n}}
\newcommand{\frako}{\mathfrak{o}}

\newcommand{\frakq}{\mathfrak{q}}
\newcommand{\fraks}{\mathfrak{s}}
\newcommand{\frakt}{\mathfrak{t}}

\newcommand{\JTP}[3]{\left\{#1,#2,#3\right\}}
\newcommand{\Set}[2]{\left\{#1\,\middle|\,#2\right\}} 
\newcommand{\set}[2]{\{#1\,|\,#2\}}
\newcommand{\half}{{\nicefrac{1}{2}}}
\newcommand{\B}[2]{B_{#1,\,#2}}
\renewcommand{\b}[1]{{\bf #1}}

\theoremstyle{plain}
\newtheorem{theorem}{Theorem}[section]
\newtheorem{proposition}[theorem]{Proposition}
\newtheorem{lemma}[theorem]{Lemma}
\newtheorem{corollary}[theorem]{Corollary}

\newtheorem{thmalph}{Theorem}[section]

\theoremstyle{definition}

\newtheorem{remark}[theorem]{Remark}

\numberwithin{equation}{section}

\title[Bessel operators on Jordan pairs and small representations]{Bessel operators on Jordan pairs and small representations of semisimple Lie groups}

\author{Jan M\"ollers}
\address{Department Mathematik, FAU Erlangen--N\"{u}rnberg, Cauerstr. 11, 91058 Erlangen, Germany}
\email{moellers@math.fau.de}

\author{Benjamin Schwarz}
\address{Institut f\"ur Mathematik, Universit\"at Paderborn, Warburger Str. 100, 33098 Paderborn, Germany}
\email{bschwarz@math.upb.de}

\begin{document}

\begin{abstract}
We provide a uniform construction of $L^2$-models for all small unitary representations in degenerate principal series of semisimple Lie groups which are induced from maximal parabolic subgroups with abelian nilradical. This generalizes previous constructions to the case of a maximal parabolic subgroup which is not necessarily conjugate to its opposite, and hence the previously used Jordan algebra methods have to be generalized to Jordan pairs.\\
The crucial ingredients for the construction of the $L^2$-models are the Lie algebra action and the spherical vector. Working in the so-called Fourier transformed picture of the degenerate principal series, the Lie algebra action is given in terms of Bessel operators on Jordan pairs. We prove that precisely for those parameters of the principal series where small quotients occur, the Bessel operators are tangential to certain submanifolds. Further, we show that the small quotients are unitarizable if and only if these submanifolds carry equivariant measures. In this case we can express the spherical vectors in terms of multivariable K-Bessel functions, and some delicate estimates for these Bessel functions imply the existence of the $L^2$-models.
\end{abstract}



\maketitle

\section*{Introduction}

In the representation theory of semisimple Lie groups the principal series plays a central role. It is constructed by parabolic induction from minimal parabolic subgroups, inducing from finite-dimensional representations of the Levi factor, and forms a standard family of representations of the group. The celebrated Harish-Chandra Subquotient Theorem asserts that every irreducible representation on a Banach space, in particular every irreducible unitary representation, is equivalent to a subquotient of a representation in the principal series. In this sense, every irreducible unitary representation can be constructed through the principal series.

However, it turns out to be a highly non-trivial problem to find the irreducible unitarizable constituents of the principal series, in particular when the rank of the group is large. Moreover, the unitary realizations of representations obtained in this way are of an algebraic nature, and often rather complicated and difficult to work with.

A more accessible setting arises if one replaces the minimal parabolic subgroup by a larger parabolic subgroup, for instance a maximal one. The corresponding parabolically induced representations are called \textit{degenerate principal series representations} and they essentially depend on a single complex parameter. This simplifies the problem of finding the irreducible constituents and determining the unitarizable ones among them. In various cases this problem has been solved, and in this way particularly small unitarizable irreducible representations are obtained, among them so-called \textit{minimal representations}.

If additionally the nilradical of the parabolic subgroup is abelian, it turns out that the Euclidean Fourier transform on the nilradical can be used to construct very explicit models of the small unitarizable quotients of the degenerate principal series, realized on Hilbert spaces of $L^2$-functions on certain orbits of the Levi factor of the parabolic. These so-called \textit{$L^2$-models} have been obtained by Rossi--Vergne~\cite{RV76}, Sahi~\cite{Sah92}, Dvorsky--Sahi~\cite{DS99,DS03}, Kobayashi--{\O}rsted~\cite{KO03c} and Barchini--Sepanski--Zierau~\cite{BSZ06} in all cases where the maximal parabolic subgroup is conjugate to its opposite parabolic subgroup. Their constructions are different in nature, and in particular do not easily generalize to the case where the parabolic subgroup is not conjugate to its opposite.

In this paper we provide a uniform approach to $L^2$-models of small representations that works for all maximal parabolic subgroups with abelian nilradical. Our key ingredients are twofold:
\begin{enumerate}
\item\label{Ingredients:1} We extend the Bessel operators, previously defined for Jordan algebras, to the context of Jordan pairs. These operators form the crucial part of the Lie algebra action.
\item\label{Ingredients:2} We express the spherical vectors in the degenerate principal series and its Fourier transformed picture in terms of multivariable K-Bessel functions.
\end{enumerate}
That the Bessel operators on Jordan algebras appear in the Lie algebra action of small representations was first observed by Mano~\cite{Man08} and later used by Hilgert--Kobayashi--M\"{o}llers~\cite{HKM14} to uniformly construct $L^2$-models for minimal representations. The occurrence of classical one-variable K-Bessel functions as spherical vectors appears first in the work by Dvorsky--Sahi~\cite{DS99} (see also \cite{HKM14,KO03c}). In this sense, our ingredients \eqref{Ingredients:1} and \eqref{Ingredients:2} are generalizations of previously used techniques.

In this setting, we show that small constituents arise precisely for those parameters where the Bessel operators are tangential to certain submanifolds. Further, these small constituents turn out to be unitarizable if and only if there exist equivariant measures on these submanifolds. In this case the Bessel operators are shown to be symmetric with respect to these measures, and the Lie algebra acts by formally skew-adjoint operators with respect to the $L^2$-inner product. To integrate this Lie algebra representation to the group level, we use the explicit description of the spherical vector in terms of a multivariable K-Bessel function and prove some delicate estimates for this Bessel function that ensure that the Lie algebra representation indeed integrates to an irreducible unitary representation on the $L^2$-space of the submanifold, providing an $L^2$-model of the small representations.\\

We now explain our results in more detail. Let $G$ be a connected simple real non-compact Lie group with maximal parabolic subgroup $P=MAN\subseteq G$ whose nilradical $N$ is abelian. Then $G/P$ is a compact Riemannian symmetric space. Write $\overline{P}=MA\overline N$ for the opposite parabolic and $\fraka$, $\frakn$ and $\overline\frakn$ for the Lie algebras of $A$, $N$ and $\overline N$, then $(\frakn,\overline\frakn)$ naturally carries the structure of a \textit{Jordan pair}. This Jordan pair is associated to a Jordan algebra if and only if $P$ and $\overline P$ are conjugate. The rank $r$ of $(\frakn,\overline\frakn)$ agrees with the rank of the compact symmetric space $G/P$.

For $\nu\in\fraka_\CC^*$ we form the degenerate principal series (smooth normalized parabolic induction)
$$ \pi_\nu = \Ind_P^G(\1\otimes e^\nu\otimes\1) $$
and realize it on a space $I(\nu)\subseteq C^\infty(\overline\frakn)$ of smooth functions on $\overline\frakn$. The structure of these representations has been completely determined by Sahi~\cite{Sah92} for $G$ Hermitian (see also Johnson~\cite{Joh90,Joh92} and {\O}rsted--Zhang~\cite{OZ95}), by Sahi~\cite{Sah95} and Zhang~\cite{Zha95} for $G$ non-Hermitian and $P$ and $\overline P$ conjugate, and by the authors~\cite{MS12} for the remaining cases. The precise results differ from case to case, but they all have a common feature: There exist
$$ 0<\nu_{r-1}<\ldots<\nu_1<\nu_0=\rho=\tfrac{1}{2}\Tr\ad_\frakn\in\fraka^* $$
such that the representation $\pi_\nu$ is reducible for $\nu=\nu_k$, $0\leq k\leq r-1$, and has a unique irreducible quotient $J(\nu_k)$ which is a small representation of rank $k$ (see Theorem~\ref{thm:StructureDegPrincipalSeries}). Here a representation is said to have \textit{rank $k$} if its $K$-types, identified with their highest weights, are contained in a single $k$-dimensional affine subspace. In particular, $J(\nu_0)$ is the trivial representation.

Strictly speaking, there are two cases in which this statement only holds if slightly modified. More precisely, if $G$ is Hermitian or if $G=SO_0(p,q)$, $p\neq q$, one has to allow parabolic induction from a possibly non-trivial unitary character of $M$ as well. Since these two special cases were treated before in great detail (see \cite{HKM14,KM08,RV76}), we exclude them in the introduction, but we remark that similar statements hold and therefore our construction is uniform. In the rest of the paper it is clearly stated which results hold in general and which need slight modifications, and we also provide the corresponding references for the respective statements.

To construct $L^2$-models of the representations $J(\nu_k)$ we use the Euclidean Fourier transform $\calF:\calS'(\overline\frakn)\to\calS'(\frakn)$ corresponding to the non-degenerate pairing $\frakn\times\overline\frakn\to\RR$ given by the Killing form. Twisting with the Fourier transform we obtain a new realization $\tilde\pi_\nu=\calF\circ\pi_\nu\circ\calF^{-1}$ of $\pi_\nu$ on $\tilde I(\nu)=\calF(I(\nu))\subseteq\calS'(\frakn)$. This realization is called the \textit{Fourier transformed picture}.

In the realization $\tilde\pi_\nu$ the Levi factor $L=MA$ essentially acts by the left-regular representation of the adjoint action of $L$ on $\frakn$. Under the $L$-action $\frakn$ decomposes into $r+1$ orbits $\calO_0,\ldots,\calO_r$, ordered such that $\calO_j$ is contained in the closure of $\calO_k$ if $j\leq k$.

In Proposition~\ref{prop:FTpicturealgebraaction} we also compute the Lie algebra action of $\tilde\pi_\nu$, showing that the non-trivial part is given in terms of so-called \textit{Bessel operators} $\calB_\lambda$ on the Jordan pair $(\frakn,\overline\frakn)$ for $\lambda=2(\rho-\nu)$. These are $\overline\frakn$-valued second order differential operators on $\frakn$ and can be defined by the simple formula
$$ \calB_\lambda=Q\left(\frac{\partial}{\partial x}\right)x+\lambda\frac{\partial}{\partial x}, \qquad \lambda\in\CC, $$
where $\frac{\partial}{\partial x}:C^\infty(\frakn)\to C^\infty(\frakn)\otimes\overline\frakn$ denotes the gradient with respect to the (renormalized) Killing form $\frakn\times\overline\frakn\to\RR$ and $Q:\overline\frakn\to\Hom(\frakn,\overline\frakn)$ the quadratic operator of the Jordan pair $(\frakn,\overline\frakn)$.

We first relate the existence of small quotients of $\tilde I(\nu)$ to the Bessel operators.

\begin{thmalph}[see Theorem~\ref{thm:pullbackBessel}]\label{thm:IntroTangential}
Let $0\leq k\leq r-1$, then $\tilde I(\nu)$ has an irreducible quotient of rank $k$ if and only if the Bessel operator $\calB_\lambda$ for $\lambda=2(\rho-\nu)$ is tangential to the orbit $\calO_k$. This is precisely the case for $\nu=\nu_k$.
\end{thmalph}

The previous theorem implies that the Lie algebra representation $\td\tilde\pi_{\nu_k}$ has $\tilde I_0(\nu_k)=\set{f\in\tilde I(\nu)}{f|_{\calO_k}=0}$ as a subrepresentation and hence induces a representation on the quotient $\tilde J(\nu_k)=\tilde I(\nu_k)/\tilde I_0(\nu_k)$ which is identified with $\set{f|_{\calO_k}}{f\in\tilde I(\nu)}$. This is the unique irreducible quotient of $\tilde I(\nu_k)$ which is a small representation of rank $k$. We now turn to the question whether this quotient is unitarizable.

\begin{thmalph}[see Theorem~\ref{thm:equivariantmeasures}]\label{thm:IntroMeasures}
Let $0\leq k\leq r-1$, then the irreducible quotient $\tilde J(\nu_k)$ is unitarizable if and only if the orbit $\calO_k$ carries an $L$-equivariant measure $\td\mu_k$. This is precisely the case for $(\frakg,\frakl)\not\simeq(\sl(p+q,\FF),\fraks(\gl(p,\FF)\oplus\gl(q,\FF)))$, $\FF=\RR,\CC,\HH$ with $p\neq q$, where $\frakl$ denotes the Lie algebra of $L$.
\end{thmalph}

The space $L^2(\calO_k,\td\mu_k)$ is an obvious candidate for the Hilbert space completion of the unitarizable quotient $\tilde J(\nu_k)$. In this case the Lie algebra has to act by skew-adjoint operators on $L^2(\calO_k,\td\mu_k)$. The cruciual part of this action is given in terms of the Bessel operators which have to be self-adjoint on $L^2(\calO_k,\td\mu_k)$.

\begin{thmalph}[see Theorem~\ref{thm:BesselSymmetricOnOrbits}]\label{thm:IntroSymmetric}
Let $0\leq k\leq r-1$ and assume that the orbit $\calO_k$ carries an $L$-equivariant measure $\td\mu_k$. Then the Bessel operator $\calB_\lambda$, $\lambda=2(\rho-\nu_k)$, is symmetric on $L^2(\calO_k,\td\mu_k)$.
\end{thmalph}

We finally obtain the $L^2$-model for $\tilde J(\nu_k)$:

\begin{thmalph}[see Theorems~\ref{thm:UniRepOnL2} and \ref{thm:Intertwiner}]\label{thm:IntroL2Models}
Let $0\leq k\leq r-1$, and assume that the orbit $\calO_k$ carries an $L$-equivariant measure $\td\mu_k$. Then the operator
$$ T_k:\tilde I(\nu_k)\to\tilde I(-\nu_k), \quad f\mapsto f|_{\calO_k}\td\mu_k, $$
is defined and intertwining for $\tilde\pi_{\nu_k}$ and $\tilde\pi_{-\nu_k}$. Its kernel is the subrepresentation $\tilde I_0(\nu_k)$ and its image is a subrepresentation $\tilde I_0(-\nu_k)\subseteq\tilde I(-\nu_k)$ isomorphic to the irreducible quotient $\tilde J(\nu_k)=\tilde I(\nu_k)/\tilde I_0(\nu_k)$. The subrepresentation $\tilde I_0(-\nu_k)$ is contained in $L^2(\calO_k,\td\mu_k)$ and the $L^2$-inner product is a $\tilde\pi_{-\nu_k}$-invariant Hermitian form, so that the representation $(\tilde\pi_{-\nu_k},\tilde I_0(-\nu_k))$ extends to an irreducible unitary representation of $G$ on $L^2(\calO_k,\td\mu_k)$.
\end{thmalph}

For $k=1$ the representation $J(\nu_1)$ is in many cases the minimal representation of $G$ (see Corollary~\ref{cor:MinRep} for the precise statement).

We remark that in the case where $P$ and $\overline P$ are conjugate, the operator $T_k$ is a regularization of the Knapp--Stein family of standard intertwining operators $\tilde\pi_\nu\to\tilde\pi_{-\nu}$. However, in the case where $P$ and $\overline P$ are not conjugate, such a family does not exist and the operators $T_k$ are a geometric version of the non-standard intertwining operators constructed in \cite{MS12}.

In addition to the previously stated results on the Bessel operators, the proof of Theorem~\ref{thm:IntroL2Models} consists of a detailed analysis of the $K$-spherical vector in $\tilde I(\nu)$, where $K$ is a maximal compact subgroup of $G$. In Theorem~\ref{thm:PsiInL1} we express the spherical vector $\psi_\nu\in\tilde I(\nu)$ for $\Re\nu>-\nu_{r-1}$ in terms of a multivariable K-Bessel function. It is supported on the whole vector space $\frakn$ and depends holomorphically on $\nu\in\CC$. For its values at $\nu=-\nu_k$, $0\leq k\leq r-1$, we have the following result:

\begin{thmalph}[see Theorems~\ref{thm:PsiInL1} and \ref{thm:PsiInL2}]\label{thm:IntroSphericalVector}
Let $0\leq k\leq r-1$, and assume that the orbit $\calO_k$ carries an $L$-equivariant measure $\td\mu_k$. Then the spherical vector $\psi_{-\nu_k}\in\tilde I(-\nu_k)$ is given by $\psi_{-\nu_k}=\const\times\Psi_k\,\td\mu_k$, where $\Psi_k\in C^\infty(\calO_k)$ is in polar coordinates given by
$$ \Psi_k(mb_t) = K_{\mu_k}\left((\tfrac{t_1}{2})^2,\ldots,(\tfrac{t_k}{2})^2\right), \qquad m\in M\cap K,\,t\in C_k^+. $$
We further have $\Psi_k\in L^2(\calO_k,\td\mu_k)$ and hence the representation $(\tilde\pi_{-\nu_k},L^2(\calO_k,\td\mu_k))$ is spherical.
\end{thmalph}

Here we use polar coordinates $(M\cap K)\times C_k^+\to\calO_k,\,(m,t)\mapsto mb_t$ where $C_k^+=\set{t\in\RR^k}{t_1>\ldots>t_k>0}$, and $K_\mu(s_1,\ldots,s_k)$ is a multivariable K-Bessel function of parameter $\mu=\mu_k\in\RR$. The K-Bessel function can be obtained from the theory of Bessel functions on symmetric cones (see Appendix~\ref{app:KBesselFunctions}), and we prove some involved estimates for these functions that guarantee for instance that $\Psi_k\in L^1(\calO_k,\td\mu_k)\cap L^2(\calO_k,\td\mu_k)$. These estimates are essential in the proof of Theorem~\ref{thm:IntroL2Models}.

We remark that for $k=1$ the function $K_\mu(s_1)$ is essentially the classical K-Bessel function in one variable $s_1\in\RR_+$. Its occurrence as the spherical vector in $L^2$-models for small representations was first observed by Dvorsky--Sahi~\cite{DS99} (see also Kobayashi--{\O}rsted~\cite{KO03c}).\\

Let us comment on the validity of the stated results in the cases of $G$ Hermitian and $G=SO_0(p,q)$. Theorems~\ref{thm:IntroTangential}, \ref{thm:IntroMeasures} and \ref{thm:IntroSymmetric} are in fact proven in general, i.e. also for $G$ Hermitian and $G=SO_0(p,q)$. The proofs use a description of the orbits $\calO_k$ by local coordinates and are carried out in the framework of Jordan pairs. Theorem~\ref{thm:IntroL2Models} also holds for $G$ Hermitian as stated here, but for $G=SO_0(p,q)$ one has to assume that $p+q$ is even. The latter case is treated in detail in \cite{HKM14} which is why we omit the details here. Theorem~\ref{thm:IntroSphericalVector} is not true in the two special cases, in fact the corresponding small representations for $G$ Hermitian and $G=SO_0(p,q)$, $p\neq q$, are not spherical. For $G$ Hermitian they have a non-trivial scalar minimal $K$-type given by an exponential function in the $L^2$-model, and for $G=SO_0(p,q)$, $p\neq q$, the minimal $K$-type is of dimension $>1$ (see Remark~\ref{rem:MinKtypeHermitianAndOpq} for more explanations). The comparable statements for $G$ Hermitian are well-known and can e.g. be found in \cite[Section 2.1]{Moe13}, and for $G=SO_0(p,q)$ we refer to \cite[Chapter 3]{KM08} or \cite[Theorem 2.19~(c) and Proposition 2.24~(b)]{HKM14} for a detailed treatment of the $L^2$-model.\\

Theorem~\ref{thm:IntroTangential}, \ref{thm:IntroMeasures} and \ref{thm:IntroSymmetric} were previously shown in \cite{HKM14} for the case where $P$ and $\overline P$ are conjugate, or equivalently $\frakn$ is a Jordan algebra. For Theorems~\ref{thm:IntroTangential} and \ref{thm:IntroSymmetric} the proofs use zeta functions on Jordan algebras which are not available for Jordan pairs. The proofs we present here work uniformly for all Jordan pairs since they merely use a local description of the orbits $\calO_k$. Theorem~\ref{thm:IntroL2Models} was previously obtained by Rossi--Vergne~\cite{RV76} for $G$ Hermitian (see also Sahi~\cite{Sah92}), and by Dvorsky--Sahi~\cite{DS99,DS03} for $G$ non-Hermitian and $P$ and $\overline P$ conjugate (see also Barchini--Sepanski--Zierau~\cite{BSZ06}), and is new for $P$ and $\overline P$ not conjugate. Theorem~\ref{thm:IntroSphericalVector} seems to be new in this general form, but for $k=1$ the K-Bessel function is (up to renormalization) the classical one-variable K-Bessel function and in this case the corresponding result was first obtained by Dvorsky--Sahi~\cite{DS99}. We remark that multivariable K-Bessel functions also appear in other contexts in representations theory, for instance in the construction of Fock models for scalar type unitary highest weight representations, see \cite{Moe13}.\\

The structure of this paper is as follows. In Section~\ref{sec:degenerateprincipalseries} we recall the construction of the degenerate principal series, the main results about its structure, and the relation to Jordan pairs. In particular, we obtain an explicit formula for the spherical vector in the non-compact picture, see Proposition~\ref{prop:SphericalVectorNonCptPicture}. Section~\ref{sec:OrbitStructure} provides details on the local and global structure of the orbits $\calO_k$, in particular we find explicit local coordinates in Proposition~\ref{prop:LOrbits} and determine the cases where the orbits admit equivariant measures in Theorem~\ref{thm:equivariantmeasures}. Bessel operators on Jordan pairs are defined in Section~\ref{sec:Besseloperator} where we further show that for certain parameters these operators are tangential to the orbits $\calO_k$ (see Theorem~\ref{thm:pullbackBessel}) and symmetric with respect to the equivariant measures (see Theorem~\ref{thm:BesselSymmetricOnOrbits}). Finally, in Section~\ref{sec:L2models} we construct the $L^2$-models by a detailed analysis of the spherical vectors in terms of multivariable K-Bessel functions. Appendix~\ref{app:KBesselFunctions} provides some details about these Bessel functions in the framework of symmetric cones.

\section{Degenerate principal series and Jordan pairs}\label{sec:degenerateprincipalseries}

In this section we fix the necessary notation, recall the relevant reducibility and unitarizability results for degenerate principal series representations and give a description of the representations in terms of Jordan pairs.

\subsection{Degenerate principal series}\label{subsec:degenerateprincipalseries}

Let $G$ be a connected simple real non-compact Lie group with finite center. Assume that $G$ has a maximal parabolic subgroup $P=MAN\subseteq G$ whose nilradical $N\subseteq P$ is abelian. Let $\theta$ be a Cartan involution of $G$ which leaves the Levi subgroup $L=MA$ invariant and denote by $K=G^\theta$ the corresponding maximal compact subgroup of $G$. Then $M\cap K=M^\theta$ is maximal compact in $M$ and $L$, and the generalized flag variety $X=G/P$ is a compact Riemannian symmetric space $X=K/(M\cap K)$, sometimes referred to as a \textit{symmetric $R$-space}. Let $\frakg$, $\frakk$, $\frakl$, $\frakm$, $\fraka$ and $\frakn$ be the Lie algebras of $G$, $K$, $L$, $M$, $A$ and $N$, respectively, and let $\theta$ also denote the Cartan involution on $\frakg$ with respect to $\frakk$. There exists a (unique) grading element $Z_0$ in $\fraka$ such that $\frakg$ decomposes under the adjoint action of $Z_0$ into
\[
	\frakg = \overline\frakn\oplus\frakl\oplus\frakn
\]
with eigenvalues $-1$, $0$, $1$. Then $\overline\frakn = \theta\frakn$, and $\frakl\oplus\frakn$ is the Lie algebra of $P$. The dimension of $X$ is denoted by $n=\dim\overline\frakn=\dim\frakn$. Concerning root data, we fix a maximal abelian subalgebra $\frakt$ in $\frakk\cap(\overline\frakn\oplus\frakn)$. The dimension $r=\dim\frakt$ is called the \emph{rank} of the symmetric space $X$. The adjoint action on $\frakg_\CC$ yields a (restricted) root system $\Phi(\frakg_\CC,\frakt_\CC)$ which is either of type $C$ or of type $BC$. We write
\begin{align*}
	\Phi(\frakg_\CC,\frakt_\CC) = 
	\begin{cases}
		\Set{\tfrac{1}{2}(\pm\gamma_j\pm\gamma_k)}{1\leq j,k\leq r}\setminus\{0\}
		&\text{for type $C_r$}\\
		\Set{\tfrac{1}{2}(\pm\gamma_j\pm\gamma_k),\ \pm\tfrac{1}{2}\,\gamma_j}
		{1\leq j,k\leq r}\setminus\{0\} &\text{for type $BC_r$},
	\end{cases}
\end{align*}
where $\gamma_1,\ldots,\gamma_r$ are strongly orthogonal roots. Depending on the type of this root system, $X$ is called \emph{unital} if $\Phi(\frakg_\CC,\frakt_\CC)$ is of type $C_r$, and \emph{non-unital} otherwise. These notions are due to the Jordan theoretic description of symmetric $R$-spaces, see Section~\ref{sec:JordanTheoryOfSymmRSpaces}. More precisely, in the unital case $\frakn$ is a Jordan algebra whereas in the non-unital case it is only a Jordan triple system.

The adjoint action of $\frakt_\CC$ on $\frakk_\CC$ also yields a root system $\Phi(\frakk_\CC,\frakt_\CC)$, and the \emph{structure constants} of the symmetric $R$-space $X$ are defined by the multiplicities of the roots,
\begin{align}\label{eq:structureconstants}
	\begin{aligned}
	e&=\dim(\frakk_\CC)_{\pm\gamma_j}, &
	d_+&=\dim(\frakk_\CC)_{\tfrac{\gamma_j-\gamma_k}{2}},\\
	b&=\dim(\frakk_\CC)_{\pm\tfrac{\gamma_j}{2}}, &
	d_-&=\dim(\frakk_\CC)_{\pm\tfrac{\gamma_j+\gamma_k}{2}}.	
	\end{aligned}
\end{align}
These dimensions are independent of the choice of sign $\pm$ and $1\leq j,k\leq r$, $j\neq k$, and the possible combinations of root systems and structure constants are given in Table~\ref{tab:RootSystems}. Depending on $\Phi(\frakk_\CC,\frakt_\CC)$, we say $X$ is of type $A$, $B$, $BC$, $C$ or $D$, and if necessary also refer to the rank of the root system.
\begin{table}
	\begin{tabular}{llll}
		\hline
		$\Phi(\frakg_\CC,\frakt_\CC)$ & $\Phi(\frakk_\CC,\frakt_\CC)$ & structure constants & naming\\
		\hline\hline
		$C_r$ & $A_{r-1}$ & $b=0,\ e=0,\ d_-=0$ & Euclidean, in particular unital\\
		$C_r$ & $C_r$ & $b=0,\ e\neq 0$ & unital non-reduced\\
		$C_r$ & $D_r$ & $b=0,\ e=0,\ d_-\neq0$ & unital reduced\\
		\hline
		$BC_r$ & $B_r$ & $b\neq0,\ e=0$ & non-unital reduced\\
		$BC_r$ & $BC_r$ & $b\neq0,\ e\neq 0$ & non-unital non-reduced\\
		\hline\vspace{-3mm}
  \end{tabular}
  \caption{Possible root systems and structure constants}
  \label{tab:RootSystems}
\end{table}
For convenience, we set
\begin{align}\label{eq:dDef}
	d=\frac{d_++d_-}{2}.
\end{align}
If $d_+,d_->0$ then $d=d_+=d_-$ unless the root system is reducible which happens in case $D_2$. Here $\frakg=\frako(p,q)$ with $d_+=p-1$, $d_-=q-1$. Further, it can happen that $d_-=0$, then the root system is of type $A$ and $d_+=2d$. In this case $\frakg$ is a Hermitian Lie algebra. This yields:

\begin{proposition}
Precisely one of the following holds:
\begin{enumerate}
	\item (Type $A$) $\frakg$ is Hermitian of tube type (in this case $d_-=b=e=0$),
	\item (Type $D_2$) $\frakg=\frako(p,q)$ with $p\neq q$, $p,q\geq3$ (in this case $d_+=p-1$, $d_-=q-1$ and $b=e=0$),
	\item\label{d+=d-} $d_+=d_-$.
\end{enumerate}
\end{proposition}

In those statements that only hold in case \eqref{d+=d-} we will state $d_+=d_-$ as an assumption. In fact, this is the algebraic reason why modifications of the statements are necessary in type $A$ and $D_2$.

For each strongly orthogonal root $\gamma_k$, we fix an $\sl_2$-triple $(E_k,H_k,F_k)$, $E_k\in\frakn$, $F_k\in\overline\frakn$, with $\theta(E_k) = -F_k$ and $(E_k-F_k)\in\frakt$, $(E_k+F_k\pm iH_k)\in(\frakg_\CC)_{\pm\gamma_k}$. Let $\kappa$ denote the Killing form of $\frakg$. Then, 
\begin{align}\label{eq:pDefinition}
	p=\tfrac{1}{8}\kappa(H_k,H_k) = (e+1) + (r-1)\,d + \tfrac{b}{2},
\end{align}
is independent of $1\leq k\leq r$, and is called the \textit{genus} of the symmetric $R$-space $X$.

For $\nu\in\fraka_\CC^*$ we consider the degenerate principal series $\pi_\nu=\Ind_P^G({\bf 1}\otimes e^\nu\otimes{\bf 1})$, realized on the space
\[
	I(\nu)=\Set{f\in C^\infty(G)}{f(gman) = a^{-\nu-\rho}f(g)\,\forall\,g\in G,\,man\in MAN},
\]
where $\rho=\frac{1}{2}\Tr\ad_\frakn\in\fraka^*$ denotes half the sum of all positive roots. Identifying $\fraka_\CC^*\simeq\CC$ by $\nu\mapsto\frac{p}{n}\,\nu(Z_0)$ we have $\rho=\frac{p}{2}$. The restriction of $\pi_\nu$ to $K$ decomposes into a multiplicity-free direct sum of $K$-modules of the form $V_\b m$ with highest weight $\sum_{k=1}^r m_k\gamma_k$, where $\b m=(m_1,\ldots,m_r)$ is a non-increasing $r$-tuple of integers. More precisely,
\[
	I(\nu)_\Kfinite = \bigoplus_{\b m\in\Lambda} V_\b m,
\]
where
\begin{align*}
	\Lambda=\begin{cases}
		\Set{\b m\in\ZZ^r}{m_1\geq\cdots\geq m_r} & \text{for type $A$,}\\
		\Set{\b m\in\ZZ^r}{m_1\geq\cdots\geq m_{r-1}\geq|m_r|} &
		\text{for type $D$,}\\
		\Set{\b m\in\ZZ^r}{m_1\geq\cdots\geq m_r\geq 0} & \text{else.}\\
	\end{cases}
\end{align*}

The questions of reducibility, composition series and unitarizability of the composition factors of $I(\nu)$ were completely answered by Johnson~\cite{Joh90,Joh92}, {\O}rsted--Zhang~\cite{OZ95} and Sahi~\cite{Sah92} for case $A$, by Sahi~\cite{Sah95} and Zhang~\cite{Zha95} for the cases $C$ and $D$, and by the authors~\cite{MS12} for the cases $B$ and $BC$. The precise results differ from case to case, but they all have a common feature. Let
$$ \nu_k=\frac{p}{2}-k\frac{d}{2}, \qquad \mbox{for }k=0,\ldots,r-1. $$
We say an irreducible constituent of $I(\nu)$ has \textit{rank $k$}, $0\leq k\leq r$, if its $K$-types, identified with their heighest weights in $\Lambda\subseteq\ZZ^r$, are contained in a single $k$-dimensional subspace of $\RR^r$.

\begin{theorem}\label{thm:StructureDegPrincipalSeries}
\begin{enumerate}
\item\label{StructureDegPrincipalSeries1} For $X$ not of type $A$ or $D_2$, and for any $k=0,\ldots,r-1$ the representation $I(\nu)$ has an irreducible quotient of rank $k$ if and only if $\nu=\nu_k$. This quotient $J(\nu_k)$ has the $K$-type decomposition
$$ J(\nu_k)_{K-\textup{finite}} = \bigoplus_{\b m\in\Lambda,\,m_{k+1}=0}\hspace{-.5cm}V_\b m. $$
Moreover, $J(\nu_k)$ is unitarizable if and only if either $k=0$ or $k>0$ and the pair $(\frakg,\frakl)$ is not of the form $(\sl(p+q,\FF),\fraks(\gl(p,\FF)\times\gl(q,\FF)))$, $\FF=\RR,\CC,\HH$ with $p\neq q$.
\item For $X$ of type $A$ or $D_2$ a similar statement as in (\ref{StructureDegPrincipalSeries1}) holds if one replaces $I(\nu)$ by the induced representation $\Ind_P^G(\sigma\otimes e^\nu\otimes\1)$ for a certain unitary character $\sigma$ of $M$. In type $A$ the corresponding quotient is always unitarizable, whereas in type $D_2$ we have $\frakg=\frako(p,q)$ and the quotient is unitarizable if and only if $p+q$ is even.
\end{enumerate}
\end{theorem}

Because of the special role of type $A$ and $D_2$ we exclude these from some of our statements by assuming $d_+=d_-$. We remark that modifications of the main statements also hold in type $A$ or $D_2$. In fact, for $X$ of type $A$ the Lie algebra $\frakg$ is Hermitian of tube type and the quotients $J(\nu_k)$ are contained in the analytic continuation of the holomorphic discrete series (the discrete part of the so-called Wallach or Berezin--Wallach set). $L^2$-models for these representations were constructed by Rossi--Vergne~\cite{RV76} (see also Sahi~\cite{Sah92}). For $X$ of type $D_2$ the Lie algebra $\frakg$ is the indefinite orthogonal algebra $\frako(p,q)$ and the irreducible quotient $J(\nu_1)$ is the minimal representation of $O(p,q)$. An $L^2$-model was constructed by Kobayashi--{\O}rsted~\cite{KO03c} (see also Kobayashi--Mano~\cite{KM08} for a detailed analysis of this model). A construction of these models by the same methods as used in this paper can be found in \cite[Section 2.1]{Moe13} for the Hermitian cases and in \cite{HKM14} for the case $k=1$ (including $\frakg=\frako(p,q)$).

\subsection{Jordan pairs}\label{sec:JordanTheoryOfSymmRSpaces}

Recall that the pair $V=(V^+,V^-)=(\frakn,\overline\frakn)$ turns into a real Jordan pair when equipped with the trilinear products
\[
	\JTP{\,}{\,}{\,}\colon V^\pm\times V^\mp\times V^\pm\to V^\pm, \quad (x,y,z)\mapsto\JTP{x}{y}{z}=-[[x,y],z].
\]
We refer to \cite{Loo75} for a detailed introduction to Jordan pairs and to \cite{Ber00} for a classification. Simplicity of $G$ implies that $V$ is in fact a simple real Jordan pair. As usual, we define additional operators $Q_x$, $Q_{x,z}$, and $D_{x,y}$ by
\[
	Q_xy=\tfrac{1}{2}\JTP{x}{y}{x},\qquad Q_{x,z}y=D_{x,y}z=\JTP{x}{y}{z}.
\]
The following three identities are needed later and are contained in the Appendix of \cite{Loo77} as JP$7$, JP$8$ and JP$16$ (note that we use different normalizations):
\begin{align}
 & D_{\JTP{x}{y}{z},y} = D_{z,Q_yx} + D_{x,Q_yz},\label{JP7}\\
 & D_{x,\JTP{y}{x}{z}} = D_{Q_xy,z} + D_{Q_xz,y},\label{JP8}\\
 & \JTP{\JTP{x}{y}{u}}{v}{z} - \JTP{u}{\JTP{y}{x}{v}}{z} = \JTP{x}{\JTP{v}{u}{y}}{z} - \JTP{\JTP{u}{v}{x}}{y}{z}.\label{JP16}
\end{align}

The trace of the operator $D_{x,y}$ on $V^+$ defines a non-degenerate pairing,
\begin{align}\label{eq:traceform}
	\tau\colon V^+\times V^-\to\RR, \quad (x,y)\mapsto \tfrac{1}{2p}\Tr_{V^+}(D_{x,y}),
\end{align}
the \emph{trace form} on $(V^+,V^-)$, where $p$ is the structure constant defined in \eqref{eq:pDefinition}. The trace form is related to the Killing form $\kappa$ of $\frakg$ by
\[
	\tau(x,y) = -\tfrac{1}{4p}\,\kappa(x,y),
\]
see \cite[I.\S\,7]{Sat80}, and the normalization is chosen such that
\begin{align}\label{eq:traceformonprimitive}
	\tau(E_k,-F_k) = \tfrac{1}{8p}\kappa(H_k,H_k)=1,
\end{align}
where $E_k$ and $F_k$ are the root vectors corresponding to strongly orthogonal roots. 

In the Jordan setting, we abbreviate the adjoint action of $T\in\frakl$ on $x\in V^\pm$ by $T x=[T,x]$. Likewise, $h x=\Ad(h)x$ denotes the adjoint action of $h\in L$ on $x\in V^\pm$. We note that 
\[
	h \JTP{x}{y}{z} = \JTP{hx}{hy}{hz}
\]
for all $x,z\in V^\pm$ and $y\in V^\mp$, i.e., the pair $(h|_{V^+},h|_{V^-})\in\GL(V^+)\times\GL(V^-)$ is an automorphism of the Jordan pair $V$. There are automorphisms of particular importance: For $(x,y)\in V^\pm\times V^\mp$ the \emph{Bergman operator} $B_{x,y}\in\End(V^\pm)$ is defined by
\[
	\B{x}{y}=\id_{V^\pm}-D_{x,y}+Q_xQ_y.
\]
It is well-know that $\B{x}{y}$ is invertible if and only if $\B{y}{x}$ is invertible, and in this case, $(\B{x}{y},\B{y}{x}^{-1})$ is an automorphism of $V$. Moreover, there is an element $h\in L$ such that $h|_{V^+}=\B{x}{y}$ and $h|_{V^-}=\B{y}{x}^{-1}$. By abuse of notation, we simply write $(\B{x}{y},\B{y}{x}^{-1})\in L$. The benefit of the Bergman operator becomes apparent in the following identity, which describes the decomposition of particular group elements $g\in G$ according to $\overline NLN\subseteq G$, see \cite[Theorem 8.11]{Loo77}:

\begin{lemma}
	Let $x\in V^+$, $y\in V^-$ be such that $\B{x}{y}$ is invertible. Then,
	\begin{align}\label{eq:BergmanIdentity}
		\exp(x)\exp(y) = \exp(y^x)(\B{x}{y},\B{y}{x}^{-1})\exp(x^y)
	\end{align}
	where $x^y=\B{x}{y}^{-1}(x-Q_xy)$, $y^x=\B{y}{x}^{-1}(y-Q_yx)$.
\end{lemma}

The determinant $\Det_{V^+}\B{x}{y}$ of the Bergman operator turns out to be the $2p$'th power of an irreducible polynomial $\Delta\colon V^+\times V^-\to\RR$,
\begin{align}\label{eq:JPDet}
	\Det_{V^+}\B{x}{y} = \Delta(x,y)^{2p},
\end{align}
where $p$ is the structure constant of $X$ defined in \eqref{eq:pDefinition}. The polynomial $\Delta$ is called the \emph{Jordan pair determinant} (or the \emph{generic norm}, cf.\ \cite{Loo75}).

The Cartan involution $\theta$ on $\frakg$ induces an involution $V^\pm\to V^\mp$ on the Jordan pair $(V^+,V^-)$, which for simplicity is denoted by
\[
	\overline x=\theta(x)
\]
for $x\in V^\pm$. For later calculations, we note the following useful Lemma.

\begin{lemma}\label{lem:expDecomp}
	For $y\in V^-$, the Bergman operator $\B{\overline y}{-y}$ is positive definite with 
	respect to the trace form $\tau$. Moreover, the element $\exp(y)$ admits a decomposition 
	$\exp(y)=k_y\ell_yn_y$ according to $G=KLN$ with 
	$$ \ell_y =(\B{-\overline{y}}{y}^\half, \B{y}{-\overline y}^{-\half}). $$
\end{lemma}

\begin{proof}
Write $\exp(y)=k_y\ell_yn_y$ with $k_y\in K$, $\ell_y\in L$ and $n_y\in N$. We may choose $\ell_y$ such that $\theta(\ell_y)=\ell_y^{-1}$ by putting the $(M\cap K)$-part of $\ell_y$ into $k_y\in K$. Then, on the one hand
\begin{align*}
 \theta(\exp(y)^{-1})\exp(y) &= \theta(n_y)^{-1}\ell_yk_y^{-1}k_y\ell_yn_y
 	= \theta(n_y)^{-1}\ell_y^2n_y.
\end{align*}
On the other hand, $\theta(\exp(y)^{-1}) = \exp(-\overline y)$ with $\overline y\in V^+$, so \eqref{eq:BergmanIdentity} yields
\[
	\theta(\exp(y)^{-1})\exp(y)
		= \exp(y^{-\overline y})(\B{-\overline{y}}{y}, \B{y}{-\overline y}^{-1})\exp((-\overline y)^y).
\]
Since the decomposition of an element in $G$ according to $\overline{N}LN\subseteq G$ is unique, this implies $\ell_y^2 = (\B{-\overline{y}}{y}, \B{y}{-\overline y}^{-1})$.
\end{proof}

\subsection{The non-compact picture}

The character $\chi_\nu=\1\otimes e^\nu$ of $L=MA$ which is used to define the principal series $I(\nu)$ in Section~\ref{sec:degenerateprincipalseries} is recovered in the Jordan setting as
\begin{align}\label{eq:characterformula}
	\chi_\nu(h) = |\Det_{V^+}(h)|^{\frac{\nu}{p}} = |\Det_{V^-}(h)|^{-\frac{\nu}{p}} \qquad \text{for $h\in L$.}
\end{align}
This follows from the observation that for $t\in\RR$ the central element $\exp(tZ_0)\in L$ acts on $V^\pm$ by $e^{\pm t}\,\id_{V^\pm}$. Up to powers, there exists only one positive character of $L=MA$, since $\dim A=1$ and $M$ is reductive with compact center.

Since $\overline NP\subseteq G$ is dense, the restriction of functions in $I(\nu)$ to $\overline N$ is an injective map. Moreover, since the exponential map $\exp\colon V^-\to\overline N$ is a diffeomorphism, it follows that
\[
	I(\nu)\hookrightarrow C^\infty(V^-), \quad f\mapsto f_{V^-}(y)=f(\exp(y))
\]
is an injective map. In the following, we identify $I(\nu)$ with the corresponding subspace of $C^\infty(V^-)$, and by abuse of notation we write $f(y)=f_{V^-}(y)$ for $f\in I(\nu)$ and $y\in V^-$. This is called the \emph{non-compact picture} of $\pi_\nu$. We note that
$$ \calS(V^-)\subseteq I(\nu)\subseteq C^\infty(V^-)\cap\calS'(V^-), $$
where $\calS(V^-)$ denotes the Schwartz space of rapidly decreasing functions and $\calS'(V^-)$ its dual, the space of tempered distributions.

\begin{proposition}\label{prop:noncompactgroupaction}
	The action of $\exp(a)\in\overline N$, $h\in L$ and $\exp(b)\in N$ on
	$f\in I(\nu)$ is given by ($y\in V^-$)
	\begin{align*}
 		\pi_\nu(\exp(a))f(y) &= f(y-a),\\
 		\pi_\nu(h)f(y) &= \chi_{\nu+\tfrac{p}{2}}(h)f(h^{-1}y),\\
 		\pi_\nu(\exp(b))f(y) &= |\Delta(-b,y)|^{-2\nu-p}f(y^{-b}).
	\end{align*}
\end{proposition}

\begin{proof}
This is an easy consequence of the identities \eqref{eq:BergmanIdentity} and \eqref{eq:JPDet}.
\end{proof}

The infinitesimal action of $\frakg$ on $I(\nu)$ is determined in a straightforward way from this by using \eqref{eq:characterformula} and the derivatives
$$ \left.\td_b(x\mapsto y^x)\right|_{x=0} = Q_yb, \qquad \left.\td_b(x\mapsto\Delta(x,y))\right|_{x=0} = -\tau(b,y). $$
Here $\td_v$ denotes the differential of a map in the direction of $v$. We thus obtain the following result.

\begin{proposition}\label{prop:liealgaction}
The infinitesimal action $\td\pi_\nu$ of $\pi_\nu$ is given by
\begin{align*}
 \td\pi_\nu(a)f(y) &= -\td_af(y), & &a\in V^-,\\
 \td\pi_\nu(T)f(y) &= -\td_{Ty}f(y)+(\tfrac{\nu}{p}+\tfrac{1}{2})\,\Tr_{V^+}(T)\,f(y), 
 & &T\in\frakl,\\
 \td\pi_\nu(b)f(y) &= -\td_{Q_yb}f(y)-(2\nu+p)\,\tau(b,y)\,f(y), & &b\in V^+.
\end{align*}
\end{proposition}

The representations $(I(\nu),\pi_\nu)$ are spherical and the trivial representation of $K$ occurs with multiplicity one. We determine an explicit expression for the $K$-invariant vector in $I(\nu)$:

\begin{proposition}\label{prop:SphericalVectorNonCptPicture}
The unique $K$-invariant vector $\phi_{\nu}$ in $I(\nu)$ with $\phi_{\nu}(0)=1$ is given by
$$ \phi_{\nu}(y) = \Delta(-\overline{y},y)^{-\nu-\frac{p}{2}}, \qquad y\in V^-. $$
\end{proposition}

\begin{proof}
Let $f\colon G\to\CC$ denote the $K$-invariant vector in the induced picture of $I(\nu)$, normalized by $f(\1)=\1$. Then
\begin{align*}
 \phi_{\nu}(y) &= f(\exp(y)).
\end{align*}
We write $\exp(y)=k_y\ell_yn_y$ with $k_y\in K$, $\ell_y\in L$ and $n_y\in N$ as in Lemma~\ref{lem:expDecomp}. Due to $K$-invariance of $f$ this yields
\begin{align*}
	\phi_{\nu}(y) = \chi_{\nu+\tfrac{p}{2}}(\ell_y)^{-1}.
\end{align*}
Finally, applying \eqref{eq:characterformula} to $\ell_y=(\B{-\overline{y}}{y}^\half, \B{y}{-\overline y}^{-\half})$ and using \eqref{eq:JPDet} proves the assertion.
\end{proof}

\section{Structure and geometry of the orbits}\label{sec:OrbitStructure}

It is well-known that the action of $L$ decomposes $V^+$ into finitely many orbits, which can be described explicitly. For this recall the $\sl_2$-triples $(E_k,H_k,F_k)$, $1\leq k\leq r$, with $E_k\in\frakn$, $H_k\in\frakl$, $F_k\in\overline\frakn$ defined in Section~\ref{subsec:degenerateprincipalseries}. For $0\leq k,\ell\leq r$ with $k+\ell\leq r$ put
\[
	o_{k,\ell} = \sum_{i=1}^k E_i - \sum_{j=k+1}^{k+\ell} E_j\in\frakn,
\]
and let $\calO_{k,\ell}=L\cdot o_{k,\ell}$ denote the $L$-orbit of $o_{k,\ell}$. The following result is due to Kaneyuki~\cite{Kan98}, see also \cite[Part II]{FaEtAl00}:

\begin{theorem}
Every $L$-orbit in $V^+$ is of the form $\calO_{k,\ell}$ for some $0\leq k,\ell\leq r$ with $k+\ell\leq r$. The orbit $\calO_{k,\ell}$ is open if and only if $k+\ell=r$. Moreover, the non-open orbits in $V^+$ are given by
\begin{align*}
	&\Set{\calO_{k,\ell}}{k,\ell\geq 0,\ k+\ell\leq r-1} & &\text{for type $A$,}\\
	&\Set{\calO_{k,0}}{0\leq k\leq r-1} & &\text{for type $B$, $BC$, $C$ or $D$.}
\end{align*}
\end{theorem}

In this section, we find local charts for the orbits $\calO_{k,\ell}$, which are used in Section~\ref{sec:tangency} to show that for certain parameters the Bessel operators act tangentially along these orbits. We further determine those orbits which carry an $L$-equivariant measure, and provide integral formulas for these measures. This will be carried out in the framework of Jordan pairs and idempotents.

\subsection{Idempotents, Peirce decomposition, and rank}

For convenience to the reader, we recall some basic notions from Jordan theory used in the sequel. Most statements are valid for arbitrary simple real Jordan pairs $V=(V^+,V^-)$.

An \emph{idempotent} is a pair $\b e=(e,e')\in V^+\times V^-$ satisfying
\[
	Q_ee'=e \qquad \text{and} \qquad Q_{e'}e=e'.
\]
Then, $\b e$ induces a \emph{Peirce decomposition} of $V^\pm$ into eigenspaces of $D_{e,e'}$ and $D_{e',e}$, respectively,
\begin{align*}
	V^\pm=V_2^\pm(\b e)\oplus V_1^\pm(\b e)\oplus V_0^\pm(\b e)
	\qquad\text{with}\quad
	\left\{\begin{aligned}
	&V^+_k(\b e) = \set{x\in V^+}{D_{e,e'}(x) = kx},\\
	&V^-_k(\b e) = \set{y\in V^-}{D_{e',e}(y) = ky}.
	\end{aligned}\right.
\end{align*}
For a given $\b e$, we often simply write $V^\pm_k=V^\pm_k(\b e)$. The \emph{Peirce rules}
\begin{align}\label{eq:peircerules}
	\JTP{V^\pm_k}{V^\mp_\ell}{V^\pm_m}\subseteq V^\pm_{k-\ell+m}\qquad\text{and}\qquad 
	\JTP{V^\pm}{V^\mp_2}{V^\pm_0}=0
\end{align}
describe the algebraic relations between Peirce spaces, see \cite[\S\,5]{Loo75} for details. In particular, each $V_k=(V^+_k,V^-_k)$ is a Jordan subpair of $V=(V^+,V^-)$. Two idempotents $\b e,\b c$ are \emph{orthogonal}, if $\b e\in V^+_0(\b c)\times V^-_0(\b c)$. In this case, the sum $\b e+\b c=(e+c,e'+c')$ is also idempotent. An idempotent $\b e$ is called \emph{primitive}, if it cannot be decomposed into the sum of non-zero orthogonal idempotents. A maximal system of orthogonal primitive idempotents is called a \emph{frame}. Any frame has the same number of elements, called the \emph{rank} of the Jordan pair $(V^+,V^-)$. The rank of an idempotent $\b e$ is by definition the rank of $(V_2^+(\b e),V_2^+(\b e))$. Any element $e\in V^+$ admits a completion to an idempotent $\b e=(e,e')$ with $e'\in V^-$. The rank of $\b e$ is independent of the choice of $e'$, so $\rank(e)=\rank(\b e)$ is well-defined. 

\begin{proposition}\label{prop:ClassificationV1}
Let $V=(V^+,V^-)$ be a simple real Jordan pair, ${\b e}=(e,e')\in V^+\times V^-$ an idempotent of rank $1\leq k\leq r-1$ and $V^\pm=V_2^\pm\oplus V_1^\pm\oplus V_0^\pm$ the corresponding Peirce decomposition.
\begin{enumerate}[label={(\alph*)}]
\item\label{ClassificationV1-1} The subpair $V_1=(V_1^+,V_1^-)$ is either a simple real Jordan pair or the direct sum of two simple real Jordan pairs. In the latter case, the two simple summands are isomorphic if and only if $V\not\simeq(M(p\times q,\FF),M(q\times p,\FF))$, $\FF=\RR,\CC,\HH$ with $p\neq q$.
\item\label{ClassificationV1-2} We have $\Tr_{V^+}(T)=0$ for every $T\in\frakl$ with $T|_{V_2^+\oplus V_0^+}=0$ and $TV_1^+\subseteq V_1^+$ if and only if $V\not\simeq(M(p\times q,\FF),M(q\times p,\FF))$, $\FF=\RR,\CC,\HH$ with $p\neq q$.
\end{enumerate}
\end{proposition}

\begin{proof}
We first prove the statements for simple complex Jordan pairs. These are classified, and we check \ref{ClassificationV1-1} and \ref{ClassificationV1-2} for all pairs separately. Regarding statement \ref{ClassificationV1-2}, we note that $[T,D_{u,v}]=D_{Tu,v}+D_{u,Tv}=0$ if $(u,v)\in V_2^+\times V_2^-$ or $(u,v)\in V_0^+\times V_0^-$. Hence $T$ commutes with the Lie subalgebra $\frakl_{0,2}$ of $\frakl$ generated by $D_{u,v}$ with $(u,v)\in V_2^+\times V_2^-$ and $(u,v)\in V_0^+\times V_0^-$. If $\frakl_{0,2}$ acts irreducibly on each simple factor of $V_1^+$, then $T$ has to be scalar on each factor. For $V_1$ simple this scalar $\lambda$ has to be zero, because the Peirce rule $0\neq\JTP{V_2^+}{V_1^-}{V_0^+}\subseteq V_1^+$ implies $-\lambda=\lambda$. (Note that if $T$ acts on $V_1^+$ by $\lambda$ then it acts on $V_1^-$ by $-\lambda$.) For $V_1$ the sum of two simple pairs $U$ and $W$ the scalars have to add up to zero since $0\neq\JTP{U^+}{W^-}{V_2^+}\subseteq V_2^+$. Hence the trace of $T$ is zero if and only if the dimensions of $U$ and $W$ agree. We will see that this is the case if and only if $U\simeq W$.\\
By this previous observation, for statement \ref{ClassificationV1-2} it is sufficient to show that $\frakl_{0,2}$ acts irreducibly on each simple factor of $V_1^+$. This is true for all but one simple complex Jordan pair (see case \eqref{exceptionalcase} for the exception). We now proceed to show \ref{ClassificationV1-1} and \ref{ClassificationV1-2} for all simple complex Jordan pairs case by case.
\begin{enumerate}
\item $V^+=M(p\times q,\CC)$. Let $\b e$ be an idempotent of rank $0\leq k\leq r=\min(p,q)$, then $V_1^+\simeq M(k\times(q-k),\CC)\times M((p-k)\times k,\CC)$ which is the direct sum of two simple ideals. Note that these are isomorphic if and only if they have the same dimension. Further, the structure algebra of $V_2^+\simeq M(k\times k,\CC)$ is isomorphic to $\sl(k,\CC)\times\sl(k,\CC)\times\CC$ and the structure algebra of $V_0^+\simeq M((p-k)\times(q-k),\CC)$ is isomorphic to $\sl(p-k,\CC)\times\sl(q-k,\CC)\times\CC$. Clearly $\sl(k,\CC)\times\sl(q-k,\CC)$ acts irreducibly on $M(k\times(q-k),\CC)$ and $\sl(p-k,\CC)\times\sl(k,\CC)$ acts irreducibly on $M((p-k)\times k,\CC)$.
\item $V^+=\Sym(r,\CC)$. Let $\b e$ be an idempotent of rank $0\leq k\leq r$, then $V_1^+\simeq M(k\times(r-k),\CC)$ which is simple. Further, the structure algebra of $V_2^+\simeq\Sym(k,\CC)$ is isomorphic to $\gl(k,\CC)$ and the structure algebra of $V_0^+\simeq\Sym(r-k,\CC)$ is isomorphic to $\gl(r-k,\CC)$. Clearly $\gl(k,\CC)\times\gl(r-k,\CC)$ acts irreducibly on $M(k\times(r-k),\CC)$.
\item $V^+=\Skew(m,\CC)$, $m=2r$ or $m=2r+1$. Let $\b e$ be an idempotent of rank $0\leq k\leq r$, then $V_1^+\simeq M(2k\times(m-2k),\CC)$ which is simple. Further, the structure algebra of $V_2^+\simeq\Skew(2k,\CC)$ is isomorphic to $\gl(2k,\CC)$ and the structure algebra of $V_0^+\simeq\Skew(m-2k,\CC)$ is isomorphic to $\gl(m-2k,\CC)$. Clearly $\gl(2k,\CC)\times\gl(m-2k,\CC)$ acts irreducibly on $M(2k\times(m-2k),\CC)$,
\item\label{exceptionalcase} $V^+=\CC^n$. Let $\b e$ be an idempotent of rank $0\leq k\leq 2$. If $k=0$ or $k=2$ then $V_1^+=\{0\}$. If $k=1$ then $V_1^+\simeq\CC^{n-2}$ which is simple. Further, the subalgebra of $T\in\frakl=\so(n,\CC)\oplus\CC$ with $T|_{V_2^+}=0$ and $T|_{V_0^+}=0$ is isomorphic to $\so(n-2,\CC)$ and hence $\Tr_{V_1^+}(T)=0$.
\item $V^+=\Herm(3,\OO_\CC)$. Let $\b e$ be an idempotent of rank $0\leq k\leq3$. If $k=0$ or $k=3$ then $V_1^+=\{0\}$. In the cases $k=1$ and $k=2$ the subpair $V_1$ and the Lie algebra $\str(V_2)\times\str(V_0)$ is the same and therefore it suffices to treat $k=1$. In this case $V_1^+\simeq M(1\times 2,\OO_\CC)$ which is simple. Further, the structure group of $V_0^+\simeq\Herm(2,\OO_\CC)\simeq\CC^{10}$ is isomorphic to $\so(10,\CC)\oplus\CC$ which acts irreducibly on $M(1\times 2,\OO_\CC)$ (which is isomorphic to $\CC^{16}$ as a vector space) by the spin representation.
\item $V^+=M(1\times2,\OO_\CC)$. Let $\b e$ be an idempotent of rank $0\leq k\leq 2$. If $k=0$ or $k=2$ then $V_1^+\simeq\OO_\CC\simeq\CC^8$ on which the structure group $\so(8,\CC)\oplus\CC$ of $V_2^+\oplus V_0^+\simeq\OO_\CC\simeq\CC^8$ acts irreducibly by the standard representation. If $k=1$ then $V_1^+\simeq\Skew(5,\CC)$ which is simple (see e.g. \cite{Roo08}). Further, the structure algebra of $V_0^+\simeq M(1\times5,\CC)$ is isomorphic to $\gl(5,\CC)$ which acts irreducibly on $\Skew(5,\CC)=\bigwedge^2\CC^5$.
\end{enumerate}
Next assume $V$ is not complex, then its complexification $V_\CC$ is a simple complex Jordan pair and the idempotent $\b e$ is idempotent in $V_\CC$. Hence $(V_1)_\CC$ is either a simple complex Jordan pair, whence $V_1$ is also simple, or the direct sum of two simple complex Jordan pairs, whence $V_1$ is either simple or the direct sum of two simple real Jordan pairs. These are non-isomorphic if and only if $(V_1)_\CC$ is the sum of two non-isomorphic simple pairs which happens only in the case $V_\CC\simeq M(p\times q,\CC)$, $p\neq q$. Further, any $T\in\frakl$ extends $\CC$-linearly to $T\in\frakl_\CC$ and the second statement follows.\\
It only remains to show that $V^+=M(p\times q,\FF)$, $\FF=\RR,\HH$, are the only real forms of $V_\CC^+=M(p\times q,\CC)$ for which $V_1^+$ is the sum of two simples. But by classification (see e.g. \cite{Ber00}) there is only one more real form, namely $V^+=\Herm(m,\CC)$, $m=p=q$, and for this real form $V_1^+\simeq M(k\times(m-k),\CC)$ which is simple. This finishes the proof.
\end{proof}

The rank of $(V^+,V^-)$ coincides with the rank of the symmetric space $X=K/(M\cap K)$. Indeed, recall that the root vectors $E_k$ and $F_k$ of strongly orthogonal roots define a particular frame $(\b e_1,\ldots,\b e_r)$, given by $\b e_k=(E_k,-F_k)$, $k=1,\ldots, r$. This frame has the additional property of being compatible with the Cartan involution, i.e., $\overline E_k=-F_k$. More generally, an element $e\in V^+$ is called \emph{tripotent}, if the pair $(e,\overline e)$ is an idempotent. Accordingly, we call $e$ \emph{primitive}, if $(e,\overline e)$ is primitive, and two tripotents $e,c$ are \emph{orthogonal}, if $e\in V_0^+(c,\overline c)$. A maximal system of othogonal primitive tripotents is called a \emph{frame of tripotents}.

Orthogonal idempotents yield compatible Peirce decompositions. Therefore, a frame $(\b e_1,\ldots,\b e_r)$ induces a \emph{joint Peirce decomposition}
\[
	V^{\pm} = \bigoplus_{0\leq i\leq j\leq r} V_{ij}^\pm,
\]
where
\begin{align*}
	V_{ij}^+ &= \set{x\in V^\pm}{\JTP{e_k}{e_k'}{x} = (\delta_{k i}+\delta_{kj})x},\\
	V_{ij}^- &= \set{y\in V^\pm}{\JTP{e_k'}{e_k}{y} = (\delta_{k i}+\delta_{kj})y}.
\end{align*}
This corresponds to the root space decomposition of $\frakn$ and $\overline\frakn$ with respect to $\frakt_\CC$ of Section~\ref{sec:degenerateprincipalseries}. For a frame $(e_1,\ldots,e_r)$ of tripotents, the Cartan involution relates the positive and negative joint Peirce spaces:
\[
	\overline{V_{ij}^+}=\theta V_{ij}^+=V_{ij}^-.
\]
For $e=e_1+\cdots+e_r$, the maps $x\mapsto Q_e\overline x$ and $y\mapsto Q_{\overline e}\overline y$ define involutions on $V_{ij}^+$ and $V_{ij}^-$ for $1\leq i,j\leq r$ with $\pm 1$-eigenspace decomposition
\[
	V_{ij}^\pm = A_{ij}^\pm\oplus B_{ij}^\pm.
\]
The structure constants \eqref{eq:structureconstants} of the symmetric $R$-space $X$ are related to these refined Peirce spaces via
\begin{align}\label{eq:Jordanstructureconstants}
	\dim B_{ii}^\pm=e,\qquad \dim A_{ij}^\pm=d_+,\qquad \dim B_{ij}^\pm=d_-,\qquad \dim V_{0i}^\pm=b.
\end{align}
Moreover, $A_{ii}^+=\RR e_i$ and $A_{ii}^-=\RR\overline e_i$. The constant $p$ defined in \eqref{eq:pDefinition} is also called the \textit{genus} of the Jordan pair $V$.

We prove some summation formulas which are needed later on.

\begin{lemma}\label{lem:basessums}
Set $I=\{1,\ldots, n\}$, and let $\{c_\alpha\}_{\alpha\in I}$ be a basis of $V^+$, and 
$\{\widehat{c}_\alpha\}_{\alpha\in I}$ be the basis of $V^-$ dual to 
$\{c_\alpha\}_{\alpha\in I}$ with respect to the trace form $\tau$.
\begin{enumerate}
\item\label{basessums1}
	We have
	\[
		\sum_{\alpha\in I} D_{c_\alpha,\widehat{c}_\alpha} = 2p\cdot\id_{V^+}.
	\]
\item\label{basessums2}
	If $\b e=(e,e')$ is an idempotent of rank $k$, and the basis $\{c_\alpha\}_{\alpha\in I}$ is
	compatible with the Peirce decomposition with respect to $\b e$, i.e.,
	$I=I_2\sqcup I_1\sqcup I_0$ with $c_\alpha\in V_\ell^+$ if and only if $\alpha\in I_\ell$, then
	\begin{align*}
		\sum_{\alpha\in I_2} D_{c_\alpha,\widehat{c}_\alpha} &= p_2\cdot D_{e,e'},\\
		\sum_{\alpha\in I_1}\left. D_{c_\alpha,\widehat{c}_\alpha}\right|_{V_2^+} &= 2(p-p_2)\cdot\id_{V_2^+} = (2(r-k)d+b)\cdot\id_{V_2^+},\\
		\sum_{\alpha\in I_1}\left. D_{c_\alpha,\widehat{c}_\alpha}\right|_{V_0^+} &= 2(p-p_0)\cdot\id_{V_0^+} = 2kd\cdot\id_{V_0^+}.
	\end{align*}
	Here, $p_2=(e+1)+(k-1)d$ resp. $p_0=(e+1)+(r-k-1)d+\frac{b}{2}$ is the structure constant defined as in \eqref{eq:pDefinition} but with respect to the simple Jordan pair $(V_2^+,V_2^-)$ resp. $(V_0^+,V_0^-)$.
\item If further $V\not\simeq(M(p\times q,\FF),M(q\times p,\FF))$, $\FF=\RR,\CC,\HH$ with $p\neq q$,, then
	\begin{align*}
		\sum_{\alpha\in I_1}\left. D_{c_\alpha,\widehat{c}_\alpha}\right|_{V_1^+} &= (2p-p_2-p_0)\cdot\id_{V_1^+} = \left(rd+\tfrac{b}{2}\right)\cdot\id_{V_1^+}.
	\end{align*}
\end{enumerate}
\end{lemma}

\begin{proof}
\begin{enumerate}
\item This identity is a consequence of the following calculation. For arbitrary $v\in V^+$, $w\in V^-$, associativity of the trace form yields
\begin{align*}
	\tau(\sum_{\alpha\in I} D_{c_\alpha,\widehat{c}_\alpha}v, w) &= \sum_{\alpha\in I}\tau(\JTP{c_\alpha}{\widehat{c}_\alpha}{v},w) = \sum_{\alpha\in I}\tau(\JTP{v}{w}{c_\alpha},\widehat{c}_\alpha)\\
	&=\Tr_{V^+}(D_{v,w}) = 2p\,\tau(v,w).
\end{align*}
\item For the first identity, we apply \eqref{basessums1} to the Jordan pair $(V_2^+,V_2^-)$ and obtain
\begin{align*}
	D_{e,e'}=\frac{1}{2p_2}\sum_{\alpha\in I_2}D_{\JTP{c_\alpha}{\widehat c_\alpha}{e},e'}
\end{align*}
Since $e=Q_ee'$, $\JTP{e'}{e}{\widehat c_\alpha}=2\widehat c_\alpha$, and
\begin{align*}
	\JTP{c_\alpha}{\widehat c_\alpha}{Q_ee'}
		&\stackrel{\eqref{JP8}}{=}
		\JTP{c_\alpha}{\JTP{e'}{e}{\widehat c_\alpha}}{e}-\JTP{c_\alpha}{e'}{Q_e\widehat c_\alpha}\\
		&\stackrel{\hphantom{\eqref{JP8}}}{=} 2\JTP{c_\alpha}{\widehat c_\alpha}{e} - \JTP{c_\alpha}{e'}{Q_e\widehat c_\alpha},
\end{align*}
hence $\JTP{c_\alpha}{\widehat c_\alpha}{e}=\JTP{c_\alpha}{e'}{Q_e\widehat c_\alpha}$, we obtain
\begin{align*}
	D_{e,e'}
		=\frac{1}{2p_2}\sum_{\alpha\in I_2} D_{\JTP{c_\alpha}{e'}{Q_e\widehat c_\alpha},e'}
		\stackrel{\eqref{JP7}}{=}
		\frac{1}{2p_2}\sum_{\alpha\in I_2} \left(D_{Q_e\widehat c_\alpha,Q_{e'}c_\alpha}
		+ D_{c_\alpha,Q_{e'}Q_e\widehat c_\alpha}\right).
\end{align*}
Since $\{Q_e\widehat c_\alpha\}_{\alpha\in I_2}$, $\{Q_{e'}c_\alpha\}_{\alpha\in I_2}$ is another pair of dual bases for the Jordan pair $(V_2^+,V_2^-)$, and the operator $\sum_{\alpha\in I_2}D_{c_\alpha,\widehat c_\alpha}$ is easily seen to be independent of the choice of such bases, we conclude that
\[
	D_{e,e'}=\frac{1}{p_2}\sum_{\alpha\in I_2} D_{c_\alpha,\widehat c_\alpha}.
\]
Concerning the second identity, we note that
\[
	\sum_{\alpha\in I_1} D_{c_\alpha,\widehat{c}_\alpha}
	=\sum_{\alpha\in I}D_{c_\alpha,\widehat{c}_\alpha}
	-\sum_{\alpha\in I_2}D_{c_\alpha,\widehat{c}_\alpha}
	-\sum_{\alpha\in I_0}D_{c_\alpha,\widehat{c}_\alpha}.
\]
Since for $\alpha\in I_0$, $D_{c_\alpha,\widehat{c}_\alpha}$ vanishes on $V_2^+$ due to the Peirce rules, the second identity follows from the previous ones.\\
Finally, for the last identity note that $\sum_{\alpha\in I_1}D_{c_\alpha,\widehat{c}_\alpha}$ commutes with the group $L_{[e]}=\set{g\in L}{gV_k^\pm\subseteq V_k^\pm\,\forall\,k=0,1,2}$. Since $L_{[e]}\to\Str(V_0),\,g\mapsto g|_{V_0}$ is surjective and $\Str(V_0)$ acts irreducibly on the simple Jordan pair $V_0$, the operator $\sum_{\alpha\in I_1}D_{c_\alpha,\widehat{c}_\alpha}$ acts on $V_0^+$ by a scalar. To determine this scalar we note that for $c_\alpha\in V_{0i}^+$ ($1\leq i\leq k$) and $x\in V_{jk}$ ($k+1\leq j,k\leq r$) we have $D_{c_\alpha,\widehat{c}_\alpha}x=0$. Hence, it suffices to compute the scalar for the Jordan algebra $\widetilde{V}^+=\bigoplus_{1\leq i,j\leq r}V_{ij}^+$. Using the previous identity this shows that the scalar is given by $2(\widetilde{p}-\widetilde{p}_0)$, where $\widetilde{p}$ and $\widetilde{p}_0$ are the structure constants for $\widetilde{V}^+$ and $\widetilde{V}_0^+=\bigoplus_{k+1\leq i,j\leq r}V_{ij}^+$. Then
$$ \widetilde{p} = e+1+(r-1)d, \qquad \widetilde{p}_0 = e+1+(r-k-1)d $$
and hence
$$ 2(\widetilde{p}-\widetilde{p}_0) = 2(k+1)d = 2(p-p_0). $$
\item Note that by Proposition~\ref{prop:ClassificationV1} the Jordan pair $V_1^+$ is the direct sum of isomorphic simple Jordan pairs. Hence, thanks to \eqref{basessums1}, $T=\sum_{\alpha\in I_1}D_{c_\alpha,\widehat{c}_\alpha}$ acts on $V_1^+$ by a fixed scalar $\lambda$. We further know by \eqref{basessums2} that $T$ acts on $V_2^+$ by $2(p-p_2)$ and on $V_0^+$ by $2(p-p_0)$. Since $\JTP{V_2^+}{V_1^-}{V_0^+}=V_1^+$ and $T$ acts on $V_1^\pm$ by $\pm\lambda$ this shows the claim.\qedhere
\end{enumerate}
\end{proof}

We state another summation formula for which we did not find a direct proof, but which can be verified case by case using the classification.

\begin{lemma}\label{lem:basessums2}
Let $\b e=(e,e')$ be an idempotent with Peirce decomposition $V=V_2\oplus V_1\oplus V_0$, and let $\{c_\alpha\}_{\alpha\in I_1}$ be a basis of $V_1^+$ and $\{\widehat{c}_\alpha\}_{\alpha\in I_1}$ the dual basis of $V_1^-$ with respect to the trace form $\tau$. Then
$$ \sum_{\alpha,\beta\in I_1} \left.Q_{c_\alpha,c_\beta}Q_{\widehat{c}_\alpha,\widehat{c}_\beta}\right|_{V_2^+} = 2p_0(p-p_2)\cdot\id_{V_2^+}. $$
\end{lemma}

\subsection{Orbit decomposition}
The notion of rank for elements in $V^+$ introduced in the last section yields the decomposition
\[
	V^+ = \bigsqcup_{k=0}^r \calV_k,
\]
where $\calV_k\subseteq V^+$ denotes the subset of elements of rank $k$. We construct local charts for $\calV_k$ to show that $\calV_k$ is an embedded submanifold. For any $e\in\calV_k$ let $\b e=(e,e')$ be a completion to an idempotent $\b e$, and denote by $V^\pm = V^\pm_2\oplus V^\pm_1\oplus V^\pm_0$ the Peirce decomposition with respect to $\b e$. For $x\in V^+$ we write $x=x_2+x_1+x_0$ according to this Peirce decomposition. Recall that the Jordan pair $(V^+_2,V^-_2)$ admits invertible elements, i.e., elements $x\in V^+_2$ such that $Q_x\colon V^-_2\to V^+_2$ is invertible. In this case, $x^{-1}=Q_x^{-1}(x)$ denotes the inverse of $x$, which is an element in $V^-_2$. Then,
\[
	N_e=\Set{x\in V^+}{x_2\text{ invertible in $V^+_2$}}
\]
is open and dense in $V^+$. We consider the map
\begin{align}\label{eq:diffeo}
	\varphi_e\colon N_e\to N_e, \quad x\mapsto x+Q_{x_1}x_2^{-1}.
\end{align}

\begin{proposition}\label{prop:LOrbits}
For $0\leq k\leq r$ and $e\in\calV_k$ the map $\varphi_e:N_e\to N_e$ is a diffeomorphism which maps $N_e\cap(V_2^+\oplus V_1^+)$ onto an open subset of $\calV_k$. In particular, $\calV_k$ is an $L$-invariant (embedded) submanifold of $V^+$, and each $L$-orbit in $V^+$ is a union of connected components of $\calV_k$ for some fixed $k$.
\end{proposition}

\begin{proof}
Standard arguments show that $\varphi_e$ is smooth, and since $Q_{x_1}x_2^{-1}$ is an element of $V_0^+$ according to the Peirce rules, a straightforward computation shows that $(x\mapsto x-Q_{x_1}x_2^{-1})$ is a smooth inverse of $\varphi_e$. Next we note that
\[
	\varphi_e(x) = \B{x_1}{-x_2^{-1}}(x_2+x_0)\qquad\text{and}\qquad
	(\B{x_1}{-x_2^{-1}},\B{-x_2^{-1}}{x_1}^{-1})\in L.
\]
Since $L$ acts by automorphisms on $(V^+,V^-)$, it follows that $\varphi_e(x)$ has the same rank as $x_2+x_0$. Since the rank of orthogonal idempotents is additive, and since $x_2$ is invertible in $V_2^+$, it follows that 
\[
	\rank(\varphi_e(x)) = \rank(x_2+x_0) = \rank(e)+\rank(x_0).
\]
Therefore, $\varphi_e(x)$ is in $\calV_k$ if and only if $x_0=0$. It remains to consider the action of $L$ on $\calV_k$. Since $L$ acts by automorphisms, it is clear that $\calV_k$ is $L$-invariant. In order to prove that the $L$-orbits consist of connected components of $\calV_k$, it suffices to show that the derived map $\frakl\to T_e\calV_k$ is surjective, where $T_e\calV_k\simeq V_2^+\oplus V_1^+$. This immediately follows from the fact that $\frakl$ acts by Jordan pair derivations, and applying the derivation $(D_{x,e'},-D_{e',x})$ to $e$ yields $D_{x,e'}(e) = \ell\cdot x$ for $x\in V_\ell^+$, $\ell\in\{0,1,2\}$.
\end{proof}

\subsection{Fibration and polar decomposition}
For $e\in V^+$, the subspace $[e]=Q_eV^-$ is called the \emph{principal inner ideal} associated to $e$. If $\b e=(e,e')$ is a completion to an idempotent, the Peirce rules \eqref{eq:peircerules} imply that $[e] = V_2^+(\b e)$. In particular, this Peirce space is independent of the choice of $e'$. Moreover, the product $x\circ y=\tfrac{1}{2}\JTP{x}{e'}{y}$ turns $[e]$ into a Jordan algebra with unit element $e$. This Jordan algebra structure is in fact also independent of the choice of $e'$, since for $x=Q_eu\in[e]$, we have
\[
	x^2 = Q_xe' = Q_{Q_eu}{e'}=Q_eQ_uQ_ee' = Q_eQ_ue
\]
by the fundamental formula, and polarization of this identity also shows independence of the product $x\circ y$ with respect to the choice of $e'$. For a detailed introduction to Jordan algebras, we refer to \cite{BK66, FK94}.

\begin{remark}
	Let $\b e=(e,e')$ be a completion of $e\in\calV_k$ to an idempotent, and let
	$\b e=\b e_1+\cdots+\b e_k$ be a decomposition into primitive idempotents, $\b e_j=(e_j,e_j')$. 
	Then, the restriction of the operator $D_{e_j,e_j'}$ to $[e]$ coincides with (left) 
	multiplication by $e_j$ with respect to the Jordan algebra structure on $[e]$. Therefore, the 
	Peirce decomposition
	\[
		[e]=\bigoplus_{i\leq j} V_{ij}^+
	\]
	coincides with the Jordan algebraic Peirce decomposition given by the left multiplication 
	operators $L_{e_j}$ on $[e]$. Moreover, if $e$ and $e_1,\ldots,e_k$ are tripotents, the map 
	$x\mapsto Q_e\overline x$ is a Cartan involution of the Jordan algebra $[e]$, and the refined Peirce 
	decomposition $V_{ij}^+=A_{ij}^+\oplus B_{ij}^+$ coincides with the respective refined Jordan 
	algebraic Peirce decomposition.	By these considerations, it follows that the structure constants 
	of the Jordan algebra (namely the dimensions of the various refined Peirce spaces) coincide with 
	the structure constants of the Jordan pair $(V^+,V^-)$.
\end{remark}

In what follows, we fix $0\leq k\leq r$. Let
\[
	\calP_k=\set{[e]}{e\in\calV_k}
\]
denote the space of principal inner ideals in $V^+$ generated by elements of rank $k$. We call $\calP_k$ the \emph{$k$'th Peirce manifold} associated to $(V^+,V^-)$.

\begin{proposition}\label{prop:Peircemanifold}
	\begin{enumerate}
		\item\label{Peircemanifold1} The $k$'th Peirce manifold $\calP_k$ is a smooth compact manifold. Moreover, $\calP_k$
					is an	$L$-homogeneous space, and the stabilizer subgroup $Q_{[e]}$ of $[e]\in\calP_k$ in 
					$L$ is parabolic in $L$. A Levi decomposition of $Q_{[e]}$ is given by 
					$Q_{[e]}=L_{[e]}U_{[e]}$ with
					\[
						L_{[e]}=Z_L(D_{e,e'}),\qquad U_{[e]}=\set{B_{e,v}}{v\in V_1^-(\b e)},
					\]
					where $Z_L(D_{e,e'})$ denotes the centralizer of $D_{e,e'}$ in $L$, i.e., $L_{[e]}$ 
					consists of elements preserving the Peirce decomposition with respect to $\b e$. 
		\item\label{Peircemanifold2} The action of $L_{[e]}$ on $[e]$ is given by elements of the structure group $\Str([e])$ 
					of the Jordan algebra $[e]$, and $U_{[e]}$ acts trivially on $[e]$. Moreover, the 
					induced Lie algebra homomorphism $\Lie(L_{[e]})\to\Lie(\Str([e]))$ is onto.
\end{enumerate}
\end{proposition}

\begin{proof}
Clearly, $L$ acts on $\calP_k$, since $h[e]=[he]$ for $e\in\calV_k$ and $h\in L$. Recall from \cite[\S\,11.8]{Loo77} that $M\cap K\subseteq L$ acts transitively on the set of frames of tripotents (modulo signs) in $V^+$. Since each principal inner ideal is also generated by a maximal tripotent element, it follows that $\calP_k$ is $(M\cap K)$-homogeneous. In particular, $\calP_k$ is a compact, $L$-homogeneous manifold. For the following, we fix $[e]\in\calP_k$ with representative $e\in\calV_k$, and let $\b e=(e,e')$ be a completion to an idempotent with $e'\in V^-$. We first show that the adjoint action of $D_{e,e'}\in\frakl$ induces the eigenspace decomposition 
\begin{align}\label{eq:ldecomposition}
	\frakl=\frakl_-\oplus\frakl_0\oplus\frakl_+\qquad\text{with}\quad
	\left\{
	\begin{aligned}
		\frakl_-&=\set{D_{u,e'}}{u\in V_1^+(\b e)},\\ 
		\frakl_0&=\set{T\in\frakl}{[D_{e,e'},T]=0},\\
		\frakl_+&=\set{D_{e,v}}{v\in V_1^-(\b e)}.
	\end{aligned}\right.
\end{align}
Indeed, by the relation $[D_{e,e'},D_{u,v}] = D_{\JTP{e}{e'}{u},v}-D_{u,\JTP{e'}{e}{v}}$, it immediately follows that $\frakl_\pm$ is the $(\pm1)$-eigenspace of $\ad(D_{e,e'})$. Now consider $T\in\frakl$, and let $Te=u_2+u_1+u_0$ and $Te'=v_2+v_1+v_0$ be the decompositions according to the Peirce decomposition of $V$ with respect to $\b e$. Since
\[
	Te = TQ_ee' = \JTP{Te}{e'}{e} + Q_eTe',
\]
it follows that $u_0=0$ and $u_2=-Q_ev_2$. The same argument applied to $Te'$ yields $v_0=0$. Setting $T'=T-D_{u_1,e'}+D_{e,v_1}$, we thus obtain
\begin{align*}
	[D_{e,e'},T'] &= -D_{Te,e'}-D_{e,Te'} +D_{u_1,e'} + D_{e,v_1} \\
	&= -D_{u_2,e'} - D_{e,v_2}
	= D_{Q_ev_2,e'}-D_{e,v_2}=0.
\end{align*}
Here, the last step follows from the relation $D_{Q_xy,z}=D_{x,\JTP{y}{x}{z}} -D_{Q_xz,y}$. This proves \eqref{eq:ldecomposition}. This implies that $Q_{[e]}=N_L(\frakl_0\oplus\frakl_+)$ is a parabolic subgroup of $L$ with Levi decomposition $Q_{[e]}=L_{[e]}U_{[e]}$ given by
$$ L_{[e]}=Z_L(D_{e,e'}), \qquad 	U_{[e]}=\exp(\frakl_+). $$
Since $\exp(D_{e,v})=\B{e}{-v}$ due to the Peirce rules, this completes the proof of \eqref{Peircemanifold1}.\\
We next consider the action of $h\in Q_{[e]}$ on $[e]$. Due to the Peirce rules, it is clear that $h\in U_{[e]}$ acts as the identity on $[e]$. Now let $h\in L_{[e]}$. Recall that the quadratic representation $P_x$ of $x\in[e]$ is given by $P_x=Q_xQ_{e'}$, and the Jordan algebra trace form of $[e]$ is a constant multiple of $\tau_{[e]}(x,y)=\tau(x,Q_{e'}y)$. Moreover, $h$ acts on $[e]$ by structure automorphisms if and only if $P_{hx}=hP_xh^\#$, where $h^\#$ denotes the adjoint of $h$ with respect to $\tau_{[e]}$. One easily checks that $h^\#=Q_eh^{-1}Q_{e'}$, and it follows that $P_{hx}=hP_xh^\#$.\\
We finally note that the Lie algebra of $\Str([e])$ is generated by all $D_{x,y}$ with $x\in V_2^+$, $y\in V_2^-$, which also belong to $\frakl_0\subseteq\frakq_{[e]}$. Since $\Lie(L_{[e]})=\frakl_0$, it follows that $\Lie(L_{[e]})\to\Lie(\Str([e]))$ is onto.
\end{proof}

We next consider the relation between the Peirce manifold $\calP_k$ and the manifold $\calV_k$.

\begin{proposition}\label{prop:fiberbundle}
	The canonical projection 
	\[
		\pi_k\colon\calV_k\to\calP_k, \quad e\mapsto [e]
	\]
	is an $L$-equivariant fiber bundle with fiber over $[e]\in\calP_k$ consisting of the set 
	$[e]^\times$ of invertible elements in the Jordan algebra $[e]$. Moreover, each connected 
	component of $[e]^\times$ is a reductive symmetric space.
\end{proposition}

\begin{proof}
Fix $e\in\calV_k$, and let $L^0$ denote the identity component of $L$. By Proposition~\ref{prop:LOrbits}, the orbit $L^0\cdot e$ is the connected component of $\calV_k$ containing $e$. Likewise, $L^0\cdot[e]$ is the connected component of $\calP_k$ containing $[e]$. Since $\pi_k$ clearly is $L$-equivariant, it follows that $\pi_k$ is locally given as a projection of $L^0$-homogeneous spaces. Hence, $\calV_k$ is a fiber bundle over $\calP_k$. Concerning the fiber over $[e]$, recall that $x\in[e]$ is invertible in the Jordan algebra $[e]$ if and only if the quadratic operator $P_x=Q_xQ_{e'}$ is invertible. Equivalently, $[x]=[e]$, which amounts to the condition that $x$ has the same rank as $e$, so $x\in\calV_k$.\\
Let $Y\subseteq[e]^\times$ be a connected component of $[e]^\times$. We may assume that $e\in Y$, otherwise consider the mutation of the Jordan algebra $[e]$ by an element $y\in Y$, see \cite{BK66} for reference. Recall that $L_{[e]}$ denotes the Levi subgroup of the stabilizer $Q_{[e]}$ of $[e]\in\calP_k$ in $L$. Due to Proposition~\ref{prop:Peircemanifold}, the restriction map 
\begin{align}\label{eq:restictionmap}
	\rho_{[e]}\colon L_{[e]}\to\Str([e]), \quad h\mapsto h|_{[e]}
\end{align}
identifies the identity component of $L_{[e]}$ with the identity component of the Jordan algebraic structure group of $[e]$, denoted by $L'_{[e]}=\Str([e])^0$. It follows, that $Y$ is the orbit of $L'_{[e]}$ through $e$, i.e.
\[
	Y=L'_{[e]}\cdot e\cong L'/H'_e,
\]
where $H'_e$ denotes the stabilizer subgroup of $e$ in $L'_{[e]}$. This orbit is symmetric, since
\begin{align}\label{eq:symmetrycondition}
		L_{[e]}'^{\sigma,0}\subseteq H_e'\subseteq L_{[e]}'^\sigma,
\end{align}
where $\sigma$ is the involution on $L_{[e]}'$ given by $h\mapsto h^{-\#}$, where $h^\#=Q_eh^{-1}Q_{e'}$ is the adjoint of $h$ with respect to the Jordan trace form $\tau_{[e]}$ on $[e]$.
\end{proof}

\begin{remark}
	Let $\calE_k$ be the tautological vector bundle of $\calP_k$,
	\[
		\calE_k=\set{([e],x)}{e\in\calV_k,\, x\in[e]}\subseteq\calP_k\times V^+,
	\]
	and let $\calE_k^\times\subseteq\calE_k$ denote the fiber bundle over $\calP_k$ with fiber
	$[e]^\times$ over $[e]$. By means of Proposition~\ref{prop:fiberbundle}, $\calV_k$ is naturally 
	identified with the fiber bundle $\calE_k^\times$. Moreover, the topological closure 
	$\calV_k^\cl$ of $\calV_k$ in $V^+$ is a real algebraic variety, since it is the zero-set of the 
	ideal generated by Jordan minors associated to idempotents of rank $k+1$. It follows that
	\[
		\calV_k^\cl=\bigsqcup_{j=0}^k\calV_j,
	\]
	and the projection map
	\[
		\calE_k\to\calV_k^\cl, \quad ([e],x)\mapsto x
	\]
	is an $L$-equivariant resolution of singularities.
\end{remark}

The preceding propositions yield the following description of the $L$-orbits on $V^+$. Fix a frame $(e_1,\ldots, e_r)$ of tripotents in $V^+$, and let $e=e_1+\cdots+e_k\in\calV_k$ be the base element of the orbit
\[
	\calO_e=L\cdot e\subseteq\calV_k,\qquad \calO_e\cong L/H_e,
\]
where $H_e$ denotes the stabilizer of $e$ in $L$. The fibration of $\calV_k$ over $\calP_k$ now corresponds to the fibration
\begin{align}\label{eq:fibration}
	L/H_e\cong L\times_{Q_{[e]}} Q_{[e]}/H_e \qquad \text{with fiber} \qquad
	Q_{[e]}/H_e\cong L_{[e]}/(L_{[e]}\cap H_e),
\end{align}
where $Q_{[e]}$ and $L_{[e]}$ are given as in Proposition~\ref{prop:Peircemanifold}. Geometrically, the base of this fibration coincides with the Peirce manifold $\calP_k$, and the canonical fiber is realized as the $L_{[e]}$-orbit of $e$ in the Jordan algebra $[e]$. Let $L_{[e]}'$ denote the image of $L_{[e]}$ under the restriction map $\rho_{[e]}$ given in \eqref{eq:restictionmap}. Then, by Proposition~\ref{prop:Peircemanifold}~\eqref{Peircemanifold2} $L_{[e]}'$ is an open subgroup of the Jordan algebraic structure group $\Str([e])$. In general, $L$ has finitely many connected components, so we note that $L_{[e]}'$ might differ from the corresponding group in the proof of Proposition~\ref{prop:fiberbundle}. In any case, the canonical fiber
\begin{align}\label{eq:fiber}
	L_{[e]}/(L_{[e]}\cap H_e)\cong L_{[e]}'/H_e',
\end{align}
is a reductive symmetric space, where $H_e'$ denotes the stabilizer of $e\in[e]$ in $L_{[e]}'$, which still satisfies \eqref{eq:symmetrycondition} with respect to the involution $\sigma$ on $L_{[e]}'$.

The structure theory of reductive symmetric spaces yields a polar decomposition of the canonical fiber. Recall that $(M\cap K)$ denotes the maximal compact subgroup of $L$. Since $e$ is tripotent, $L_{[e]}$ is $\theta$-stable, so $(M\cap K)_{[e]}=(M\cap K)\cap L_{[e]}$ is maximal compact in $L_{[e]}$, and the image $(M\cap K)_{[e]}'$ of $(M\cap K)_{[e]}$ under the restriction map $\rho_{[e]}$ is maximal compact in $L_{[e]}'$. Define
\begin{align}\label{eq:fiberCartan}
	\fraka_k=\bigoplus_{j=1}^k \RR D_{e_j,\overline e_j},\qquad A_k=\exp(\fraka_k).
\end{align}
Since $D_{e_j,\overline e_j}$ acts on $[e]$ by the Jordan algebraic multiplication with $e_j$, we may consider $\fraka_k$ by abuse of notation also as a subalgebra of $\Lie(L_{[e]}')=\str([e])$, and $A_k$ can be considered as a subgroup of $L_{[e]}'$. Then, $\fraka_k$ is maximal abelian in the subspace of elements $X\in\str([e])$ satisfying $\sigma(X)=-X$, $\theta(X)=-X$. This yields the following polar decompositions of $L_{[e]}$ and $L$, which induce polar decompositions of the corresponding orbits.

\begin{proposition}\label{prop:polardecomposition}
	With the above notation,
	\[
		L_{[e]}'=(M\cap K)_{[e]}'A_kH_e',\qquad L =(M\cap K) A_k H_e.
	\]
\end{proposition}

\begin{proof}
The first identity is a standard result for reductive symmetric spaces. For the decomposition of $L$, note that Proposition~\ref{prop:Peircemanifold} yields $L=(M\cap K)Q_{[e]}$, since $Q_{[e]}$ is parabolic. Recall that $Q_{[e]}=L_{[e]}U_{[e]}$. Since $L_{[e]}'=L_{[e]}/\ker\rho_{[e]}$ with $U_{[e]}\subseteq\ker\rho_{[e]}\subseteq H_e$, this yields the second identity.
\end{proof}

For later use, we also note the following local description of the $L$-orbit $\calO_e$. Let $\overline Q_{[e]}=\theta(Q_{[e]})$ be the parabolic subgroup opposite to $Q_{[e]}$. Then,
\begin{align}\label{eq:oppositeParabolic}
	\overline Q_{[e]}=L_{[e]}\overline U_{[e]}\qquad\text{with}\qquad
	\overline U_{[e]}=\theta(U_{[e]})=\set{\B{v}{\overline e}}{v\in V_1^+}
\end{align}
is the corresponding Levi decomposition, and since $\overline U_{[e]}L_{[e]}U_{[e]}\subseteq L$ is open and dense, the $\overline Q_{[e]}$-orbit of $e\in\calO_e$ is open and dense in $\calO_e$. In the following, we aim for a description of this part of $\calO_e$ with respect to the parametrization of $\calO_e$ induced by the diffeomorphism $\varphi_e$ in \eqref{eq:diffeo}.

Let $V_\ell^+$, $\ell=0,1,2$, denote the Peirce spaces with respect to $(e,\overline e)$, and recall that $N_e\subseteq V^+$ denotes the open set of elements $x=x_2+x_1+x_0$ with invertible $x_2\in [e]$. Then $\varphi_e\colon N_e\to N_e$ is a diffeomorphism, and its restriction to $N_e\cap(V_2^+\oplus V_1^+)$ is a diffeomorphism onto an open and dense subset of $\calO_e$. We note that $V_2^+=[e]$. By means of \eqref{eq:oppositeParabolic} it is straightforward to check that $N_e\subseteq V^+$ is preserved under the standard action of $\overline Q_{[e]}$ on $V^+$, and hence the $\overline Q_{[e]}$-orbit of $e\in\calO_e$ is contained in $N_e\cap\calO_e$. 

\begin{lemma}\label{lem:pullbackaction}
	The open subset $N_e\subseteq V^+$ is invariant under the action of $\overline Q_{[e]}\subseteq L$, 
	and the pullback of this action along $\varphi_e$ is given and denoted by
	\[
		\xi(h)(x)=hx,\qquad
		\xi(\B{v}{\overline e})(x)=x-\JTP{v}{\overline e}{x_2}
	\]
	for $h\in L_{[e]}$, $v\in V_1^+$ and $x\in N_e$. In particular, the pullback 
	action of $\overline Q_{[e]}$ is linear again.
\end{lemma}

\begin{proof}
Since $\overline Q_{[e]}$ is generated by $L_{[e]}$ and elements of the form $\B{v}{\overline e}$, the invariance of $N_e$ under the action of $\overline Q_{[e]}$ follows immediately from the Peirce rules. By definition, the pullback of the $\overline Q_{[e]}$-action is given by $\xi(q)(x)=\varphi_e^{-1}(q\cdot\varphi_e(x))$. Since $h\in L_{[e]}$ acts on each Jordan pair $(V_\ell^+, V_\ell^-)$, $\ell=0,1,2$, by automorphisms, it follows that $h(x_2^{-1})=(hx_2)^{-1}$, which implies the formula for $\xi(h)$. For $q=\B{v}{\overline e}$, we obtain
\begin{align*}
	\xi(\B{v}{\overline e})(x)
		&= x_2+(x_1-\JTP{v}{\overline e}{x_2}) \\
		&\quad+
		(x_0 + Q_vQ_{\overline e}x_2-\JTP{v}{\overline e}{x_1}+Q_{x_1}x_2^{-1}-Q_{x_1-\JTP{v}{\overline e}{x_2}}x_2^{-1}),
\end{align*}
where the terms are sorted with respect to the Peirce space decomposition. Recall from Proposition~\ref{prop:LOrbits} that the restriction of $\varphi_e$ to $N_e\cap(V_2^+\oplus V_1^+)$ is a diffeomorphism onto an open subset of $N_e\cap\calV_k$. Since $\calV_k$ is $\overline Q_{[e]}$-invariant, it follows that $N_e\cap(V_2^+\oplus V_1^+)$ is invariant for the pullback action of $\overline Q_{[e]}$. Therefore, if $x_0=0$, the $V_0^+$-part of $\xi(q)$ must vanish for all $q$. Applied to $\xi(\B{v}{\overline e})$, it follows that $\xi(\B{v}{\overline e})(x_2+x_1)=x_2+(x_1-\JTP{v}{\overline e}{x_2})$, and comparing this with the formula above completes the proof. (We note that using \eqref{JP16} it can also be seen directly that the $V_0^+$-part of $\xi(\B{v}{\overline e})(x_2+x_1)$ vanishes.)
\end{proof}

This immediately implies:

\begin{proposition}\label{prop:orbitcoordinates}
	Let $\Omega_e$ denote the $L_{[e]}$-orbit of $e\in[e]$. Then the restriction of $\varphi_e$ to 
	the open subset $\Omega_e+V_1^+\subseteq N_e$ is a diffeomorphism onto the $\overline Q_{[e]}$-orbit 
	of $e$ in $\calO_e$. The pullback of the $\overline Q_{[e]}$-action along this diffeomorphism is 
	given and denoted by
	\[
		\xi(h)(x_2+x_1)=hx_2+hx_1,\qquad
		\xi(\B{v}{\overline e})(x_2+x_1)=x_2 + (x_1-\JTP{v}{\overline e}{x_2})
	\]
	for $h\in L_{[e]}$, $v\in V_1^+$ and $x_2+x_1\in\Omega_e+V_1^+\subseteq V_2^+\oplus V_1^+$.
\end{proposition}

\subsection{Equivariant measures}

In this section we determine which of the $L$-orbits in $V^+$ carry an $L$-equivariant measure. Fix $0\leq k\leq r$, and let $\calO_e=L\cdot e$ be the $L$-orbit through $e\in\calV_k$. We may assume that $e$ is tripotent, and $e=e_1+\cdots+e_k$, where $e_1,\ldots,e_r$ is a frame of tripotents. We adopt all notations from the last section, so in particular, $H_e$ denotes the stabilizer of $e$ in $L$, so $\calO_e\cong L/H_e$. Recall that a measure $\td\mu$ on $\calO_e$ is called $\chi$-equivariant, $\chi$ a positive character of $L$, if 
\[
	\td\mu(hx)=\chi(h)\td\mu(x) \qquad \text{for }h\in L.
\]
A simple criterion for the existence and uniqueness of equivariant measures is given in terms of the modular functions $\Delta_L$ and $\Delta_{H_e}$ of $L$ and $H_e$: The orbit $\calO_e$ admits a $\chi$-equivariant measure if and only if 
\[
	\chi(h)=\frac{\Delta_L(h)}{\Delta_{H_e}(h)} \qquad \text{for all $h\in H_e$.}
\]
Since $L$ is reductive and hence unimodular we have $\Delta_L=1$, and therefore $\calO_e$ carries an $L$-equivariant measure if and only if the modular function $\Delta_{H_e}$ extends to a positive character of $L$. Moreover, uniqueness of the equivariant measure corresponds to uniqueness of this extension. Recall that all positive characters of $L$ are of the form
\[
	\chi_\lambda(h)=|\Det_{V^+}(h)|^{\frac{\lambda}{p}}
\]
for some $\lambda\in\RR$, see \eqref{eq:characterformula}.

\begin{lemma}
The group $H_e\subseteq L$ decomposes as the semidirect product $H_e=H_{e,\overline e}\ltimes U_{[e]}$, where $H_{e,\overline e}=H_e\cap L_{[e]}$ and $L_{[e]}$, $U_{[e]}$ are given in Proposition~\ref{prop:Peircemanifold}. Moreover, the modular function $\Delta_{H_e}$ of $H_e$ is given by
$$ \Delta_{H_e}(h) = |\Det_{V_1^+}(h')|^{-1} = \frac{|\Det_{V_0^+}(h')|}{|\Det_{V^+}(h')|}\qquad\text{for }h=h'u\in H_{e,\overline e}U_{[e]}, $$
where $V^\pm=V_2^\pm\oplus V_1^\pm\oplus V_0^\pm$ denotes the Peirce decomposition with respect to $(e,\overline e)$.
\end{lemma}
\begin{proof}

Note that $H_{e,\overline e}\subseteq H_e$ is the subgroup that fixes $e\in V^+$ as well as $\overline e\in V^-$. Then each Peirce space $V_k^\pm$ is invariant under the action of $H_{e,\overline e}$. Recall from Proposition~\ref{prop:Peircemanifold} the subgroups $L_{[e]},U_{[e]}\subseteq L$. Then due to \eqref{eq:fibration}, $H_e$ decomposes into the semidirect product $H_e=H_{e,\overline e}\ltimes U_{[e]}$. We may identify the Lie algebra of $U_{[e]}$ with $V_1^-$. Since $\Ad(h)D_{e,v}=D_{e,hv}$ for $h\in H_{e,\overline e}$ and $v\in V_1^-$, this identification is $H_{e,\overline e}$-equivariant. Then, the modular function of the semidirect product $H_e=H_{e,\overline e}U_{[e]}$ is given by
\[
	\Delta_{H_e}(h)=|\Det_{V_1^-}(h')|\Delta_{H_{e,\overline e}}(h')\Delta_{U_{[e]}}(u),
\]
where $h=h'u\in H_{e,\overline e}U_{[e]}$. Since $H_{e,\overline e}$ is $\theta$-stable, it is reductive, and hence unimodular. Moreover, $U_{[e]}$ is abelian. Therefore, the modular function of $H_e$ simplifies to
\begin{align}\label{eq:modularfunction}
	\Delta_{H_e}(h)=|\Det_{V_1^-}(h')|\qquad\text{for }h=h'u\in H_{e,\overline e}U_{[e]}.
\end{align}
Since the trace form $\tau$ gives an $H_{e,\overline e}$-invariant, non-degenerate pairing of $V_1^+$ and $V_1^-$, it follows that
\begin{equation}
	\Det_{V_1^-}(h)=\Det_{V_1^+}(h)^{-1}\qquad\text{for }h\in L\label{eq:determinantVpm}
\end{equation}
and hence the first formula for $\Delta_{H_e}$ follows. Moreover, since $h'\in H_{e,\overline e}$ preserves the Peirce spaces $V_\ell^+$, $\ell=0,1,2$, we obtain
\begin{align}\label{eq:determinantproduct}
	\Det_{V^+}(h') = \Det_{V^+_2}(h')\cdot\Det_{V^+_1}(h')\cdot\Det_{V^+_0}(h').
\end{align}
Furthermore, due to Proposition~\ref{prop:Peircemanifold}, $h'$ acts on the simple Jordan algebra $V_2^+=[e]$ as a structure automorphism preserving the identity element $e$. Recall that the determinant of such an automorphism on a simple Jordan algebra is of absolute value $1$ (see e.g. \cite{FK94}), hence $|\Det_{V^+_2}(h')|=1$ for all $h'\in H_{e,\overline e}$. In combination with \eqref{eq:modularfunction}, \eqref{eq:determinantVpm} and \eqref{eq:determinantproduct} the second formula for $\Delta_{H_e}$ follows.
\end{proof}

\begin{theorem}\label{thm:equivariantmeasures}
	For $0\leq k\leq r$ the $L$-orbit $\calO_e=L.e$ with $e\in\calV_k$ carries an $L$-equivariant 
	measure with positive character $\chi_\lambda$ if and only if one of the following 
	is satisfied
	\begin{enumerate}
		\item\label{equivariantmeasures1} $k=0$ and $\lambda=0$,
		\item\label{equivariantmeasures2} $1\leq k\leq r-1$, $V\not\simeq(M(p\times q,\FF),M(q\times p,\FF))$, $\FF=\RR,\CC,\HH$ with $p\neq q$, and $\lambda=kd$,
		\item\label{equivariantmeasures3} $k=r$, $V$ is unital and $\lambda\in\RR$,
		\item\label{equivariantmeasures4} $k=r$, $V$ is non-unital and $\lambda=p$.
	\end{enumerate}
	Moreover, the equivariant measure (if it exists) is unique up to scalars.
\end{theorem}

\begin{proof}
By the previous lemma $\calO_e$ admits an $L$-equivariant measure if and only if $|\Det_{V_0^+}|$ (or equivalently $|\Det_{V_1^+}|$) is a power of $|\Det_{V^+}|$ on $H_{e,\overline e}$.\\
For $k=0$ we have $\calO_e=\{0\}$ and \eqref{equivariantmeasures1} follows. Next assume that $k=r$, i.e., $\calO_e$ is an open orbit. Then $V_0^+=\{0\}$, and $\Delta_{H_e}$ extends to the character $\chi_{-p}$. If, in addition, $V$ is unital, then $V_1^+=\{0\}$, and $\Delta_{H_e}(h)=1$ for all $h\in H_e$. In this case, any character $\chi_\lambda$ is an extension of $\Delta_{H_e}$. If $V$ is non-unital, $\Delta_{H_e}(h)$ is non-trivial and the extension to $\chi_{-p}$ is unique. This implies \eqref{equivariantmeasures3} and \eqref{equivariantmeasures4}, so that it remains to show \eqref{equivariantmeasures2}.\\
Let $1\leq k\leq r-1$. Then, $V_0^+$ is non-trivial. If $|\Det_{V_0^+}(h)|$ is a power of $|\Det_{V^+}(h)|$, evaluation at $h=\exp(tD_{e_{k+1},\overline e_{k+1}})$, $t\in\RR$, shows that 
\[
	\Delta_{H_e}(h)=|\Det_{V^+}(h)|^{-\frac{kd}{p}}=\chi_{-kd}(h),
\]
which is non-trivial on $H_e$. This shows that an $L$-equivariant measure on $\calO_e$ (if it exists) is unique up to normalization, and the corresponding positive character is $\chi_{kd}$. It remains to determine those cases, in which $|\Det_{V_0^+}(h)|$ is a power of $|\Det_{V^+}(h)|$. Since $|\Det_{V^+}|$ and $|\Det_{V^+_0}|$ are positive characters it is sufficient to show that $\Tr_{V^+_0}$ is a scalar multiple of $\Tr_{V^+}$ on $\frakh_{e,\overline e}$ if and only if $V\not\simeq M(p\times q,\FF)$, $\FF=\RR,\CC,\HH$ with $p\neq q$. Let $T\in\frakh_{e,\overline e}$. Since $T|_{V_2}$ is a structure endomorphism of the simple Jordan pair $V_2$ and the structure endomorphisms are generated by $D_{u,v}$ with $(u,v)\in V_2^+\times V_2^-$, there exists $T_2\in\frakl$ which is a linear combination of $D_{u,v}$, $(u,v)\in V_2^+\times V_2^-$, such that $T|_{V_2}=T_2|_{V_2}$. The same is true for $T|_{V_0}$, so there exists $T_0\in\frakl$ which is a linear combination of $D_{u,v}$, $(u,v)\in V_0^+\times V_0^-$, such that $T|_{V_0}=T_0|_{V_0}$. Since $D_{u,v}|_{V_0^+}=0$ for $(u,v)\in V_2^+\times V_2^-$ and $D_{u,v}|_{V_2^+}=0$ for $(u,v)\in V_0^+\times V_0^-$ the structure endomorphism $T_1=T-T_2-T_0$ vanishes on both $V_2$ and $V_0$. Note that $T_2e=0$ since $Te=0$.
Now, we have
$$ \Tr_{V^+}(T) = \Tr_{V^+}(T_2)+\Tr_{V^+}(T_1)+\Tr_{V^+}(V_0). $$
We claim that $\Tr_{V^+}(T_2)=\frac{p}{p_2}\Tr_{V^+_2}(T_2)$. In fact, the bilinear forms
$$ B(u,v)=\Tr_{V^+}(D_{u,v}) \qquad \mbox{and} \qquad B_2(u,v)=\Tr_{V_2^+}(D_{u,v}) $$
on $V_2^+\times V_2^-$ are both $L_{[e]}$-invariant. Since
$$ L_{[e]}\to\Str(V_2), \quad g\mapsto g|_{V_2} $$
is surjective, these forms are also $\Str(V_2)$-invariant. The form $B_2$ is non-degenerate since $V_2$ is simple, hence $B(u,v)=B_2(Au,v)$ for some $A\in\End(V_2^+)$. Both forms being $\Str(V_2)$-invariant, the endomorphism $A$ commutes with $\Str(V_2)$. Since $\Str(V_2)$ acts irreducibly on $V_2$, the map $A$ is a scalar multiple of the identity (allowing the scalar to be complex if $V_2$ is a complex Jordan pair). Evaluating both forms at $(u,v)=(e,\overline e)$ yields $B(u,v)=\frac{p}{p_2}B_2(u,v)$. Now, $T_2$ is a linear combination of operators of the form $D_{u,v}$, $(u,v)\in V_2^+\times V_2^-$ and hence $\Tr_{V^+}(T_2)=\frac{p}{p_2}\Tr_{V_2^+}(T_2)$. The same argument can be applied to $T_0$ acting on $V_0^+$ to obtain $\Tr_{V^+}(T_0)=\frac{p}{p_0}\Tr_{V_0^+}(T_0)$. Now, $\Tr_{V_2^+}(T_2)=0$ since $T_2$ is a structure automorphism of the simple Jordan algebra $V_2^+$ which annihilates the identity element $e$ (see e.g. \cite{FK94}). Hence
$$ \Tr_{V^+}(T) = \Tr_{V^+}(T_1)+\frac{p}{p_0}\Tr_{V_0^+}(T_0) = \Tr_{V_1^+}(T_1)+\frac{p}{p_0}\Tr_{V_0^+}(T). $$
Therefore, $\Tr_{V^+_0}$ is a scalar multiple of $\Tr_{V^+}$ if and only if $\Tr_{V_1^+}(T_1)=0$ for every $T_1\in\frakl_{[e]}$ with $T|_{V_2^+\oplus V_0^+}=0$. By Proposition~\ref{prop:ClassificationV1} this is the case if and only if $V\not\simeq(M(p\times q,\FF),M(q\times p,\FF))$, $\FF=\RR,\CC,\HH$ with $p\neq q$.
\end{proof}

The next goal is to determine an explicit formula for integration over $\calO_e$ with respect to an $L$-equivariant measure by means of the polar decomposition given in Proposition~\ref{prop:polardecomposition}. Recall from \eqref{eq:fibration} the fibration of $\calO_e$ over the Peirce manifold $\calP_k$ with canonical fiber $L_{[e]}'/H_e'$, which is a symmetric space. We first determine a polar decomposition of the integral over the fiber $L_{[e]}'/H_e'$. With $\fraka_k$ as in \eqref{eq:fiberCartan}, we define
\[
	\fraka_k^+=\Set{\sum_{j=1}^k \frac{\tau_j}{2}\,D_{e_j,\overline e_j}}{\tau_1>\cdots>\tau_k}\subseteq\fraka_k,
\]
and for $(\tau_1,\ldots,\tau_r)\in\RR^r$, we set 
\[
	a_\tau=\exp(\tfrac{\tau_1}{2}D_{e_1,\overline e_1}+\cdots+\tfrac{\tau_k}{2}D_{e_k,\overline e_k}).
\]
We endow $\fraka_k^+$ with a Lebesgue measure $\td\tau$.

\begin{proposition}\label{prop:fiberintegration}
	With the notation as above, and after suitable normalization of the measures,
	\begin{align}\label{eq:fiberintegralformula}
		\int_{L_{[e]}'} f(h)\,\td h=\int_{(M\cap K)_{[e]}'}\int_{\fraka_k^+}\int_{H_e'}
		 f(m'a_\tau\ell')\cdot J_e(\tau)\,\td\ell'\,\td\tau\,\td m'
	\end{align}
	for all $f\in L^1(L_{[e]}')$, where the Jacobian $J_e(\tau)$ is given by 
	\[
		J_e(\tau_1,\ldots,\tau_k)
		=\prod_{1\leq i<j\leq k}
		\sinh^{d_+}\left(\frac{\tau_i-\tau_j}{2}\right)
		\cosh^{d_-}\left(\frac{\tau_i-\tau_j}{2}\right).
	\]
\end{proposition}

\begin{proof}
This integral formula is a standard result from the theory of semisimple symmetric spaces, see e.g.\ \cite{HS94}, based on the following data: The root system $\Phi(\mathfrak{str}([e]),\fraka_k)$ is of type $A_{k-1}$. More precisely, let $\gamma_i\in\fraka_k^*$ be defined by $\gamma_i(D_{e_j,\overline e_j})=\delta_{ij}$, then
\[
	\Phi(\mathfrak{str}([e]),\fraka_k) = \Set{\frac{\gamma_i-\gamma_j}{2}}{1\leq i\neq j\leq k}
\]
with root spaces 
\[
	\mathfrak{str}([e])_{\frac{\gamma_i-\gamma_j}{2}}=\set{D_{x,\overline e_j}}{x\in V_{ij}^+}.
\]
Therefore, the multiplicity of $\frac{\gamma_i-\gamma_j}{2}$ is $2d=d_++d_-$. Moreover, $V_{ij}^+=A_{ij}^+\oplus B_{ij}^+$ corresponds to the decomposition of $\mathfrak{str}([e])_{\frac{\gamma_i-\gamma_j}{2}}$ into $\sigma\theta$-stable subspaces with according multiplicities $\dim A_{ij}^+=d_+$ and $\dim B_{ij}^+=d_-$. Fix the ordering $\gamma_1>\gamma_2>\cdots>\gamma_k$. Then, $\fraka_k^+$ is the positive Weyl chamber. Moreover, this also coincides with the Weyl chamber of $\Sigma(\mathfrak{str}([e])^{\sigma\theta},\fraka_k)$.
\end{proof}

For $0\leq k\leq r$ let
\begin{equation}
 C_k^+ = \Set{t\in\RR^k}{t_1>\ldots>t_k>0}\label{eq:Defbt}
\end{equation}
and put
\begin{equation}
 b_t = t_1e_1+\cdots+t_ke_k, \qquad t\in C_k^+.\label{eq:DefCk+}
\end{equation}
Then by Proposition~\ref{prop:polardecomposition} the map
$$ (M\cap K)\times C_k^+\to\calO_k, \quad (m,t)\mapsto mb_t $$
is onto an open dense subset of $\calO_k$.

\begin{proposition}\label{prop:equivariantmeasure}
	Assume that $\calO_e=L\cdot e\subseteq\calV_k$, $0\leq k\leq r$, admits an $L$-equivariant measure $\td\mu$ with positive character $\chi_\lambda$. Then, for all $f\in L^1(\calO_e,\td\mu)$,
	\begin{align*}
		\int_{\calO_e} f(x)\,\td\mu(x)
		&= \int_{M\cap K}\int_{C_k^+} f(mb_t)\cdot J(t)\,\td t\,\td m,
	\end{align*}
	where
	\[
		J(t)=\prod_{j=1}^k t_j^{\lambda+(r-2k+1)d+\frac{b}{2}-1}
			\prod_{1\leq i< j\leq k}(t_i-t_j)^{d_+}(t_i+t_j)^{d_-}.
	\]
\end{proposition}

\begin{proof}
Recall that the equivariant measure $\td\mu$ on $\calO_e=L/H_e$ is uniquely determined by the relation
\begin{align}\label{eq:equivariantintegral}
	\int_{L/H_e}\int_{H_e} f(gh)\,\td h\,\td\mu(g H_e) = \int_L\chi_\lambda(g)f(g)\,\td g
\end{align}
for all $f\in L^1(L)$. We determine a suitable decomposition of the right hand side. Then, the comparision with the left hand side proves the statement. In the following, all measures on Lie groups are meant to be left-invariant.\\
We first use the fibration \eqref{eq:fibration} of $\calO_e$. Recall that $Q_{[e]}\subseteq L$ is a parabolic subgroup of $L$, hence $L=(M\cap K)Q_{[e]}$, and $Q_{[e]}=L_{[e]}U_{[e]}$ is the Levi decomposition of $Q_{[e]}$. Therefore, 
\begin{align*}
	\int_L\chi_\lambda(g)f(g)\,\td g
		&= \int_{M\cap K}\int_{Q_{[e]}}\Delta_{Q_{[e]}}(q)\chi_\lambda(mq)f(mq)\,\td q\,\td m\\
		&= \int_{M\cap K}\int_{L_{[e]}}\int_{U_{[e]}}
			 \frac{\Delta_{U_{[e]}}(u)}{\Delta_{Q_{[e]}}(u)}\Delta_{Q_{[e]}}(hu)
			 	\chi_\lambda(h)f(mhu)\,\td u\,\td h\,\td m\\
		&= \int_{M\cap K}\int_{L_{[e]}}\int_{U_{[e]}}
			 \Delta_{Q_{[e]}}(h)\chi_\lambda(h)f(mhu)\,\td u\,\td h\,\td m.
\end{align*}
Here, we used that $\chi_\lambda(mhu)=\chi_\lambda(h)$, since $\chi_\lambda$ is a positive character on $L$, and hence $\chi_\lambda$ is trivial on the compact subgroup $M\cap K\subseteq L$ and the unipotent subgroup $U_{[e]}\subseteq L$. The modular function of $Q_{[e]}=L_{[e]}U_{[e]}$ is given by
\[
	\Delta_{Q_{[e]}}(hu)=|\Det_{\Lie(U_{[e]})}\Ad(h)|,\qquad hu\in L_{[e]}U_{[e]}.
\]
As in \eqref{eq:modularfunction}, we may identify the Lie algebra of $U_{[e]}$ with $V_1^-$. Then, the adjoint action of $h\in L_{[e]}$ on $v\in V_1^-$ is given by $\Ad(h)v=\JTP{\overline e}{he}{hv}$. We thus obtain
\[
	\Delta_{Q_{[e]}}(hu)=|\Det_{V_1^-}(D_{\overline e,he})|\cdot|\Det_{V_1^-}(h)|,
	\qquad hu\in L_{[e]}U_{[e]}.
\]
Now recall from \eqref{eq:fiber} that the fiber
\[
	L_{[e]}/H_{e,\overline e}\cong L_{[e]}'/H_e'
\]
is a symmetric space, where $L_{[e]}'$ and $H_e'$ is the image of $L_{[e]}$ and $H_e'$ under the restriction map $\rho_{[e]}$ defined in \eqref{eq:restictionmap}. We may identify $L_{[e]}'$ and $H_e'$ with the quotients $L_{[e]}/R$ and $H_{e,\overline e}/R$, where $R=\ker\rho_{[e]}$. Since normal subgroups of unimodular Lie groups are unimodular, we thus obtain
\[
	\int_{L_{[e]}} F(h)\,\td h = \int_{L_{[e]}/R}\int_R F(hr)\,\td r\,\td(hR)
	= \int_{L_{[e]}'}\int_R F(hr)\,\td r\,\td h.
\]
Now, applying Proposition~\ref{prop:fiberintegration} yields
\begin{align*}
	\int_L&\chi_\lambda(g)f(g)\,\td g\\
		&= \int_{M\cap K}\int_{\fraka_k^+}\int_{H_e'}\int_R\int_{U_{[e]}}
			 \Delta_{Q_{[e]}}(a_\tau\ell'r)\chi_\lambda(a_\tau\ell'r)J_e(\tau)
			 f(ma_\tau\ell'ru)\,\td u\,\td r\,\td \ell'\,\td\tau\,\td m.
\end{align*}
Since $H_e'=H_{e,\overline e}/R$ and $H_e=H_{e,\overline e}U_{[e]}$, we thus obtain
\begin{align*}
	\int_L\chi_\lambda(g)f(g)\,\td g
		&=\int_{M\cap K}\int_{\fraka_k^+}\int_{H_{e,\overline e}}\int_{U_{[e]}}
			 \Delta_{Q_{[e]}}(a_\tau\ell)\chi_\lambda(a_\tau\ell)J_e(\tau)f(ma_\tau\ell u)\,\td u\,\td\ell\,\td\tau\,\td m\\
		&=\int_{M\cap K}\int_{\fraka_k^+}\int_{H_e} \Delta_{Q_{[e]}}(a_\tau h)
			\chi_\lambda(a_\tau h)J_e(\tau)f(ma_\tau h)\,\td h\,\td\tau\,\td m.
\end{align*}
We note that $\Delta_{Q_{[e]}}(h)=\Delta_{H_e}(h)$ for all $h\in H_e$. Moreover, $\Delta_{H_e}(h)=\chi_\lambda(h)^{-1}$ by equivariance of $\td\mu$. This eventually implies
\begin{align*}
	\int_L\chi_\lambda(g)f(g)\,\td g
		=\int_{M\cap K}\int_{\fraka_k^+}\left(\int_{H_e} f(ma_\tau h)dh\right)\tilde J(\tau)\,\td\tau\,\td m
\end{align*}
with
\[
	\tilde J(\tau)=\Delta_{Q_{[e]}}(a_\tau)\chi_\lambda(a_\tau)J_e(\tau).
\]
Comparing this integral formula with \eqref{eq:equivariantintegral} shows that
\[
	\int_{\calO_e} f(x)\,\td\mu(x)
		= \int_{M\cap K}\int_{\fraka_k^+}f(ma_\tau\cdot e)\tilde J(\tau)\,\td\tau\,\td m
\]
for $f\in L^1(\calO_e,\td\mu)$. For $a_\tau=\exp(\frac{\tau_1}{2}D_{e_1,\overline e_1}+\cdots+\frac{\tau_k}{2}D_{e_k,\overline e_k})$, the joint Peirce decomposition with respect to the frame $(e_1,\ldots, e_r)$ shows that
\[
	\Delta_{Q_{[e]}}(a_\tau) = \prod_{j=1}^k e^{((r-k)d+\frac{b}{2})\tau_j} \qquad \mbox{and} \qquad \chi_\lambda(a_\tau) = \prod_{j=1}^k e^{\lambda \tau_j}.
\]
Therefore,
\[
	\tilde J(\tau) = \prod_{j=1}^k e^{(\lambda+(r-k)d+\frac{b}{2})\tau_j}
			\prod_{1\leq i<j\leq k}\sinh^{d_+}\left(\frac{\tau_i-\tau_j}{2}\right)
				\cosh^{d_-}\left(\frac{\tau_i-\tau_j}{2}\right).
\]
Finally, the change of coordinates $t_j=e^{\tau_j}$ yields the proposed integration formula.
\end{proof}

Finally, we determine a formula for the $L$-equivariant measure $\td\mu$ in the coordinates of $\calO_e$ given by the diffeomorphism $\varphi_e$ in \eqref{eq:diffeo}. Recall from Proposition~\ref{prop:orbitcoordinates} that the restriction of $\varphi_e$ to $\Omega_e+V_1^+$ is a diffeomorphism onto the $\overline Q_{[e]}$-orbit of $e\in\calO_e$, which is an open and dense subset of $\calO_e$. Here, $\Omega_e$ is the $L_{[e]}$-orbit of $e\in[e]$, which is open in $[e]=V_2^+$.

\begin{proposition}\label{prop:PullbackMeasure}
	Let $e\in\calV_k$, $1\leq k\leq r-1$, and let $\td\mu$ be a $\chi_{kd}$-equivariant measure on $\calO_e$. Then, the pullback of $\td\mu$ along $\varphi_e$ is given by
	\[
		\int_{\calO_e} f(x) \td\mu(x)
		= \int_{\Omega_e+V_1^+}f(\varphi_e(x_2+x_1))|\Delta(x_2)|^{kd-p}\,\td\lambda(x_2+x_1),
	\]
	where $\td\lambda$ denotes a suitably normalized Lebesgue measure on $V_2^+\oplus V_1^+$.
\end{proposition}

\begin{proof}
First we note that the pullback of $\td\mu$ is absolutely continuous with respect to the Lebesgue measure on $V_2^+\oplus V_1^+$. Therefore, $\varphi^*\td\mu=g\,\td\lambda$ for a nowhere vanishing continuous map $g$. Now, $\overline Q_{[e]}$-equivariance of $\varphi_e$ with respect to the actions of $\overline Q_{[e]}$ given as in Proposition~\ref{prop:orbitcoordinates} implies that
\[
	g(qe)=\frac{\chi_{kd}(q)}{|\Det_{V_2^+\oplus V_1^+}(d_e\xi(q))|}\,g(e)
\]
for all $q\in\overline Q_{[e]}$. With $q=h\B{v}{\overline e}$ according to \eqref{eq:oppositeParabolic}, we obtain
\[
	\Det_{V_2^+\oplus V_1^+}(d_e\xi(q))=\Det_{V_2^+\oplus V_1^+}(h), \qquad \chi_{kd}(q)=\Det_{V^+}(h)^{\frac{kd}{p}}.
\]
Hence $g$ is independent of $x_1$, and as a function of $x_2$ it satisfies
$$ g(hx_2) = \frac{|\Det_{V^+}(h)|^{\frac{kd}{p}}}{|\Det_{V_2^+\oplus V_1^+}(h)|}g(x_2). $$
Similar to the proof of Theorem~\ref{thm:equivariantmeasures} (but now on the group level) we decompose $h\in L_{[e]}$ as $h=h_2h_1h_0$ with
$$ h|_{V_2^+}=h_2|_{V_2^+}, \quad h_2|_{V_0^+}=\id_{V_0^+}, \qquad \mbox{and} \qquad h|_{V_0^+}=h_0|_{V_0^+}, \quad h_0|_{V_2^+}=\id_{V_2^+}. $$
Then $h_1$ acts trivially on $V_2^+\oplus V_0^+$ and by Proposition~\ref{prop:ClassificationV1}~\ref{ClassificationV1-2} we have $\Det_{V_2^+\oplus V_1^+}(h_1)=\Det_{V^+}(h_1)=1$. Further, by the proof of Theorem~\ref{thm:equivariantmeasures} we have
\begin{align*}
 \Det_{V_0^+}(h_0) &= \Det_{V^+}(h_0)^{\frac{p_0}{p}} = \Det_{V^+}(h_0)^{1-\frac{kd}{p}},\\
 \Det_{V_2^+}(h_2) &= \Det_{V^+}(h_2)^{\frac{p_2}{p}}.
\end{align*}
Together this gives
\begin{align*}
 \frac{\Det_{V^+}(h)^{\frac{kd}{p}}}{\Det_{V_2^+\oplus V_1^+}(h)} &= \Det_{V^+}(h_2)^{\frac{kd}{p}-1}\Det_{V^+}(h_1)^{\frac{kd}{p}-1}\Det_{V^+}(h_0)^{\frac{kd}{p}-1}\Det_{V_0^+}(h_0)\\
 &= \Det_{V_2^+}(h_2)^{\frac{1}{p_2}(kd-p)} = \Det_{V_2^+}(h)^{\frac{1}{p_2}(kd-p)}
\end{align*}
and hence
$$ g(hx_2) = |\Det_{V_2^+}(h)|^{\frac{1}{p_2}(kd-p)}g(x_2). $$
Since $L_{[e]}\to\Str(V_2^+),\,h\mapsto h|_{V_2^+}$ is surjective and the (absolute value of the) Jordan determinant $|\Delta(x_2)|$ is the (up to scalar multiples) unique function on $V_2^+$ such that $|\Delta(gx_2)|=|\Det(g)|^{\frac{1}{p_2}}|\Delta(x_2)|$ we find that
\begin{equation*}
 g(x_2) = |\Delta(x_2)|^{kd-p}. \qedhere
\end{equation*}
\end{proof}

\section{Bessel operators on Jordan pairs}\label{sec:Besseloperator}

Bessel operators were first defined for Euclidean Jordan algebras by Dib~\cite{Dib90} (see also Faraut--Koranyi~\cite{FK94}) and later generalized to arbitrary semisimple Jordan algebras by Mano~\cite{Man08} and Hilgert--Kobayashi--M\"{o}llers~\cite{HKM14}. We extend this definition to Jordan pairs and show that for certain parameters the operators are tangential to the rank submanifolds $\calV_k$ in $V^+$. In the case where $\calV_k$ carries an $L$-equivariant measure we further prove that the Bessel operators are symmetric with respect to the corresponding $L^2$-inner product. For Jordan algebras these results were obtained by Hilgert--Kobayashi--M\"{o}llers~\cite{HKM14} using zeta functions which are not available for Jordan pairs. The proofs we give here work uniformly for all Jordan pairs, since they merely use local parametrizations of the submanifolds $\calV_k$ and the explicit form of the measures $\td\mu_k$ in these parametrizations.

\subsection{Definition, equivariance, and symmetry}\label{sec:BesselDefinition}

We fix a basis $\{c_\alpha\}_{\alpha=1,\ldots, n}$ of $V^+$, and let $\{\widehat c_\alpha\}_{\alpha=1,\ldots, n}$ be the dual basis of $V^-$ with respect to the trace form $\tau$ defined in \eqref{eq:traceform}. Then, for any $\lambda\in\CC$ the \emph{Bessel operator}
\[
	\calB_\lambda\colon C^\infty(V^+)\to C^\infty(V^+)\otimes V^-
\]
is defined by the formula
\[
	\calB_\lambda f(x)
		=\frac{1}{2}\sum_{\alpha,\beta} \frac{\partial^2 f}{\partial x_\alpha\partial x_\beta}(x)
			\JTP{\widehat c_\alpha}{x}{\widehat c_\beta} +
			\lambda\,\sum_\alpha\frac{\partial f}{\partial x_\alpha}(x)\,\widehat c_\alpha,
\]
which is easily seen to be independent of the choice of $\{c_\alpha\}_{\alpha=1,\ldots, n}$. On a formal level, this is also sometimes denoted as
\[
	\calB_\lambda = Q\left(\frac{\partial}{\partial x}\right)x + \lambda\frac{\partial}{\partial x},
\]
where $\tfrac{\partial}{\partial x}\colon C^\infty(V^+)\to C^\infty(V^+)\otimes V^-$ denotes the gradient with respect to $\tau$, which is defined by the condition
\[
	\td_vf(x) = \tau\left(v,\frac{\partial f}{\partial x}(x)\right)
	\qquad\text{for all }v\in V^+.
\]
One of the basic properties of the Bessel operator is its equivariance under the action of $L\subseteq G$. For $h\in L$, let $\rho(h)$ denote the action on functions on $V^+$,
\[
	(\rho(h)f)(x) = f(h^{-1}x).
\]
Recall that $hx=\Ad(h)x$ denotes the adjoint action of $h\in L$ on $x\in V^+=\frakn$ resp. $V^-=\overline\frakn$.

\begin{lemma}\label{lem:BesselEquivariance}
	For any $\lambda\in\CC$ and $h\in L$,
	\[
		\calB_\lambda\circ\rho(h) = (\rho(h)\otimes h)\circ\calB_\lambda.
	\]
\end{lemma}

\begin{proof}
It suffices to note that the action of $h\in L$ on $V^\pm$ is by Jordan pair automorphisms, i.e., satisfies $h\JTP{x}{y}{z} = \JTP{hx}{hy}{hz}$ for all $x,z\in V^+$, $y\in V^-$. In particular, the pairing $\tau$ is $L$-invariant, hence $\{hc_\alpha\}_{\alpha=1,\ldots,n}$ and $\{h\widehat c_\alpha\}_{\alpha=1,\ldots,n}$ is another pair of dual bases for $V^+$ and $V^-$. The equivariance of $\calB_\lambda$ now follows from standard transformation rules.
\end{proof}

For later use we note the following symmetry property of the Bessel operator.
\begin{proposition}\label{prop:BesselPartialIntegration}
	Let $\lambda\in\CC$, then
	\[
		\int_{V^+}\calB_\lambda f(x)\cdot g(x)\,\td x = \int_{V^+}f(x)\cdot\calB_{2p-\lambda}g(x)\,\td x.
	\]
\end{proposition}

\begin{proof}
Integrating by parts and using Lemma~\ref{lem:basessums}~\eqref{basessums1} we obtain
\begin{align*}
 & \int_{V^+}{\calB_\lambda f(x)\cdot g(x)\,\td x}\\
 ={}& \int_{V^+}{f(x)\cdot\left(
 	\frac{1}{2}\sum_{\alpha,\beta}\frac{\partial^2}{\partial x_\alpha\partial x_\beta}
 	\Big(g(x)\JTP{\widehat{c}_\alpha}{x}{\widehat{c}_\beta}\Big)
 	-\lambda\sum_\alpha\frac{\partial g}{\partial x_\alpha}(x)\,\widehat c_\alpha\right)\td x}\\
 ={}& \int_{V^+}{f(x)\cdot\left(
 	\calB_0g(x) + \sum_{\alpha,\beta}\frac{\partial g}{\partial x_\alpha}(x)
 	\JTP{\widehat{c}_\alpha}{c_\beta}{\widehat{c}_\beta}
 	-\lambda\sum_\alpha\frac{\partial g}{\partial x_\alpha}(x)\,\widehat c_\alpha\right)\td x}\\
 ={}& \int_{V^+}{f(x)\cdot\left(
 	\calB_0g(x) + 2p\sum_\alpha\frac{\partial g}{\partial x_\alpha}(x)\,\widehat{c}_\alpha
 	-\lambda\sum_\alpha\frac{\partial g}{\partial x_\alpha}(x)\,\widehat c_\alpha\right)\td x}\\
 ={}& \int_{V^+}{f(x)\cdot\calB_{2p-\lambda}g(x)\,\td x}.\qedhere
\end{align*}
\end{proof}

\subsection{Restriction to submanifolds}\label{sec:tangency}

We now turn to tangential differential operators. Recall that a differential operator $D$ on $V^+$ is \emph{tangential} to a submanifold $S\subseteq V^+$, if for any smooth function $f\in C^\infty(V^+)$ with $f|_S=0$ we have $(Df)|_S=0$. For a linear subspace $S\subseteq V^+$ this means that $D$ only contains derivatives in the direction of $S$.

For an idempotent $e\in\calV_k$ recall the map $\varphi_e$ from \eqref{eq:diffeo}. The pullback of the Bessel operator is defined by 
\[
	(\varphi_e^*\calB_\lambda f)(x) = \calB_\lambda(f\circ\varphi_e^{-1})(\varphi_e(x)).
\]
We fix a pair of dual bases $\{c_\alpha\}_{\alpha\in I}$ of $V^+$ and $\{\widehat c_\alpha\}_{\alpha\in I}$ of $V^-$ with respect to the trace form $\tau$. Choose these bases compatible with the Peirce decomposition $V^\pm=V_2^\pm\oplus V_1^\pm\oplus V_0^\pm$ with respect to $e$, i.e. $I=I_2\sqcup I_1\sqcup I_0$ with $c_\alpha\in V_\ell^+$ if and only if $\alpha\in I_\ell$, $\ell=0,1,2$. We further write $\nabla$ for the gradient with respect to the trace form, and $\nabla_\ell$ for its projection to $V_\ell^-$, $\ell=0,1,2$.

\begin{theorem}\label{thm:pullbackBessel}
	The pullback of $\calB_\lambda$ along $\varphi_e$ at $x=x_2+x_1\in(V_2^+)^\times\times V_1^+$
	is given by
	\begin{align*}
		(\varphi_e^*\calB_\lambda f)(x) 
		={}&\frac{1}{2}\sum_{\alpha,\beta\in I_2\sqcup I_1}
		\frac{\partial^2f}
				 {\partial x_\alpha\partial x_\beta}(x)
			\JTP{\widehat{c}_\alpha}{x}{\widehat{c}_\beta} +\frac{1}{2}\sum_{\alpha,\beta\in I_1}
		\frac{\partial^2f}
				 {\partial x_\alpha\partial x_\beta}(x)
			\JTP{\widehat{c}_\alpha}{Q_{x_1}x_2^{-1}}{\widehat{c}_\beta}\\
		 &+\lambda\nabla_2f(x)+\lambda\nabla_1f(x)+(\lambda-kd)\B{x_2^{-1}}{x_1}\nabla_0f(x).
	\end{align*}
	In particular, $\calB_\lambda$ is tangential to $\calV_k$ if and only if $\lambda=kd$.
\end{theorem}

\begin{proof}
For simplicity, we write $\calB_\lambda'=\varphi_e^*\calB_\lambda$. We first prove the claimed formula for $x\in V_2^+$ and then use the equivariance of the Bessel operator to determine the general formula. Set $\varphi=\varphi_e$, and recall that for general $x\in V^+$, $\varphi(x)=x+Q_{x_1}x_2^{-1}$ and $\varphi^{-1}(x)=x-Q_{x_1}x_2^{-1}$. Therefore,
\[
	\td\varphi^{\pm1}(x)(u)=u\pm\JTP{x_1}{x_2^{-1}}{u_1} + Q_{x_1}Q_{x_2^{-1}}u_2.
\]
For the following, we assume that $x\in(V_2^+)^\times$. Then, it follows that
\begin{align}\label{eq:derivatives}
	\td\varphi^{\pm1}(x)(u)=u,\qquad
	\td^2\varphi^{\pm1}(x)(u,v)=\pm\JTP{u_1}{x^{-1}}{v_1}.
\end{align}
Write $\calB_\lambda=\calB_0+\lambda\nabla$. The pullback of the gradient $\nabla$ is easily shown to be
\begin{align}\label{eq:gradienttrafo}
	(\varphi^*\nabla)(x)=\nabla(f\circ\varphi^{-1})(\varphi(x)) = \td\varphi(x)^{-*}(\nabla f)(x)
	=(\nabla f)(x).
\end{align}
It remains to determine $\calB'_0$ which is
\[
	\calB'_0f(x) = \frac{1}{2}\sum_{\alpha,\beta\in I}
		\frac{\partial^2(f\circ\varphi^{-1})}
				 {\partial x_\alpha\partial x_\beta}(\varphi(x))
		\JTP{\widehat{c}_\alpha}{\varphi(x)}{\widehat{c}_\beta}.
\]
Due to the chain rule,
\begin{align*}
	(\calB'_0 f)(x)
	&= 	\frac{1}{2}\sum_{\gamma,\delta\in I}\left(\sum_{\alpha,\beta\in I}
				\frac{\partial\varphi_\gamma^{-1}}{\partial x_\alpha}(\varphi(x))
				\frac{\partial\varphi_\delta^{-1}}{\partial x_\beta}(\varphi(x))
				\JTP{\widehat{c}_\alpha}{\varphi(x)}{\widehat{c}_\beta}\right)
				\frac{\partial^2 f}{\partial x_\gamma\partial x_\delta}(x)\\
		&\quad + \frac{1}{2}\sum_{\gamma\in I}\left(\sum_{\alpha,\beta\in I}
			\frac{\partial^2\varphi_\gamma^{-1}}
								 {\partial x_\alpha\partial x_\beta}(\varphi(x))
			\JTP{\widehat{c}_\alpha}{\varphi(x)}{\widehat{c}_\beta} \right)\,
			\frac{\partial f}{\partial x_\gamma}(x),
\end{align*}
where $\varphi^{-1}_\gamma(x) = \tau(\varphi^{-1}(x),\widehat{c}_\gamma)$. Applying \eqref{eq:derivatives} yields
\begin{align*}
	(\calB'_0 f)(x)
	&= \frac{1}{2}\sum_{\gamma,\delta\in I}\JTP{\widehat{c}_\gamma}{x}{\widehat{c}_\delta}
				\frac{\partial^2 f}{\partial x_\gamma\partial x_\delta}(x) - \frac{1}{2}\sum_{\gamma\in I_0}\left(\sum_{\beta\in I_1}
			\JTP{\JTP{x^{-1}}{c_\beta}{\widehat{c}_\gamma}}{x}{\widehat{c}_\beta}\right)\,
			\frac{\partial f}{\partial x_\gamma}(x).
\end{align*}
We determine the $\beta$-sum. According to \eqref{JP16} and the Peirce rules (more precisely, \eqref{JP16} with $(u,v,x,y,z)\widehat=(x^{-1},x,\widehat{c}_\gamma,c_\beta,\widehat{c}_\beta)$):
\[
	\JTP{\JTP{x^{-1}}{c_\beta}{\widehat{c}_\gamma}}{x}{\widehat{c}_\beta}
	=\JTP{\widehat{c}_\gamma}{\JTP{x}{x^{-1}}{c_\beta}}{\widehat{c}_\beta}.
\]
Using $D_{x,x^{-1}}=D_{e,e'}$ and Lemma~\ref{lem:basessums}~\eqref{basessums1}, it follows that the $\beta$-sum evaluates to
\[
	\sum_{\beta\in I_1}\JTP{\widehat{c}_\gamma}{c_\beta}{\widehat{c}_\beta}
	=2p\,\widehat{c}_\gamma - \sum_{\beta\in I_0} D_{\widehat{c}_\beta,c_\beta}\widehat{c}_\gamma.
\]
The last sum only involves terms with $c_\beta\in V_0^+$ since $\JTP{\widehat c_\beta}{c_\beta}{\widehat c_\gamma}=0$ for $c_\beta\in V_2^+$. This sum evaluates to $2p_0\,\widehat c_\gamma$, where $p_0$ is the structure constant defined as in \eqref{eq:pDefinition} but with respect to the subpair $V_0=(V_0^+,V_0^-)$. Since $p-p_0=kd$, we conclude that 
\begin{align*}
	(\calB'_0 f)(x)
	&= \frac{1}{2}\sum_{\gamma,\delta\in I}\JTP{\widehat{c}_\gamma}{x}{\widehat{c}_\delta}
				\frac{\partial^2 f}{\partial x_\gamma\partial x_\delta}(x)
				-kd\sum_{\gamma\in I_0}\widehat{c}_\gamma
				\frac{\partial f}{\partial x_\gamma}(x).
\end{align*}
In combination with the gradient term, we thus obtain
\[
	(\calB'_\lambda f)(x) = \calB_0 f(x) +\lambda\nabla f(x)- kd\nabla_0f(x)\qquad\text{for }x\in(V_2^+)^\times.
\]
We next extend this formula to elements in $(V_2^+)^\times\times V_1^+$. Recall from Lemma~\ref{lem:pullbackaction} that the action of $\overline Q_{[e]}=V_1^+\rtimes L_{[e]}$ on $V^+$ admits a pullback along $\varphi$ which is given by
\begin{align*}
	\xi(h)(x)&=hx,\qquad
	\xi(\B{v}{b})(x)=x - \JTP{v}{b}{x_2}
\end{align*}
for any $h\in L_{[e]}$ and $v\in V_1^+$, $b\in V_2^-$, and $x\in N_e$. 
The equivariance of the Bessel operator $\calB_\lambda$ given by Lemma~\ref{lem:BesselEquivariance} translates to the following equivariance of the pullback Bessel operator $\calB'_\lambda$:
\[
	\calB'_\lambda(f\circ\xi(q)^{-1})(x)=q\cdot(\calB'_\lambda f)(\xi(q)^{-1}x).
\]
In order to avoid confusion, in the following we write $a\in(V_2^+)^\times$ instead of $x$ for the fixed element. For $q=\B{v}{a^{-1}}$ with $v\in V_1^+$, we note that $q^{-1}=\B{-v}{a^{-1}}$, and the action of $q^{-1}$ on element of $V^-$ is given by $\B{a^{-1}}{-v}^{-1}=\B{a^{-1}}{v}$. Since
\[
	\xi(q)^{-1}a=\xi(\B{-v}{a^{-1}})a=a+\JTP{v}{a^{-1}}{a}=a+v,
\]
equivariance yields
\[
	\calB'_\lambda f(a+v)=\B{a^{-1}}{v}\,\calB'_\lambda(f\circ\xi(\B{-v}{a^{-1}}))(a).
\]
For the gradient part of $\calB'_\lambda$, we obtain
\[
	\nabla(f\circ\xi(\B{-v}{a^{-1}}))(a) = \xi(\B{-v}{a^{-1}})^*\nabla f(x+v),
\]
where adjoint $\xi(\B{-v}{a^{-1}})^*$ of $\xi(\B{-v}{a^{-1}})$ with respect to the trace form $\tau$ is given by
\begin{align}\label{eq:adjointaction}
	\xi(\B{-v}{a^{-1}})^*(w) = \begin{cases} 
		\B{a^{-1}}{-v}(w) &,\ w\in V_2^-\oplus V_1^-,\\
		w &,\ w\in V_0^-.
	\end{cases}
\end{align}
The $V_0^+$-grandient is invariant for the action of $\xi(\B{-v}{a^{-1}})$ since $V_0^+$ is fixed under this action. We thus conclude that
\[
	\left(\varphi^*(\lambda\nabla-kd\nabla_0)\right)f(x)=\lambda\nabla_2f(x)+\lambda\nabla_1f(x)
	+(\lambda-kd)\B{a^{-1}}{v}\nabla_0f(x).
\]
It remains to determine $\calB'_0f$ at $a+v$. Since
\[
	\sum_{\alpha,\beta\in I}
		\frac{\partial^2(f\circ T)}
				 {\partial x_\alpha\partial x_\beta}(a)\JTP{\widehat{c}_\alpha}{a}{\widehat{c}_\beta}
	=	\sum_{\alpha,\beta\in I}
		\frac{\partial^2f}
				 {\partial x_\alpha\partial x_\beta}(T(a))\JTP{T^*\widehat{c}_\alpha}{a}{T^*\widehat{c}_\beta}
\]
holds for any linear isomorphism $T\in\GL(V^+)$, it follows that
\begin{equation}\label{eq:besseltrafo}
	\calB'_0f(a+v) = \frac{1}{2}\sum_{\alpha,\beta\in I}
		\frac{\partial^2f}
				 {\partial x_\alpha\partial x_\beta}(a+v)\cdot\B{a^{-1}}{v}
		\JTP{\xi(\B{-v}{a^{-1}})^*\widehat{c}_\alpha}{a}
		{\xi(\B{-v}{a^{-1}})^*\widehat{c}_\beta}.
\end{equation}
If either $\widehat{c}_\alpha$ or $\widehat{c}_\beta$ is in $V_0^-$, the Jordan product in the sum of \eqref{eq:besseltrafo} vanishes due to the Peirce rules since $a\in V_2^+$. Therefore, we may assume that $\alpha,\beta\in I_2\sqcup I_1$, and due to \eqref{eq:adjointaction},
\begin{align*}
	\JTP{\xi(\B{-v}{a^{-1}})^*\widehat{c}_\alpha}{a}
		{\xi(\B{-v}{a^{-1}})^*\widehat{c}_\beta}
	&= \B{a^{-1}}{-v}\JTP{\widehat{c}_\alpha}{\B{-v}{a^{-1}}(a)}{\widehat{c}_\beta}\\
	&=	\B{a^{-1}}{-v}\JTP{\widehat{c}_\alpha}{a+v+Q_va^{-1}}{\widehat{c}_\beta}.
\end{align*}
According to the Peirce rules we thus obtain
\begin{align*}
	\calB'_0f(a+v) &= \frac{1}{2}\sum_{\alpha,\beta\in I_2\sqcup I_1}
		\frac{\partial^2f}
				 {\partial x_\alpha\partial x_\beta}(a+v)
			\JTP{\widehat{c}_\alpha}{a+v}{\widehat{c}_\beta}\\
		&\quad+\frac{1}{2}\sum_{\alpha,\beta\in I_1}
		\frac{\partial^2f}
				 {\partial x_\alpha\partial x_\beta}(a+v)
			\JTP{\widehat{c}_\alpha}{Q_va^{-1}}{\widehat{c}_\beta}.
\end{align*}
Collecting all terms, this completes the proof.
\end{proof}

\begin{theorem}\label{thm:BesselSymmetricOnOrbits}
Let $\lambda=kd$, $0\leq k\leq r-1$, so that $\calB_\lambda$ is tangential to $\calV_k$, and assume that $\calV_k$ carries an $L$-equivariant measure $\td\mu_k$ with corresponding character $\chi_{kd}$. Then $\calB_\lambda$ is symmetric on $L^2(\calV_k,\td\mu_k)$.
\end{theorem}

\begin{proof}
We prove this result using the parametrization of $\calV_k$ by $\varphi=\varphi_e$, see \eqref{prop:LOrbits}. By Proposition~\ref{prop:PullbackMeasure} the $L^2$-inner product associated to the pullback measure $\varphi^*\td\mu_k$ is given by
\[
	(f,g)\mapsto\int_{(V_2^+)^\times\times V_1^+} f(x)\overline{g(x)}\Delta(x)^{kd-p}\,\td\lambda(x),
\]
and by Theorem~\ref{thm:pullbackBessel} the pullback of the Bessel operator $\calB_\lambda$ along $\varphi$ is given by
\begin{multline*}
	(\varphi^*\calB_\lambda)f(x) = \frac{1}{2}\sum_{\alpha,\beta\in I_2\sqcup I_1}\frac{\partial^2f}{\partial x_\alpha\partial x_\beta}(x)\JTP{\widehat{c}_\alpha}{x}{\widehat{c}_\beta}\\
	+\frac{1}{2}\sum_{\alpha,\beta\in I_1}\frac{\partial^2f}{\partial x_\alpha\partial x_\beta}(x)\JTP{\widehat{c}_\alpha}{Q_{x_1}x_2^{-1}}{\widehat{c}_\beta}+\lambda\nabla_2f(x)+\lambda\nabla_1f(x).
\end{multline*}
Recall that the formal adjoint $D^\ad$ of an operator 
\[
	D = \sum_{\alpha,\beta}a_{\alpha\beta}(x)\frac{\partial^2}{\partial x_\alpha\partial x_\beta} + \sum_\alpha a_\alpha(x)\frac{\partial}{\partial x_\alpha}
\]
with coefficient functions $a_\alpha(x)$ and $a_{\alpha\beta}(x)$ is given by
\[
	D^\ad f(x)
		=\frac{1}{\Delta^{kd-p}}\sum_{\alpha,\beta}\frac{\partial^2}{\partial x_\alpha\partial x_\beta}\left(a_{\alpha\beta}\cdot\Delta^{kd-p}\cdot f\right)-\frac{1}{\Delta^{kd-p}}
			\sum_{\alpha}\frac{\partial}{\partial x_\alpha}\left(a_\alpha\cdot\Delta^{kd-p}\cdot f\right).
\]
In case of the Bessel operator we obtain
\begin{align*}
	(\varphi^*\calB_\lambda)^\ad f(x) = 
		&\underbrace{\frac{1}{2\Delta^{kd-p}}\sum_{\alpha,\beta\in I_2}
			\frac{\partial^2}{\partial x_\alpha\partial x_\beta}
			\left(\JTP{\widehat c_\alpha}{x_2}{\widehat c_\beta}\cdot\Delta^{kd-p}\cdot f\right)}_{(1)}\\
		&\underbrace{+\frac{1}{\Delta^{kd-p}}\sum_{\alpha\in I_1,\beta\in I_2}
			\frac{\partial^2}{\partial x_\alpha\partial x_\beta}
			\left(\JTP{\widehat c_\alpha}{x}{\widehat c_\beta}\cdot\Delta^{kd-p}\cdot f\right)}_{(2)}\\
		&\underbrace{+\frac{1}{2\Delta^{kd-p}}\sum_{\alpha,\beta\in I_1}
			\frac{\partial^2}{\partial x_\alpha\partial x_\beta}
			\left(\JTP{\widehat c_\alpha}{x}{\widehat c_\beta}\cdot\Delta^{kd-p}\cdot f\right)}_{(3)}\\
		&\underbrace{+\frac{1}{2\Delta^{kd-p}}\sum_{\alpha,\beta\in I_1}
			\frac{\partial^2}{\partial x_\alpha\partial x_\beta}
			\left(\JTP{\widehat c_\alpha}{Q_{x_1}x_2^{-1}}{\widehat c_\beta}
			\cdot\Delta^{kd-p}\cdot f\right)}_{(4)}\\
		&\underbrace{-\frac{\lambda}{\Delta^{kd-p}}\sum_{\alpha\in I_1\sqcup I_2}
			\frac{\partial}{\partial x_\alpha}
			\left(\widehat c_\alpha\cdot\Delta^{kd-p}\cdot f\right).}_{(5)}
\end{align*}
The Jordan algebra determinant $\Delta$ is independent of $x_1$, and satisfies 
\[
	\frac{1}{\Delta}\frac{\partial}{\partial x_\alpha}(\Delta) = \tau(c_\alpha,x_2^{-1})\qquad
	\text{for }\alpha\in I_2,
\]
hence
\[
	\frac{1}{\Delta^\sigma}\frac{\partial}{\partial x_\alpha}(\Delta^\sigma)
	=\sigma\cdot\tau(c_\alpha,x_2^{-1})
\]
and
\begin{align*}
	\frac{1}{\Delta^\sigma}\frac{\partial^2}{\partial x_\alpha\partial x_\beta}(\Delta^\sigma)
	&=\sigma^2\cdot\tau(c_\alpha,x_2^{-1})\cdot\tau(c_\beta,x_2^{-1})
	+ \sigma\cdot\tau(c_\alpha,\frac{\partial x_2^{-1}}{\partial x_\beta})\\
	&=\sigma^2\cdot\tau(c_\alpha,x_2^{-1})\cdot\tau(c_\beta,x_2^{-1})
	- \sigma\cdot\tau(c_\alpha,Q_{x_2^{-1}}c_\beta),
\end{align*}
which follows from differentiating the relation $Q_{x_2} x_2^{-1} = x_2$. Using Lemma~\ref{lem:basessums}, it follows that
\begin{align*}
	(1) ={}& (kd-p)\,f(x)\sum_{\alpha\in I_2}\JTP{\widehat c_\alpha}{c_\alpha}{x_2^{-1}}+\sum_{\alpha\in I_2}\JTP{\widehat c_\alpha}{c_\alpha}{\nabla_2f}+(kd-p)\JTP{x_2^{-1}}{x_2}{\nabla_2 f}\\
	&+\frac{(kd-p)^2}{2}\,f(x)\cdot\JTP{x_2^{-1}}{x_2}{x_2^{-1}}-\frac{kd-p}{2}\,f(x)\sum_{\alpha\in I_2}\JTP{Q_{x_2^{-1}}c_\alpha}{x_2}{\widehat c_\alpha}\\
	&+\frac{1}{2}\sum_{\alpha,\beta\in I_2}\frac{\partial^2 f}{\partial x_\alpha\partial x_\beta}(x)\JTP{\widehat c_\alpha}{x_2}{\widehat c_\beta}\\
			={}& (kd-p)(kd+p_2-p)\,f(x)\cdot x_2^{-1}+2\,(kd+p_2-p)\nabla_2 f(x)\\
				& +\frac{1}{2}\sum_{\alpha,\beta\in I_2}\frac{\partial^2 f}{\partial x_\alpha\partial x_\beta}(x)\JTP{\widehat c_\alpha}{x_2}{\widehat c_\beta},\\
\intertext{where the relation $\JTP{Q_{x_2^{-1}}c_\alpha}{x_2}{\widehat c_\alpha}=\JTP{x_2^{-1}}{c_\alpha}{\widehat c_\alpha}$ follows from \eqref{JP8},}
	(2) ={}& \sum_{\alpha\in I_1}
				(kd-p)\cdot\JTP{\widehat c_\alpha}{c_\alpha}{x_2^{-1}}\cdot f(x)
				+\sum_{\alpha\in I_1}\JTP{\widehat c_\alpha}{c_\alpha}{\nabla_2 f}\\
				&+\sum_{\beta\in I_2}\JTP{\nabla_1 f}{c_\beta}{\widehat c_\beta}
			+(kd-p)\JTP{\nabla_1 f}{x}{x_2^{-1}}\\
			&+\sum_{\alpha\in I_1,\beta\in I_2}
			\frac{\partial^2 f}{\partial x_\alpha\partial x_\beta}(x)\,
			\JTP{\widehat c_\alpha}{x}{\widehat c_\beta}\\
			={}&2(p-p_2)\,\big((kd-p)\cdot f(x)\cdot x_2^{-1} + \nabla_2 f(x)\big)
			+(kd+p_2-p)\,\nabla_1f(x)\\
			&+(kd-p)\JTP{\nabla_1 f}{x_1}{x_2^{-1}}
			+ \sum_{\alpha\in I_1,\beta\in I_2}
			\frac{\partial^2 f}{\partial x_\alpha\partial x_\beta}(x)\,
		\JTP{\widehat c_\alpha}{x}{\widehat c_\beta},\\[2mm]
	(3) ={}& \sum_{\alpha\in I_1}\JTP{\widehat c_\alpha}{c_\alpha}{\nabla_1 f}+\frac{1}{2}\sum_{\alpha,\beta\in I_1}\frac{\partial^2 f}{\partial x_\alpha\partial x_\beta}(x)
				\JTP{\widehat c_\alpha}{x_2}{\widehat c_\beta}\\
		={}& (2p-p_2-p_0)\nabla_1 f+\frac{1}{2}\sum_{\alpha,\beta\in I_1}\frac{\partial^2 f}{\partial x_\alpha\partial x_\beta}(x)
				\JTP{\widehat c_\alpha}{x_2}{\widehat c_\beta},\\[2mm]
	(4) ={}& \frac{1}{2}\sum_{\alpha,\beta\in I_1}f(x)\cdot
				\JTP{\widehat c_\alpha}{\JTP{c_\alpha}{x_2^{-1}}{c_\beta}}{\widehat c_\beta}
				+\sum_{\alpha\in I_1}
				\JTP{\widehat c_\alpha}{\JTP{c_\alpha}{x_2^{-1}}{x_1}}{\nabla_1 f}\\
				&+\frac{1}{2}\sum_{\alpha,\beta\in I_1}
				 \frac{\partial^2 f}{\partial x_\alpha\partial x_\beta}(x)
				 \JTP{\widehat c_\alpha}{Q_{x_1}x_2^{-1}}{\widehat c_\beta},
\end{align*}
\begin{align*}
	(5) ={}& -\lambda\nabla_2 f(x)-\lambda\nabla_1 f(x)-\lambda(kd-p)\,f(x)\cdot x_2^{-1}.
\end{align*}
Since $\varphi^*\calB_\lambda$ and the measure $\td\mu_k$ are $\overline Q_{[e]}$-equivariant, the adjoint operator $(\varphi^*\calB_\lambda)^\ad$ has the same equivariance property. Now, $\overline Q_{[e]}$ acts transitively on $\Omega_e\times V_1^+$, and therefore it suffices to show that $(\varphi^*\calB_\lambda)^\ad f(x)=\varphi^*\calB_\lambda f(x)$ for $x\in(V_2^+)^\times$. This follows by putting $x_1=0$ in the above formulas, regrouping the various terms, and using Lemma~\ref{lem:basessums2}.
\end{proof}

\subsection{Action on radial functions}

We calculate the action of $\calB_\lambda$ on $(M\cap K)$-invariant functions on the orbits $\calO_k=L\cdot(e_1+\cdots+e_k)$, where $e_1,\ldots,e_r$ is a fixed frame of tripotents in $V^+$. Recall that the map
\[
	(M\cap K)\times C_k^+\to\calO_k, \quad (m,t)\mapsto mb_t
\]
is a diffeomorphism onto an open dense subset of $\calO_k$ where $b_t$ and $C_k^+$ are defined in \eqref{eq:Defbt} and \eqref{eq:DefCk+}.

We first discuss the action of $\calB_\lambda$ on $(M\cap K)$-invariant functions on the open orbit $\calO_r\subseteq V^+$. In this case, there is no restriction on $\lambda$ concerning the tangentiality of $\calB_\lambda$ to $\calO_r$.

\begin{proposition}\label{prop:radialBessel}
	Suppose $f\in C^\infty(\calO_r)$ is $(M\cap K)$-invariant, and put $F(t_1,\ldots,t_r)=f(b_t)$, where 
	$b_t=t_1e_1+\cdots+t_re_r\in\calO_r$. Then
	\begin{align*}
 		\calB_\lambda f(b_t) &= \sum_{i=1}^r{\calB_\lambda^iF(t_1,\ldots,t_r)\,\overline e_i}
	\end{align*}
	with
	\begin{equation}\label{eq:DefinitionBlambdaI}
 	\calB_\lambda^i = t_i\frac{\partial^2}{\partial t_i^2} + \left(\lambda - e - (r-1)d\right)\,\frac{\partial}{\partial t_i} + \frac{1}{2}\sum_{j\neq i}\left(\frac{d_+}{t_i-t_j} + \frac{d_-}{t_i+t_j}\right)\left(t_i\frac{\partial}{\partial t_i}-t_j\frac{\partial}{\partial t_j}\right),
	\end{equation}
	where $d_+,d_-,d,e$ are the structure constants defined in \eqref{eq:structureconstants} and \eqref{eq:dDef}, see also \eqref{eq:Jordanstructureconstants}.
\end{proposition}

\begin{proof}
Let $\{c_\alpha\}_\alpha$ be an orthonormal basis of $V^+$ with respect to the inner product $(x|y)=\tau(x,\overline y)$, which is compatible with the Peirce decomposition associated to the frame $e_1,\ldots,e_r$,
\[
	V^+ = \bigoplus_{1<i\leq j\leq r}(A_{ij}^+\oplus B_{ij}^+)\oplus\bigoplus_{i = 1}^r V^+_{i0}.
\]
Then, $\{\overline c_\alpha\}_\alpha$ is an orthonormal basis of $V^-$ which is compatible with the corresponding decomposition of $V^-$ and dual to $\{c_\alpha\}_\alpha$ with respect to the trace form. Since $\dim A_{ii}^+=1$, we note that $c_\alpha=e_i$ for $c_\alpha\in A_{ii}^+$ due to the appropriate normalization \eqref{eq:traceformonprimitive} of the trace form $\tau$.\\
We note that an $(M\cap K)$-invariant function satisfies for any $X,Y\in\frakm\cap\frakk$ and $a\in V^+$ the relations
\begin{equation}
	\td f(a)(Xa) = 0 \qquad \mbox{and} \qquad \td^2f(a)(Xa,Ya) + \td f(a)(YXa) = 0.\label{eq:derivativeRelation2}
\end{equation}
We use these idenities to determine the first and second derivatives of $f$ at $a$. For any $x,y\in V^+$ the operator $D_{x,\overline y}-D_{y,\overline x}$ is an element of $\frakm\cap\frakk$. Consider $b_t = t_1e_1+\ldots+t_r e_r$ with $t\in C_r^+$, and $X = D_{e_i,\overline\eta} - D_{\eta,\overline e_i}$ with $0\neq\eta\in V_{ij}^+$, more precisely if $j\neq 0$ we assume $\eta\in A_{ij}^+$ or $\eta\in B_{ij}^+$. For convenience, we introduce the following symbols: Let $\epsilon_{ij} = 1+\delta_{ij}$, where $\delta_{ij}$ is Kronecker's delta, and let $\sigma\in\{+1,-1,0\}$ be defined by $Q_e\overline\eta = \sigma\cdot\eta$, where $e=e_1+\cdots+e_r$. Then, $Xb_t$ evaluates to
\begin{align*}
	Xb_t &= \JTP{e_i}{\overline\eta}{b_t} - \JTP{\eta}{\overline e_i}{b_t} = t_j \JTP{e_i}{\overline\eta}{e_j} - t_i\JTP{\eta}{\overline e_i}{e_i} = \epsilon_{ij}t_j Q_e\overline\eta - \epsilon_{ij}t_i\eta\\
	&= -\epsilon_{ij}(t_i-\sigma\,t_j)\eta.
\end{align*}
Therefore, $Xb_t = 0$ if and only if $\eta\in A_{ii}^+ = \RR\,e_i$, and we thus obtain
\begin{align}\label{eq:FirstDerivatives}
	\frac{\partial f}{\partial x_\alpha}(b_t) = \begin{cases}
		\frac{\partial F}{\partial t_i}(t_1,\ldots,t_r)
		& \text{if $c_\alpha\in A_{ii}^+$ for some $i$,} \\
		0 &\text{else.}
	\end{cases}
\end{align}
Now we turn to the discussion of second derivatives. Let $Y = D_{e_k,\overline\zeta} - D_{\zeta,\overline e_k}$ be another element of $\frakm\cap\frakk$ with $\zeta\in V_{k\ell}^+$ satisfying $Q_e\zeta = \sigma'\cdot\zeta$ where $\sigma'\in\{+1,-1,0\}$, then
\begin{align*}
	Yb_t = -\epsilon_{k\ell}(t_k-\sigma'\,t_\ell)\,\zeta.
\end{align*}
Due to \eqref{eq:derivativeRelation2} and \eqref{eq:FirstDerivatives}, $\td^2f(b_t)(Xb_t,Yb_t)$ vanishes if the orthogonal projection of $YXb_t$ onto $\bigoplus_{i = 1}^r A_{ii}^+$ vanishes. Since
\[
	YXb_t = \gamma_{ij}\big(\JTP{e_k}{\overline\zeta}{\eta} - \JTP{\zeta}{\overline e_k}{\eta}\big) \qquad\text{with}\qquad \gamma_{ij} = -\epsilon_{ij}(t_i-\sigma\,t_j),
\]
the Peirce rules imply that $YXb_t$ is an element of $V_{ik}^++ V_{jk}^++ V_{i\ell}^++V_{j\ell}^+$. Therefore, $\td^2f(b_t)(Xb_t,Yb_t)$ vanishes if the orthogonal projection $\pi_{k\ell}(YXb_t)$ of $YXb_t$ onto $A_{kk}^++ A_{\ell\ell}^+ = \RR\,e_k+\RR\,e_\ell$ vanishes. We therefore compute $(YXb_t|e_k)$ and $(YXb_t|e_\ell)$. Since
\begin{align*}
	&\big(\JTP{e_k}{\overline\zeta}{\eta}|e_k\big)
		= \big(\eta|\JTP{\zeta}{\overline e_k}{e_k}\big)
		= \epsilon_{kl}(\eta|\zeta), \\
	&\big(\JTP{\zeta}{\overline e_k}{\eta}|e_k\big)
		= \big(\eta|\JTP{e_k}{\overline\zeta}{e_k}\big )
		= 2\delta_{kl}\sigma'(\eta|\zeta),
\end{align*}
we obtain $(YXb_t|e_k) = \gamma_{ij}(\epsilon_{kl} -2\,\delta_{kl}\sigma')\,(\eta|\zeta)$, and a similar calculation yields $(YXb_t|e_\ell) = \gamma_{ij}\epsilon_{k\ell}(\delta_{k\ell}-\sigma')\,(\eta|\zeta)$. In any case, we conclude that if $\eta\perp\zeta$ with respect to the inner product $(-|-)$, then $\td^2f(b_t)(\eta,\zeta) = 0$. For $\zeta = \eta$, i.e.\ in particular $i = k$, $j = \ell$ and $\sigma' = \sigma$, the orthogonal projection $\pi_{ij}(YXb_t)$ of $YXb_t$ onto $A_{ii}^++ A_{jj}^+$ is given by
\begin{align*}
	\pi_{ij}(YXb_t) = \epsilon_{ij}\gamma_{ij}(\eta|\eta)(e_i-\sigma e_j).
\end{align*}
Now assume $\eta\notin A_{ii}^+$. Then, $\gamma_{ij}\neq0$, and the second derivative $\td^2f(b_t)(\eta,\eta) = -\tfrac{1}{\gamma_{ij}^2}\,\td f(b_t)(YXb_t)$ is given by 
\begin{align*}
	\td^2f(b_t)(\eta,\eta)  
		= \frac{(\eta|\eta)}{t_i-\sigma t_j}\,\left(\frac{\partial F}{\partial t_i}(t)-
			   \sigma\frac{\partial F}{\partial t_j}(t)\right).
\end{align*}
We also evaluate $\JTP{\eta}{\overline b_t}{\eta}$. For $\eta\in V_{i0}^+$, this term vanishes due to the Peirce rules. For $\eta\in A_{ij}^+$ or $\eta\in B_{ij}^+$, we have
\[
	Q_e\overline{\JTP{\eta}{\overline b_t}{\eta}} = 2\,Q_eQ_{\overline\eta}Q_e\overline b_t
	 = 2\,Q_{Q_e\overline\eta}\overline b_t = 2\,Q_{\pm\eta}\overline b_t
	 = \JTP{\eta}{\overline b_t}{\eta}.
\]
Therefore, $\JTP{\eta}{\overline b_t}{\eta}\in A_{ii}^++A_{jj}^+$, and hence $\JTP{\eta}{\overline b_t}{\eta} =\alpha_i\,e_i+\alpha_j e_j$, where the coefficients $\alpha_i,\alpha_j$ are obtained by evalutation of the inner products of $\JTP{\eta}{\overline b_t}{\eta}$ with $e_i$ and $e_j$. This yields 
\[
	\JTP{\eta}{\overline b_t}{\eta}
	= \sigma(\eta|\eta)(t_je_i + t_i e_j).
\]
Collecting everything, the sum over all second derivatives is given by
\begin{align*}
	\sum_{\alpha,\beta} &\frac{\partial^2 f}{\partial x_\alpha\partial x_\beta}(b_t)
			\JTP{\overline c_\alpha}{b_t}{\overline c_\beta} 
		= \sum_{\alpha}\frac{\partial^2f}{\partial x_\alpha^2}(b_t)
			 \JTP{\overline c_\alpha}{b_t}{\overline c_\alpha}\\
		&= \underbrace{\sum_{i=1}^r \sum_{c_\alpha\in A_{ii}^+}(..)}_{(1)}
			+ \underbrace{\sum_{i=1}^r \sum_{c_\alpha\in B_{ii}^+}(..)}_{(2)}
		  + \underbrace{\sum_{1\leq i<j\leq r} \sum_{c_\alpha\in A_{ij}^+}(..)}_{(3)}
		  + \underbrace{\sum_{1\leq i<j\leq r} \sum_{c_\alpha\in B_{ij}^+}(..)}_{(4)},
\end{align*}
and the single terms evaluate to
\begin{align*}
	(1) &= \sum_{i=1}^r 2t_i\,\frac{\partial^2 F}{\partial t_i^2}
					\,\overline e_i, \\
	(2) &= \sum_{i=1}^r \sum_{c_\alpha\in B_{ii}^+}
				 \frac{1}{t_i}\frac{\partial F}{\partial t_i}\,(-2t_i\cdot\overline e_i)
			 = -2\dim B_{ii}^+\cdot\sum_{i=1}^r \frac{\partial F}{\partial t_i}
			 		\,\overline e_i,\\
	(3) &= \sum_{i<j}^r \sum_{c_\alpha\in A_{ij}^+}\frac{1}{t_i-t_j}
					\left(\frac{\partial F}{\partial t_i}-\frac{\partial F}{\partial t_j}\right)
					(t_i\overline e_j + t_j\overline e_i) \\
			&= \sum_{i=1}^r \dim A_{ij}^+\cdot\left((1-r)\frac{\partial F}{\partial t_i} + 
				 \sum_{j\neq i}\frac{1}{t_i-t_j}
					\left(t_i\,\frac{\partial F}{\partial t_i}
								-t_j\frac{\partial F}{\partial t_j}\right)\right)\overline e_i,\\
	(4) &= \sum_{i<j}^r \sum_{e_\alpha\in B_{ij}^+}\frac{1}{t_i+t_j}
					\left(\frac{\partial F}{\partial t_i}+\frac{\partial F}{\partial t_j}\right)
					(-t_i\overline c_j - t_j\overline c_i) \\
			&= \sum_{i=1}^r\dim B_{ij}^+\cdot\left((1-r)\frac{\partial F}{\partial t_i} + 
				 \sum_{j\neq i}\frac{1}{t_i+t_j}
					\left(t_i\,\frac{\partial F}{\partial t_i}
								-t_j\frac{\partial F}{\partial t_j}\right)\right)\overline e_i.
\end{align*}
In combination with \eqref{eq:FirstDerivatives}, this completes the proof.
\end{proof}

As a consequence of Proposition~\ref{prop:radialBessel}, we also obtain the action of $\calB_\lambda$ on $(M\cap K)$-invariant functions on the lower dimensional orbits $\calO_k$. Here, $\lambda$ needs to be fixed such that $\calB_\lambda$ is tangential to $\calO_k$.

\begin{corollary}\label{cor:radialBessel}
	Let $0\leq k\leq r-1$ and $\lambda=kd$. Suppose 
	$f\in C^\infty(\calO_k)$ is $(M\cap K)$-invariant, and put $F(t_1,\ldots,t_k)=f(b_t)$, where 
	$b_t=t_1e_1+\cdots+t_ke_k\in\calO_k$. Then
	\begin{align*}
 		\calB_\lambda f(b_t) &= \sum_{i=1}^r{\calB_k^iF(t_1,\ldots,t_k)\overline e_i}
	\end{align*}
	with
	\begin{align}
 	\calB_k^i &= t_i\frac{\partial^2}{\partial t_i^2}
			+(d-e)\,\frac{\partial}{\partial t_i} +\frac{1}{2}\sum_{\substack{j=1\\j\neq i}}^k\left(\frac{d_+}{t_i-t_j} + \frac{d_-}{t_i+t_j}\right)\left(t_i\frac{\partial}{\partial t_i}-t_j\frac{\partial}{\partial t_j}\right),\label{eq:DefBLambdaK}
	\intertext{for $1\leq i\leq k$, and}
	\calB_k^i &= \frac{d_+-d_-}{2}\sum_{j=1}^k\frac{\partial}{\partial t_j},\label{eq:DefBLambdaK2}
	\end{align}
	for $k<i\leq r$.
\end{corollary}

\begin{proof}
Let $\tilde f$ be the extension of $f$ to an $(M\cap K)$-invariant function on $\calO_r$ defined by $\tilde f(mb_t)=F(t_1,\ldots, t_k)$ for $b_t=t_1e_1+\cdots+t_re_r$ with $t\in C_r^+$. Then, Proposition~\ref{prop:radialBessel} yields a formula for $\calB_\lambda\tilde f$, which simplifies since $\frac{\partial F}{\partial t_j}=0$ for $j>k$. Taking limits $t_j\to0$ for $j>k$ proves our statement.
\end{proof}

\section{$L^2$-models for small representations}\label{sec:L2models}

In this section we give a different proof of the statement in Theorem~\ref{thm:StructureDegPrincipalSeries} that the representation $I(\nu_k)$ is reducible and has an irreducible unitarizable quotient $J(\nu_k)$. We show unitarity by constructing an intrinsic invariant inner product. This inner product is most explicit in the so-called Fourier transformed picture where it simply is the $L^2$-inner product of the $L$-equivariant measure $\td\mu_k$ on the $L$-orbit $\calO_k$.

\subsection{The Fourier-transformed picture}\label{sec:FourierTranform}

The Fourier transform under consideration is the map 
\begin{align*}
	\calF\colon\calS'(V^-)\to\calS'(V^+),\quad\calF f(x) &= \int_{V^-}{e^{i\tau(x,y)}f(y)\,\td y}.
\end{align*}
Here, $\td y$ is any fixed Lebesgue measure on $V^-$, and $\tau$ is the trace form of the Jordan pair $(V^+,V^-)$. For simplicity we normalize Lebesgue measure $\td x$ on $V^+$ such that
$$ \calF^{-1}f(y) = \int_{V^+} e^{-i\tau(x,y)}f(x)\,\td x, \qquad y\in V^-. $$
Since $I(\nu)\subseteq\calS'(V^-)$ we can apply the Fourier transform to the principal series $I(\nu)$, and hence call
\[
	\tilde I(\nu)=\calF(I(\nu))\subseteq\calS'(V^+)
\]
the \emph{Fourier-transformed picture} of $I(\nu)$. We define a representation $\tilde\pi_\nu$ on $\tilde I(\nu)$ by twisting $\pi_\nu$ with the Fourier transform:
\begin{align*}
 \tilde\pi_\nu(g) &= \calF\circ\pi_\nu(g)\circ\calF^{-1}\qquad\text{for }g\in G.
\end{align*}
From Proposition~\ref{prop:noncompactgroupaction} it is easy to deduce the following explicit formulas for the action $\tilde\pi_\nu|_{\overline P}$:

\begin{proposition}\label{prop:FTpicturegroupaction}
	The action of $\exp(a)\in\overline N$ and $h\in L$ on $f\in\tilde I(\nu)$ is given by
	\begin{align*}
 		\tilde\pi_\nu(\exp(a))f(x) &= e^{i\tau(x,a)}f(x),\\
 		\tilde\pi_\nu(h)f(x) &= \chi_{\nu-\frac{p}{2}}(h)f(h^{-1}y).
	\end{align*}
\end{proposition}

Since $N$ does not act by affine linear transformations in $\pi_\nu$, the action of $N$ in the Fourier-transformed picture $\tilde\pi_\nu$ is hard to determine. However, the infinitesimal action $\td\tilde\pi_\nu$ of $\tilde\pi_\nu$ can be expressed in terms of the Bessel operators:

\begin{proposition}\label{prop:FTpicturealgebraaction}
The infinitesimal action $\td\tilde\pi_\nu$ of $\tilde\pi_\nu$ extends to $\calS'(V^+)$ and is given by
\begin{align*}
 \td\tilde\pi_\nu(a)f(x) &= i\tau(x,a)\,f(x), & &a\in\overline\frakn= V^-,\\
 \td\tilde\pi_\nu(T)f(x) &= -\td_{Tx}f(x)+(\tfrac{\nu}{p}-\tfrac{1}{2})\,\Tr_{V^+}(T)\,f(x), 
 & &T\in\frakl,\\
 \td\tilde\pi_\nu(b)f(x) &= \tfrac{1}{i}\,\tau(b,\calB_\lambda f(x)), & &b\in\frakn=V^+,
\end{align*}
where $\lambda=p-2\nu$.
\end{proposition}
\begin{proof}
The first two formulas are easily deduced from Proposition~\ref{prop:FTpicturegroupaction}. For the last formula, a short calculation shows that for fixed $y\in V^-$,
\[
	\calB_\lambda(e^{-i\tau(x,y)})(x) = -e^{-i\tau(x,y)}(Q_yx+i\lambda y).
\]
Using this identity, the formula $\tau(x,Q_yb) = \tau(b,Q_yx)$, and Propositions \ref{prop:liealgaction} and \ref{prop:BesselPartialIntegration}, we obtain
\begin{align*}
 \td\pi_\nu(b)\calF^{-1}f(y)
 &= \left(-\td_{Q_yb}-(2\nu+p)\tau(b,y)\right)
 		\int_{V^+}{e^{-i\tau(x,y)}f(x)\,\td x}\\
 &= \int_{V^+}\left(i\tau(x,Q_yb)-(2\nu+p)\tau(b,y)\right)e^{-i\tau(x,y)}f(x)\,\td x\\
 &= \tfrac{1}{i}\int_{V^+}
 		\tau\left(b,-\left(Q_yx+i(2\nu+p)y\right)e^{-i\tau(x,y)}\right)f(x)\,\td x\\
 &= \tfrac{1}{i}\,\tau\left(b,\int_{V^+}\left(\calB_{2\nu+p}e^{-i\tau(x,y)}\right)(x)
 		f(x)\,\td x\right)\\
 &= \tfrac{1}{i}\,
 		\tau\left(b,\int_{V^+}e^{-i\tau(x,y)}\calB_{p-2\nu}f(x)\,\td x\right)\\ 
 &= \int_{V^+}e^{-i\tau(x,y)}
 		\cdot\tfrac{1}{i}\,\tau\left(b,\calB_{p-2\nu}f(x)\right)\td x,
\end{align*}
and the proof is complete.
\end{proof}

\subsection{The spherical vector}

The $K$-spherical vector $\phi_\nu\in I(\nu)$ was explicitly computed in Proposition~\ref{prop:SphericalVectorNonCptPicture} in the non-compact picture. We now find its Fourier transform
$$ \psi_\nu=\calF\phi_\nu\in\tilde I(\nu)\subseteq\calS'(V^+). $$
Note that since the family of distributions $\phi_\nu\in\calS'(V^-)$ is holomorphic in the parameter $\nu\in\CC$ the same is true for the Fourier transforms $\psi_\nu\in\calS'(V^+)$. We assume $d_+=d_-$ throughout the whole section.

We make use of the following lemma which follows by the same arguments as \cite[Lemma~2.3]{DS99}:

\begin{lemma}\label{lem:KInvariantDistributions}
For any $\nu\in\CC$, $\psi_\nu$ is the unique (up to scalar multiples) $\td\tilde\pi_\nu(\frakk)$-invariant tempered distribution on $V^+$.
\end{lemma}

We now express $\psi_\nu$ in terms of a K-Bessel function on a symmetric cone. For details about K-Bessel functions on symmetric cones we refer the reader to Appendix~\ref{app:KBesselFunctions}.

Fix $0\leq k\leq r$ and let $e=e_1+\cdots+e_k\in V^+$ be the standard rank $k$ idempotent. Then $[e]\subseteq V^+$ is a Jordan algebra and we consider its fixed points $A^{(k)}\subseteq[e]$ under the involution $x\mapsto Q_e\overline x$. The subalgebra $A^{(k)}$ is Euclidean of rank $k$ and $e_1,\ldots,e_k$ is a Jordan frame of $A^{(k)}$, whence the Peirce decomposition of $A^{(k)}$ is given by
\[
	A^{(k)}=\bigoplus_{1\leq i\leq j\leq k}A_{ij}^+.
\]
Denote by $\Omega^{(k)}\subseteq A^{(k)}$ the symmetric cone of $A^{(k)}$. Then by \cite[Theorem VII.1.1]{FK94} the \textit{Gindikin Gamma function of $\Omega^{(k)}$} is given by
$$ \Gamma_{k,d}(\lambda) = (2\pi)^{k(k-1)\frac{d}{4}}\prod_{i=1}^k\Gamma(\lambda-(i-1)\tfrac{d}{2}), $$
where $\Gamma(\lambda)$ denotes the classical gamma function. Hence, $\Gamma_{k,d}(\lambda)$ is holomorphic for $\Re\lambda>(k-1)\frac{d}{2}$. In the case $k=r$ we abbreviate $A=A^{(r)}$ and $\Omega=\Omega^{(r)}$.

Now let first $k=r$ and consider for $\mu\in\CC$ the radial part $K_\mu(t_1,\ldots,t_r)$ of the K-Bessel function $\calK_\mu(x)$ on the symmetric cone $\Omega$ (see Appendix~\ref{app:KBesselFunctions} for its definition). Define an $(M\cap K)$-invariant function $\Psi_\nu$ on the open dense orbit $\calO_r\subseteq V^+$ by
$$ \Psi_\nu(mb_t) = K_{\frac{p-e+1}{2}-\nu}\left((\tfrac{t_1}{2})^2,\ldots,(\tfrac{t_r}{2})^2\right), \qquad m\in M\cap K,\,t\in C_r^+. $$
Then $\Psi_\nu$ defines a measurable function on $V^+$.

\begin{theorem}
Assume that $d_+=d_-$. Then for any $\nu\in\CC$ with $\Re\nu>-\nu_{r-1}$ the function $\Psi_\nu$ belongs to $L^1(V^+)$ and hence defines a tempered distribution $\Psi_\nu\in\calS'(V^+)$. We have
$$ \psi_\nu = \const\times\frac{\Psi_\nu}{\Gamma_{r,d}(\nu+\frac{p}{2})} \qquad \mbox{for $\Re\nu>-\nu_{r-1}$,} $$
where the constant only depends on the structure constants and the normalization of the measures. In particular, the family $\Psi_\nu\in\calS'(V^+)$ extends meromorphically in the parameter $\nu\in\CC$.
\end{theorem}

\begin{proof}
We first show that $\Psi_\nu\in L^1(V^+)$, this implies $\Psi_\nu\in\calS'(V^+)$. Using the integral formula of Proposition~\ref{prop:equivariantmeasure} we find
\begin{align*}
 \int_{V^+}|\Psi_\nu(x)|\,\td x &= \int_{M\cap K}\int_{C_r^+}\Psi_\nu(mb_t)\prod_{i=1}^rt_i^{e+b}\prod_{1\leq i<j\leq r}(t_i^2-t_j^2)^d\,\td t\,\td m\\
 &= \int_{C_r^+}K_{\frac{p-e+1}{2}-\nu}\left((\tfrac{t_1}{2})^2,\ldots,(\tfrac{t_k}{2})^2\right)\prod_{i=1}^rt_i^{e+b}\prod_{1\leq i<j\leq r}(t_i^2-t_j^2)^d\,\td t\\
 &= \const\times\int_{C_r^+}K_{\frac{p-e+1}{2}-\nu}(s_1,\ldots,s_k)\prod_{i=1}^rs_i^{\frac{e+b-1}{2}}\prod_{1\leq i<j\leq k}(s_i-s_j)^d\,\td s.
\end{align*}
By \cite[Theorem VI.2.3]{FK94} the last integral is equal to a constant multiple of
$$ \int_\Omega\calK_{\frac{p-e+1}{2}-\nu}(x)\Delta(x)^{\frac{e+b-1}{2}}\,\td x. $$
This integral is by Lemma~\ref{lem:L1L2KBessel} finite if and only if
$$ \frac{e+b-1}{2}>-1 \qquad \mbox{and} \qquad \Re\nu+e-1+\frac{b-p}{2}>-2-(r-1)\frac{d}{2}, $$
which is satisfied since $b,e\geq0$ and $\Re\nu>-\nu_{r-1}=-\frac{1}{2}(e+\frac{b}{2}+1)$. Next we show that $\Psi_\nu$ is $\td\tilde\pi_\nu(\frakk)$-invariant. Since $\Psi_\nu$ is $(M\cap K)$-invariant by definition, it suffices to show that $\td\tilde\pi_\nu(\frakk\cap(\overline\frakn\oplus\frakn))\Psi_\nu=0$. By Proposition~\ref{prop:FTpicturealgebraaction} this is equivalent to
\[
	\calB_\lambda\Psi_\nu(x)=\Psi_\nu(x)\cdot\overline x
\]
for $\lambda=p-2\nu$. In view of the equivariance property of the Bessel operator (see Lemma~\ref{lem:BesselEquivariance}) we may assume $x=b_t$, $t\in C_r^+$. According to Proposition~\ref{prop:radialBessel} we have
$$ \calB_\lambda\Psi_\nu(b_t) = \sum_{i=1}^r \calB_\lambda^i\left[K_{\frac{p-e+1}{2}-\nu}\left((\tfrac{t_1}{2})^2,\ldots,(\tfrac{t_k}{2})^2\right)\right]\overline{e}_i $$
with $\calB_\lambda^i$ as in \eqref{eq:DefinitionBlambdaI}. Then it remains to show
\begin{equation}\label{eq:DiffEqRadialKBesselRankR}
 \calB_\lambda^i\left[K_{\frac{p-e+1}{2}-\nu}\left((\tfrac{t_1}{2})^2,\ldots,(\tfrac{t_k}{2})^2\right)\right]=t_i\,K_{\frac{p-e+1}{2}-\nu}\left((\tfrac{t_1}{2})^2,\ldots,(\tfrac{t_k}{2})^2\right) \qquad \forall\,1\leq i\leq r.
\end{equation}
Since $d_+=d_-$ the operator $\calB_\lambda^i$ becomes
\begin{align*}
	\calB_\lambda^i = t_i\frac{\partial^2}{\partial t_i^2}
			+\left(\lambda-e-(r-1)d\right)\,\frac{\partial }{\partial t_i}
			+d\hspace{-4mm}\sum_{1\leq j\leq r,\ j\neq i}\frac{t_i}{t_i^2-t_j^2}
				\left(t_i\frac{\partial}{\partial t_i}-t_j\frac{\partial}{\partial t_j}\right).
\end{align*}
Substituting $s_i=(\frac{t_i}{2})^2$ we find that
$$ \frac{1}{t_i}\calB_\lambda^i = s_i\frac{\partial^2}{\partial s_i^2}
			+\frac{1}{2}\left(\lambda-e-(r-1)d+1\right)\,\frac{\partial }{\partial s_i}
			+\frac{d}{2}\sum_{1\leq j\leq r,\ j\neq i}\frac{1}{s_i-s_j}
				\left(s_i\frac{\partial}{\partial s_i}-s_j\frac{\partial}{\partial s_j}\right), $$
so that \eqref{eq:DiffEqRadialKBesselRankR} is equivalent to \eqref{eq:DiffEqKBesselRadial}. Thus $\Psi_\nu\in\calS'(V^+)$ is $\td\tilde\pi_\nu(\frakk)$-invariant, and by Lemma~\ref{lem:KInvariantDistributions} $\Psi_\nu=c(\nu)\psi_\nu$ for some constant $c(\nu)$. We determine $c(\nu)$ by evaluating the identity $c(\nu)\phi_\nu(y)=\calF^{-1}\Psi_\nu(y)$ at $y=0$. As above, using the integral formula of Proposition~\ref{prop:equivariantmeasure} and Lemma~\ref{lem:L1L2KBessel} we find
\begin{equation*}
 c(\nu) = \calF^{-1}\Psi_\nu(0) = \int_{V^+}\Psi_\nu(x)\,\td x = \int_\Omega\calK_{\frac{p-e+1}{2}-\nu}(x)\Delta(x)^{\frac{e+b-1}{2}}\,\td x = \const\times\Gamma_{r,d}(\tfrac{2\nu+p}{2}).\qedhere
\end{equation*}
\end{proof}

Now let $0\leq k\leq r-1$. For $\mu\in\CC$ we consider the radial part $K_\mu(t_1,\ldots,t_k)$ of the K-Bessel function $\calK_\mu(x)$ on the symmetric cone $\Omega^{(k)}\subseteq A^{(k)}$. Define an $(M\cap K)$-invariant function $\Psi_k$ on $\calO_k$ by
$$ \Psi_k(mb_t) = K_{\frac{1}{2}(kd-e+1)}\left((\tfrac{t_1}{2})^2,\ldots,(\tfrac{t_k}{2})^2\right), \qquad m\in M\cap K,\,t\in C_k^+. $$

\begin{theorem}\label{thm:PsiInL1}
Assume that $d_+=d_-$. Then for any $0\leq k\leq r-1$, we have $\Psi_k\in L^1(\calO_k,\td\mu_k)$ and hence $\Psi_k\,\td\mu_k\in\calS'(V^+)$. Moreover,
$$ \psi_{-\nu_k} = \const\times\Psi_k\,\td\mu_k. $$
\end{theorem}

\begin{proof}
We first show that $\Psi_k\in L^1(\calO_k,\td\mu_k)$, this implies $\Psi_k\,\td\mu_k\in\calS'(V^+)$. Using the integral formula of Proposition~\ref{prop:equivariantmeasure} we find
\begin{align*}
 & \int_{\calO_k}|\Psi_k(x)|\,\td\mu_k(x)\\
 & \hspace{1cm} =\int_{C_k^+}K_{\frac{1}{2}(kd-e+1)}\left((\tfrac{t_1}{2})^2,\ldots,(\tfrac{t_k}{2})^2\right)\prod_{i=1}^kt_i^{(r-k+1)d+\frac{b}{2}-1}\prod_{1\leq i<j\leq k}(t_i^2-t_j^2)^d\,\td t\\
 & \hspace{1cm} =\const\times\int_{C_k^+}K_{\frac{1}{2}(kd-e+1)}(s_1,\ldots,s_k)\prod_{i=1}^ks_i^{(r-k+1)\frac{d}{2}+\frac{b}{4}-1}\prod_{1\leq i<j\leq k}(s_i-s_j)^d\,\td s.
\end{align*}
By \cite[Theorem VI.2.3]{FK94} the last integral is equal to a constant multiple of
$$ \int_{\Omega^{(k)}}\calK_{\frac{kd-e+1}{2}}(x)\Delta(x)^{(r-k+1)\frac{d}{2}+\frac{b}{4}-1}\,\td x. $$
This integral is by Lemma~\ref{lem:L1L2KBessel} finite if and only if
$$ (r-k+1)\frac{d}{2}+\frac{b}{4}-1>-1 \qquad \mbox{and} \qquad (r-2k+1)\frac{d}{2}+\frac{e}{2}+\frac{b}{4}-\frac{3}{2}>-2-(k-1)\frac{d}{2}, $$
which is satisfied since $0\leq k\leq r-1$, $d>0$ and $b,e\geq0$. Next we show that $\Psi_k\,\td\mu_k$ is $\td\tilde\pi_\nu(\frakk)$-invariant for $\nu=-\nu_k$. Since $\Psi_k\,\td\mu_k$ is $(M\cap K)$-invariant by definition, it suffices to show that $\td\tilde\pi_\nu(\frakk\cap(\overline\frakn\oplus\frakn))(\Psi_k\,\td\mu_k)=0$. By Proposition~\ref{prop:FTpicturealgebraaction} this is equivalent to
\[
	\calB_\lambda\left[\Psi_k(x)\,\td\mu_k(x)\right]=\left[\Psi_k(x)\,\td\mu_k(x)\right]\cdot\overline x,\qquad x\in\calO_k,
\]
for $\lambda=p+2\nu_k=2p-kd$. Using Proposition~\ref{prop:BesselPartialIntegration} we have by duality for any $\varphi\in\calS(V^+)$:
\begin{align*}
 \langle(\calB_\lambda-\overline x)\left[\Psi_k\,\td\mu_k\right],\varphi\rangle &= \langle\Psi_k\,\td\mu_k,(\calB_{kd}-\overline x)\varphi\rangle = \int_{\calO_k}\Psi_k(x)(\calB_{kd}-\overline x)\varphi(x)\,\td\mu_k(x).
\end{align*}
Now $\calB_{kd}$ is by Theorem~\ref{thm:BesselSymmetricOnOrbits} symmetric on $L^2(\calO_k,\td\mu_k)$ and hence
$$ \langle(\calB_\lambda-\overline x)\left[\Psi_k\,\td\mu_k\right],\varphi\rangle = \int_{\calO_k}(\calB_{kd}-\overline x)\Psi_k(x)\varphi(x)\,\td\mu_k(x). $$
However, since $\Psi_k$ is not compactly supported on $\calO_k$ we have to assure that during the integration by parts (see the proof of Theorem~\ref{thm:BesselSymmetricOnOrbits}) no boundary terms appear. The Bessel operator is of order $2$ and Euler degree $-1$ and therefore it suffices to show that all first partial derivatives of $\Psi_k$ are still contained in $L^1(\calO_k,\td\mu_k)$. This follows if for every $\ell=1,\ldots,k$ we have
\begin{equation}
 \int_{C_k^+} \left|\frac{\partial}{\partial t_\ell}\left[K_{\frac{1}{2}(kd-e+1)}\left((\tfrac{t_1}{2})^2,\ldots,(\tfrac{t_k}{2})^2\right)\right]\right|\prod_{i=1}^kt_i^{(r-k+1)d+\frac{b}{2}-1}\prod_{1\leq i<j\leq k}(t_i^2-t_j^2)^d\,\td t < \infty.\label{eq:L1IntegralFirstDerivativePsik}
\end{equation}
We compute the derivative using \eqref{eq:IntFormulaKBessel2}:
\begin{align*}
 & \frac{\partial}{\partial t_\ell}\left[K_{\frac{1}{2}(kd-e+1)}\left((\tfrac{t_1}{2})^2,\ldots,(\tfrac{t_k}{2})^2\right)\right] = \frac{\partial}{\partial t_\ell}\left[\prod_{i=1}^k(\tfrac{t_i}{2})^{\frac{n}{r}-\lambda}\int_{\Omega^{(k)}} e^{-\frac{1}{2}(b_t|v+v^{-1})}\Delta(v)^{-\lambda}\,\td v\right]\\
 & \hspace{1cm} = \prod_{i=1}^k(\tfrac{t_i}{2})^{\frac{n}{r}-\lambda}\int_{\Omega^{(k)}}\left(\frac{(\frac{n}{r}-\lambda)}{t_\ell}-\frac{1}{2}(e_\ell|v+v^{-1})\right)e^{-\frac{1}{2}(b_t|v+v^{-1})}\Delta(v)^{-\lambda}\,\td v.
\end{align*}
Using $(e_i|v+v^{-1})\geq2$ for all $i=1,\ldots,k$ and $v\in\Omega^{(k)}$, we can estimate
$$ \frac{1}{t_\ell} \leq 2^{1-k}(t_1\cdots t_k)^{-1}\cdot\prod_{\substack{i=1\\i\neq\ell}}^k(t_ie_i|v+v^{-1}) \leq 2^{1-k}(t_1\cdots t_k)^{-1}(b_t|v+v^{-1})^{k-1} $$
and
$$ (e_\ell|v+v^{-1}) \leq 2^{1-k}(t_1\cdots t_k)^{-1}\prod_{i=1}^k(t_ie_i|v+v^{-1}) \leq 2^{1-k}(t_1\cdots t_k)^{-1}(b_t|v+v^{-1})^k. $$
Now $(x^{k-1}+x^k)e^{-\frac{1}{2}x}\leq\const\times e^{-\frac{1}{4}x}$ for $x\geq0$, and hence
$$ \left|\frac{\partial}{\partial t_\ell}\left[K_{\frac{1}{2}(kd-e+1)}\left((\tfrac{t_1}{2})^2,\ldots,(\tfrac{t_k}{2})^2\right)\right]\right| \leq \const\times(t_1\cdots t_k)^{-1}K_{\frac{1}{2}(kd-e+1)}\left((\tfrac{t_1}{4})^2,\ldots,(\tfrac{t_k}{4})^2\right). $$
As above this shows that the integral \eqref{eq:L1IntegralFirstDerivativePsik} is bounded by a constant times
$$ \int_{\Omega^{(k)}}\calK_{\frac{kd-e+1}{2}}(x)\Delta(x)^{(r-k+1)\frac{d}{2}+\frac{b}{4}-\frac{3}{2}}\,\td x, $$
which is by Lemma~\ref{lem:L1L2KBessel} finite if and only if
$$ (r-k+1)\frac{d}{2}+\frac{b}{4}-\frac{3}{2}>-1 \qquad \mbox{and} \qquad (r-2k+1)\frac{d}{2}+\frac{e}{2}+\frac{b}{4}-2>-2-(k-1)\frac{d}{2}. $$
These inequalities hold since $0\leq k\leq r-1$, $d\geq1$ and $b,e\geq0$. This justifies the use of Theorem~\ref{thm:BesselSymmetricOnOrbits} and it remains to show that $\calB_{kd}\Psi_k(x)=\Psi_k(x)\cdot\overline x$ for any $x\in\calO_k$. In view of the equivariance property of the Bessel operator (see Lemma~\ref{lem:BesselEquivariance}) we may assume $x=b_t$, $t\in C_k^+$. According to Corollary~\ref{cor:radialBessel} we have
$$ \calB_{kd}\Psi_k(b_t)=\sum_{i=1}^k \calB_k^i\left[K_{\frac{1}{2}(kd-e+1)}\left((\tfrac{t_1}{2})^2,\ldots,(\tfrac{t_k}{2})^2\right)\right]\overline{e}_i $$
with $\calB_k^i$ as in \eqref{eq:DefBLambdaK} and \eqref{eq:DefBLambdaK2}. Then it remains to show
\begin{align*}\label{eq:radialPDE}
	\calB_k^i\left[K_{\frac{kd-e+1}{2}}\left((\tfrac{t_1}{2})^2,\ldots,(\tfrac{t_k}{2})^2\right)\right] = \begin{cases} t_iK_{\frac{kd-e+1}{2}}\left((\tfrac{t_1}{2})^2,\ldots,(\tfrac{t_k}{2})^2\right) & \mbox{for }1\leq i\leq k,\\0 & \mbox{for }k<i\leq r.\end{cases}
\end{align*}
Due to the assumption $d_+=d_-$, the second condition is automatic, and for $1\leq i\leq k$ the operator $\calB_k^i$ becomes
\begin{align*}
	\calB_k^i=t_i\frac{\partial^2}{\partial t_i^2}
			+(d-e)\,\frac{\partial }{\partial t_i}
			+d\hspace{-4mm}\sum_{1\leq j\leq k,\ j\neq i}\frac{t_i}{t_i^2-t_j^2}
				\left(t_i\frac{\partial}{\partial t_i}-t_j\frac{\partial}{\partial t_j}\right).
\end{align*}
Substituting $s_i=(\frac{t_i}{2})^2$ it is easy to see that this is equivalent to the differential equation \eqref{eq:DiffEqKBesselRadial} for the K-Bessel function. Hence $\Psi_k\,\td\mu_k$ is a $\td\tilde\pi_\nu(\frakk)$-invariant tempered distribution on $V^+$ for $\nu=-\nu_k$ and therefore a constant multiple of $\psi_\nu$.
\end{proof}

\begin{theorem}\label{thm:PsiInL2}
Assume $d_+=d_-$. Then $\Psi_k\in L^2(\calO_k,\td\mu_k)$ for any $0\leq k\leq r-1$.
\end{theorem}

\begin{proof}
A similar computation as in the proof of Theorem~\ref{thm:PsiInL1} shows that
$$ \int_{\calO_k} |\Psi_k(x)|^2\,\td\mu_k(x) = \const\times\int_{\Omega^{(k)}} |\calK_{\frac{kd-e+1}{2}}(x)|^2 \Delta(x)^{(r-k+1)\frac{d}{2}+\frac{b}{4}-1}\,\td x $$
which is by Lemma~\ref{lem:L1L2KBessel} finite if and only if
$$ (r-k+1)\frac{d}{2}+\frac{b}{4}-1>-1 \qquad \mbox{and} \qquad (r-3k+1)\frac{d}{2}+e+\frac{b}{4}-2>-3-(k-1)d. $$
This is satisfied since $0\leq k\leq r-1$, $d>0$ and $b,e\geq0$.
\end{proof}

\subsection{Construction of the $L^2$-model}

By Theorem~\ref{thm:StructureDegPrincipalSeries} we know that for $0\leq k\leq r-1$ the $(\frakg,K)$-module
$$ \td\tilde\pi_{-\nu_k}(\calU(\frakg))\psi_{-\nu_k} \subseteq \tilde I(-\nu_k)_\Kfinite $$
generated by the spherical vector $\psi_{-\nu_k}$ is irreducible and unitarizable. We give a new proof of this fact which also provides an explicit invariant Hermitian form. Let us write $\tilde I_0(-\nu_k)$ for the corresponding subrepresentation of $\tilde I(-\nu_k)$ with underlying $(\frakg,K)$-module $\td\tilde\pi_{-\nu_k}(\calU(\frakg))\psi_{-\nu_k}$.

By Theorem~\ref{thm:PsiInL1} we have $\psi_{-\nu_k}=\const\times\Psi_k\,\td\mu_k$. For $0\leq k\leq r-1$ we let
$$ (M\cap K)_k=\Set{m\in M\cap K}{me_i=e_i\,\forall\,i=1,\ldots,k}, $$
then the map
$$ (M\cap K)/(M\cap K)_k\times C_k^+\to\calO_k, \quad (m,t)\mapsto mb_t, $$
is a diffeomorphism onto an open dense subset of $\calO_k$. In these coordinates the function $\Psi_k$ takes the form
$$ \Psi_k(mb_t) = F(t_1,\ldots,t_k), \qquad m\in M\cap K\,t\in C_k^+, $$
where $F(t_1,\ldots,t_k)=K_{\frac{kd-e+1}{2}}((\frac{t_1}{2})^2,\ldots,(\frac{t_k}{2})^2)$.

\begin{lemma}
Let $0\leq k\leq r-1$, then any $f\in\tilde I_0(-\nu_k)_\Kfinite$ can be written as
\begin{equation}
 f(mb_t) = \sum_\alpha \varphi_\alpha(m)f_\alpha(t), \qquad m\in M\cap K,\,t\in C_k^+,\label{eq:FormWk}
\end{equation}
where $\varphi_\alpha\in C^\infty((M\cap K)/(M\cap K)_k)$ and $f_\alpha$ is of the form
\begin{equation}
 t_{i_1}\cdots t_{i_p}t_{j_1}\cdots t_{j_q}\frac{\partial^p F}{\partial t_{i_1}\cdots\partial t_{i_p}}(t_1,\ldots,t_k)\label{eq:Formfalpha}
\end{equation}
with $i_1,\ldots,i_p,j_1,\ldots,j_q\in\{1,\ldots,k\}$ and $p,q\geq0$.
\end{lemma}

\begin{proof}
By the Poincar\'{e}--Birkhoff--Witt Theorem we have
$$ \calU(\frakg) = \calU(\overline \frakn)\,\calU(\frakl)\,\calU(\frakk) $$
and hence
$$ \tilde I_0(-\nu_k)_\Kfinite = \td\tilde\pi_\nu\big(\calU(\overline\frakn)\,\calU(\frakl)\,\calU(\frakk)\big)(\Psi_k\,\td\mu_k) = \td\tilde\pi_\nu\big(\calU(\overline\frakn)\,\calU(\frakl)\big)(\Psi_k\,\td\mu_k). $$
Since $\overline\frakn$ acts by multiplication with polynomials $\tau(mb_t,u)=\sum_{i=1}^kt_i\tau(me_i,u)$, $u\in V^-$, the action of $\calU(\overline\frakn)$ leaves the space of functions of the form \eqref{eq:FormWk} invariant. Hence, it suffices to show that every $f\in\td\tilde\pi_\nu(\calU(\frakl))(\Psi_k\,\td\mu_k)$ is of the form \eqref{eq:FormWk}. We show this for $\td\tilde\pi_\nu(\calU_n(\frakl))(\Psi_k\,\td\mu_k)$ by induction on $n$, where $\{\calU_n(\frakl)\}_{n\geq0}$ is the natural filtration of $\calU(\frakl)$. For $n=0$ this is clear since $\calU_0(\frakl)=\CC$. We also carry out the proof for $n=1$. Let $T\in\frakl$, then $\td\tilde\pi_\nu(T)=-\td_{Tx}+(\frac{\nu}{p}-\frac{1}{2})\Tr_{V^+}(T)$. We now compute $\td_{Tx}\Psi_k(x)$. Note that since $\Psi_k$ is $(M\cap K)$-invariant, \eqref{eq:FirstDerivatives} implies that for $x=b_t$ we have
$$ \td_{Tx}\Psi_k(x) = \sum_{j=1}^k\tau(Tb_t,\overline{e}_j)\cdot\frac{\partial F}{\partial t_j}(t_1,\ldots,t_k). $$
By the standard transformation rules this implies for $x=mb_t$:
\begin{equation}
 \td_{Tx}\Psi_k(x) = \sum_{i,j=1}^k \tau(Tme_i,m\overline{e}_j)\cdot t_i\frac{\partial F}{\partial t_j}(t_1,\ldots,t_k).\label{eq:LactionOnMinvariants}
\end{equation}
Now, $\tau(Se_i,\overline{e}_j)=0$ for $i\neq j$ and any $S\in\frakl$. In fact, $\frakl$ is generated by the operators $D_{u,v}$, $u\in V^+$, $v\in V^-$, and by \eqref{eq:peircerules}:
$$ \tau(D_{u,v}e_i,\overline{e}_j) = \tau(u,\JTP{v}{e_i}{\overline{e}_j}) = 0. $$
Therefore $\tau(Tme_i,m\overline{e}_j)=\tau(m^{-1}Tme_i,\overline{e}_j)=0$ whenever $i\neq j$ so that
$$ \td_{Tx}\Psi_k(x) = \sum_{i=1}^k \tau(Tme_i,m\overline{e}_i)\cdot t_i\frac{\partial F}{\partial t_i}(t_1,\ldots,t_k) $$
which clearly is of the form \eqref{eq:FormWk} with
$$ \varphi_i(m)=\tau(Tme_i,m\overline e_j) \qquad \mbox{and} \qquad f_i(t_1,\ldots,t_k)=t_i\frac{\partial F}{\partial t_i}(t_1,\ldots,t_k). $$
Now let us complete the induction step. Note that
$$ \td_{Tx}f(x)=\left.\frac{\td}{\td s}\right|_{s=0}f(e^{sT}x). $$
For $x=mb_t$ we write $e^{sT}mb_t=m_sb_{t,s}$, where $m_s\in M\cap K$ and $b_{t,s}$ depend differentiably on $s\in(-\varepsilon,\varepsilon)$ and $m_0=m$, $b_{t,0}=b_t$. For $f$ of the form \eqref{eq:FormWk} we obtain
$$ \td_{Tx}f(mb_t) = \sum_\alpha\left.\frac{\td}{\td s}\right|_{s=0}\varphi_\alpha(m_s)f_\alpha(b_t)+\sum_\alpha\varphi_\alpha(m)\left.\frac{\td}{\td s}\right|_{s=0}f_\alpha(b_{t,s}). $$
Clearly the first summand is again of the form \eqref{eq:FormWk}. To treat the second summand we note that for $f=\Psi_k$ this expression has to agree with \eqref{eq:LactionOnMinvariants} and hence it is of the form
$$ \sum_\alpha\sum_{i=1}^k\varphi_\alpha(m)\varphi_i(m)\cdot t_i\frac{\partial f_\alpha}{\partial t_i}(t_1,\ldots,t_k). $$
Clearly $t_i\frac{\partial}{\partial t_i}$ leaves the space of functions of the form \eqref{eq:Formfalpha} invariant so that this expression is again of the form \eqref{eq:FormWk}. This completes the induction step and the proof.
\end{proof}

\begin{proposition}\label{prop:gKmoduleL2}
For any $0\leq k\leq r-1$ we have $\tilde I_0(-\nu_k)_\Kfinite\subseteq L^2(\calO_k,\td\mu_k)$.
\end{proposition}

\begin{proof}
It remains to show that every function of the form \eqref{eq:FormWk} is contained in $L^2(\calO_k,\td\mu_k)$. In view of the integral formula of Proposition~\ref{prop:equivariantmeasure} this amounts to showing that
$$ \int_{C_k^+} \left|t_{i_1}\cdots t_{i_p}t_{j_1}\cdots t_{j_q}\frac{\partial^p F}{\partial t_{i_1}\cdots t_{i_p}}(t_1,\ldots,t_k)\right|^2\prod_{i=1}^kt_i^{(r-k+1)d+\frac{b}{2}-1}\prod_{1\leq i<j\leq k}(t_i^2-t_j^2)^d\,\td t < \infty $$
for all $i_1,\ldots,i_p,j_1,\ldots,j_q\in\{1,\ldots,k\}$ with $p,q\geq0$. By the chain rule
$$ t_{i_1}\cdots t_{i_p}t_{j_1}\cdots t_{j_q}\frac{\partial^p F}{\partial t_{i_1}\cdots t_{i_p}}(t_1,\ldots,t_k) = 2^{-p}t_{i_1}^2\cdots t_{i_p}^2t_{j_1}\cdots t_{j_q}\frac{\partial^p K_\lambda}{\partial s_{i_1}\cdots s_{i_p}}\left((\tfrac{t_1}{2})^2,\ldots,(\tfrac{t_k}{2})^2\right) $$
with $\lambda=\frac{kd-e+1}{2}$. For $p=q=0$ we showed in Theorem~\ref{thm:PsiInL2} that
$$ \int_{C_k^+} \left|K_\lambda\left((\tfrac{t_1}{2})^2,\ldots,(\tfrac{t_k}{2})^2\right)\right|^2\prod_{i=1}^kt_i^{(r-k+1)d+\frac{b}{2}-1}\prod_{1\leq i<j\leq k}(t_i^2-t_j^2)^d\,\td t < \infty. $$
Now let $s_i=(\frac{t_i}{2})^2$ and consider $s_i\frac{\partial K_\lambda}{\partial s_i}(s_1,\ldots,s_k)$. By \eqref{eq:IntFormulaKBessel1} we have
$$ s_i\frac{\partial K_\lambda}{\partial s_i}(s_1,\ldots,s_k) = -\int_{\Omega^{(k)}} (s_ie_i|v^{-1})e^{-\tr(v)-(b_s|v^{-1})}\Delta(v)^{-\lambda}\,\td v $$
so that
$$ \left|s_{i_1}\cdots s_{i_p}\frac{\partial^p K_\lambda}{\partial s_{i_1}\cdots\partial s_{i_p}}(s_1,\ldots,s_k)\right| = \int_{\Omega^{(k)}}(s_{i_1}e_{i_1}|v^{-1})\cdots(s_{i_p}e_{i_p}|v^{-1})e^{-\tr(v)-(b_s|v^{-1})}\Delta(v)^{-\lambda}\,\td v. $$
Now for any $u\in\Omega^{(k)}$ we have $(u|e_i)\geq0$ for all $i=1,\ldots,k$ and hence
$$ (s_{i_1}e_{i_1}|v^{-1})\cdots(s_{i_p}e_{i_p}|v^{-1}) \leq (b_s|v^{-1})^p \qquad \forall\,v\in\Omega^{(k)}. $$
Let $C>0$ such that $x^pe^{-x}\leq Ce^{-\frac{1}{4}x}$ for $x\geq0$, then
$$ \left|s_{i_1}\cdots s_{i_p}\frac{\partial^p K_\lambda}{\partial s_{i_1}\cdots\partial s_{i_p}}(s_1,\ldots,s_k)\right| \leq C\int_{\Omega^{(k)}} e^{-\tr(v)-\frac{1}{4}(b_s|v^{-1})}\Delta(v)^{-\lambda}\,\td v = C\cdot K_\lambda(\tfrac{1}{4}s_1,\ldots,\tfrac{1}{4}s_k). $$
Finally consider $t_{j_1}\cdots t_{j_q}K_\lambda\left((\tfrac{t_1}{2})^2,\ldots,(\tfrac{t_k}{2})^2\right)$. Using \eqref{eq:IntFormulaKBessel2} we have
$$ t_{j_1}\cdots t_{j_q}K_\lambda\left((\tfrac{t_1}{2})^2,\ldots,(\tfrac{t_k}{2})^2\right) = t_{j_1}\cdots t_{j_q}\int_{\Omega^{(k)}} e^{-\frac{1}{2}(b_t|v+v^{-1})}\Delta(v)^{-\lambda}\,\td v. $$
Since $(e_i|v+v^{-1})\geq2$ for all $v\in\Omega^{(k)}$ we find
$$ t_{j_1}\cdots t_{j_q} \leq (t_{j_1}e_{j_1}|v+v^{-1})\cdots(t_{j_q}e_{j_q}|v+v^{-1}) \leq (b_t|v+v^{-1})^q. $$
Let $C'>0$ such that $x^qe^{-\frac{1}{2}x}\leq C'e^{-\frac{1}{4}x}$ for $x\geq0$, then
$$ \left|t_{j_1}\cdots t_{j_q}K_\lambda\left((\tfrac{t_1}{2})^2,\ldots,(\tfrac{t_k}{2})^2\right)\right| \leq C'\int_{\Omega^{(k)}} e^{-\frac{1}{4}(b_t|v+v^{-1})}\Delta(v)^{-\lambda}\,\td v = C'\cdot K_\lambda\left((\tfrac{t_1}{4})^2,\ldots,(\tfrac{t_k}{4})^2\right). $$
Combining both arguments we obtain
$$ \left|t_{i_1}^2\cdots t_{i_p}^2t_{j_1}\cdots t_{j_q}\frac{\partial^p K_\lambda}{\partial s_{i_1}\cdots s_{i_p}}\left((\tfrac{t_1}{2})^2,\ldots,(\tfrac{t_k}{2})^2\right)\right| \leq \const\times K_\lambda\left((\tfrac{t_1}{8})^2,\ldots,(\tfrac{t_k}{8})^2\right). $$
The same computation as in the proof of Theorem~\ref{thm:PsiInL2} (substituting $s_i=(\frac{t_i}{8})^2$ instead of $s_i=(\frac{t_i}{2})^2$) shows that
$$ \int_{C_k^+} \left|K_\lambda\left((\tfrac{t_1}{8})^2,\ldots,(\tfrac{t_k}{8})^2\right)\right|^2\prod_{i=1}^kt_i^{(r-k+1)d+\frac{b}{2}-1}\prod_{1\leq i<j\leq k}(t_i^2-t_j^2)^d\,\td t < \infty $$
and the proof is complete.
\end{proof}

\begin{theorem}\label{thm:UniRepOnL2}
Assume $d_+=d_-$. Then $\tilde I_0(-\nu_k)\subseteq L^2(\calO_k,\td\mu_k)$ and the $L^2$-inner product is a $\tilde\pi_{-\nu_k}$-invariant Hermitian form. The representation $(\tilde\pi_{-\nu_k},\tilde I_0(-\nu_k))$ integrates to an irreducible unitary representation of $G$ on $L^2(\calO_k,\td\mu_k)$.
\end{theorem}

\begin{proof}
This is the same argument as in \cite[Theorem 2.30]{HKM14}. By Proposition~\ref{prop:gKmoduleL2} we have $\tilde I_0(-\nu_k)_\Kfinite\subseteq L^2(\calO_k,\td\mu_k)$, and by Theorem~\ref{thm:BesselSymmetricOnOrbits} and Proposition~\ref{prop:FTpicturealgebraaction} the $L^2$-inner product is an invariant Hermitian form on $\tilde I_0(-\nu_k)_\Kfinite$. Hence, the $(\frakg,K)$-module $\tilde I_0(-\nu_k)_\Kfinite$ integrates to a unitary representation $(\tau,\calH)$ on a Hilbert space $\calH\subseteq L^2(\calO_k,\td\mu_k)$. Since the Lie algebra actions of $\tau$ and $\tilde\pi_{-\nu_k}$ agree on the Lie algebra of $\overline P$, the group actions $\tau$ and $\tilde\pi_{-\nu_k}$ agree on $\calH$. But in view of Proposition~\ref{prop:FTpicturegroupaction} the subgroup $\overline P$ acts by Mackey theory irreducibly on $L^2(\calO_k,\td\mu_k)$ and hence $\calH=L^2(\calO_k,\td\mu_k)$ and $\tau$ is irreducible as $G$-representation.
\end{proof}

\begin{remark}\label{rem:MinKtypeHermitianAndOpq}
Let us comment on the two cases excluded in Theorem~\ref{thm:UniRepOnL2}.
\begin{enumerate}
\item In the case where $G$ is Hermitian, the K-Bessel function has to be replaced by the exponential function
$$ \Psi_k(mb_t) = e^{-(t_1+\cdots+t_k)}, \qquad m\in M\cap K,t\in C_k^+, $$
then one can show that $\Psi_k\,\td\mu_k$ transforms under $\td\tilde\pi_\nu(\frakk)$ by a unitary character of $\frakk$ for $\nu=-\nu_k$. Note that here the Lie algebra $\frakk$ has a one-dimensional center. In the same way as above one then shows that $\tilde I_0(-\nu_k)$ extends to an irreducible unitary representation of $G$ on $L^2(\calO_k,\td\mu_k)$ (see e.g. \cite[Section 2.1]{Moe13}). We remark that in this case the degenerate principal series has to be formed by inducing from a possibly non-trivial unitary character of $M$.
\item If $G=SO_0(p,q)$, $p\leq q$, then $r=2$ and the K-Bessel function $\Psi_1$ is essentially a classical K-Bessel function. In \cite{HKM14} it is shown that, after modifying the parameter of the Bessel function, $\Psi_1\,\td\mu_1$ generates a finite-dimensional $\frakk$-representation in $\td\tilde\pi_\nu$, $\nu=-\nu_k$, if and only if $p+q$ is even. This $\frakk$-representation is isomorphic to the representation $\calH^{\frac{q-p}{2}}(\RR^p)$ on spherical haromonics of degree $\frac{q-p}{2}$ in $p$ variables, and hence non-trivial for $p\neq q$. Also in this case the same methods as above show that the corresponding $(\frakg,K)$-module integrates to an irreducible unitary representation of $G$ on $L^2(\calO_1,\td\mu_1)$ (see \cite{HKM14} for details). We note that here one has to form the degenerate principal series by inducing from a non-trivial unitary character of $M$ if and only if $p$ and $q$ are both odd and $p-q\equiv2\mod4$.
\end{enumerate}
\end{remark}

\begin{corollary}\label{cor:MinRep}
Assume $d_+=d_-$. If $e=0$ or $V$ is complex, and $\frakg_\CC$ is not a Lie algebra of type $A$, then $J(\nu_1)$ is the unique minimal representation of $G$ (in the sense of \cite[Definition 2.16]{HKM14}).
\end{corollary}

\begin{proof}
For $P$ and $\overline P$ conjugate this is shown in \cite[Corollary 2.32]{HKM14}. For $P$ and $\overline P$ not conjugate we showed in \cite[Theorem~5.3~(2)]{MS12} that the associated variety of $J(\nu_1)$ is the minimal nilpotent $K_\CC$-orbit. Since the Joseph ideal is the unique completely prime ideal in $\calU(\frakg)$ with associated variety the minimal nilpotent $K_\CC$-orbit, it remains to show that the annihilator ideal of $J(\nu_1)$ is completely prime. We employ the same argument as in \cite[Theorem 2.18]{HKM14} using the explicit $L^2$-model. The Lie algebra acts by regular differential operators on the irreducible variety $\calO_1=\calV_1$ and this induces an algebra homomorphism from $\calU(\frakg)$ to the algebra of regular differential operators on $\calV_1$. The latter algebra has no zero-divisors which implies that the kernel of this homomorphism (which is the annihilator of $J(\nu_1)$) is completely prime. This finishes the proof.
\end{proof}

\begin{remark}
If we include the cases of type $A$ or $D_2$ (for which the same results hold), then we obtain $L^2$-models for the minimal representations of the groups
\begin{align*}
 & Sp(n,\RR),\,Sp(n,\CC),\,SO^*(4n),\,SO(p,q)\,(p+q\mbox{ even}),\,SO(n,\CC)\\
 & E_{6(6)},\,E_6(\CC),\,E_{7(7)},\,E_{7(-25)},\,E_7(\CC).
\end{align*}
This is a complete list of all groups having a maximal parabolic subalgebra with abelian nilradical, which admit a minimal representation (see \cite{HKM14}).
\end{remark}

\subsection{The non-standard intertwining operator}

For any unitarizable quotient $J(\nu)=I(\nu)/I_0(\nu)$ in a (degenerate) principal series representation $I(\nu)$ there exists an intertwining operator $T:I(\nu)\to I(-\nu)$ with kernel $I_0(\nu)$ and the property that the invariant inner product on $J(\nu)$ is given by
$$ J(\nu)\times J(\nu)\to\CC, \quad (f,g)\mapsto\int_{V^-} Tf(x)\overline{g(x)}\,\td x. $$
In many cases these operators can be obtained from standard families of intertwining operators such as the Knapp--Stein intertwiners. However, as observed in \cite{MS12} in the case where $P$ and $\overline P$ are not conjugate such families do not exist. Still, unitarizable quotients can occur in this setting, and hence the corresponding intertwiners cannot be obtained from standard families by regularization. In \cite{MS12} we constructed non-standard intertwining operators between $I(\nu_k)$ and $I(-\nu_k)$ for $0\leq k\leq r-1$ in the case where $P$ and $\overline P$ are not conjugate, using algebraic methods. Here we provide a geometric version of these non-standard intertwiners in the Fourier transformed realization $\tilde I(\nu)$.

\begin{theorem}\label{thm:Intertwiner}
For any $0\leq k\leq r-1$ the map
$$ T_k:\tilde I(\nu_k) \to \tilde I(-\nu_k), \quad f\mapsto f|_{\calO_k}\,\td\mu_k $$
is an intertwining operator $\tilde\pi_{\nu_k}\to\tilde\pi_{-\nu_k}$. Its kernel is the subrepresentation
$$ \tilde I_0(\nu_k) = \Set{f\in\tilde I(\nu_k)}{f|_{\calO_k}=0} \subseteq \tilde I(\nu_k) $$
and its image is the subrepresentation $\tilde I_0(-\nu_k)\subseteq\tilde I(-\nu_k)$.
\end{theorem}

\begin{proof}
We first show that the spherical vector $\psi_{\nu_k}\in\tilde I(\nu_k)$ is mapped to the spherical vector $\psi_{-\nu_k}\in\tilde I(-\nu_k)$. Recall that $\psi_{\nu_k}=\Psi_{\nu_k}$ is given by
$$ \Psi_{\nu_k}(mb_t) = K_{\frac{p-e+1}{2}-\nu}(t_1,\ldots,t_r), \qquad m\in M\cap K,\,t\in C_r^+. $$
Due to a theorem of Clerc~\cite[Theorem 4.1]{Cle88} (see also \cite[Proposition 3.10]{Moe13}) the restriction of $K_\mu(t_1,\ldots,t_r)$ to $t_{k+1}=\ldots=t_r$ is defined and gives the K-Bessel function of $k$ variables with the same parameter:
$$ K_\mu(t_1,\ldots,t_k,0,\ldots,0) = \const\times K_\mu(t_1,\ldots,t_k) $$
whenever $\Re\mu<1+k\frac{d}{2}$. For $\nu=\nu_k$ we have $\mu=\frac{p-e+1}{2}-\nu_k=\frac{kd-e+1}{2}$ which is less than $1+k\frac{d}{2}$ since $e\geq0$. Hence
$$ T_k\psi_{\nu_k} = \const\times\psi_{-\nu_k}. $$
Now, $\tilde I(\nu_k)_\Kfinite$ is generated by $\psi_{\nu_k}$. Further, the map $T_k$ is $\frakg$-intertwining by Proposition~\ref{prop:FTpicturealgebraaction} and the fact that the Bessel operator is symmetric with respect to the measure $\td\mu_k$ (see Theorem~\ref{thm:BesselSymmetricOnOrbits}). Hence, $T_k$ is defined on the whole $\tilde I(\nu_k)_\Kfinite$ and its image is $\tilde I_0(-\nu_k)_\Kfinite=\td\tilde\pi_{-\nu_k}(\calU(\frakg))\psi_{-\nu_k}$. That its kernel is equal to $\tilde I_0(\nu_k)_\Kfinite$ is clear by the definition of $T_k$. Finally, $T_k$ extends to $\tilde I(\nu_k)$ by the Casselman--Wallach Globalization Theorem.
\end{proof}

\begin{remark}
In the case where $P$ and $\overline P$ are conjugate, the intertwiner $T_k$ simply is a regularization of the standard family of Knapp--Stein intertwining operators as observed in \cite{BSZ06}.
\end{remark}

\appendix

\section{K-Bessel functions on symmetric cones}\label{app:KBesselFunctions}

We recall the definition and the differential equation of the K-Bessel function on a symmetric cone and prove some integrability results. For more details we refer to \cite{Dib90}, \cite[XVI.\,\S\,3]{FK94} or \cite{Moe13}.

Let $A$ be a Euclidean Jordan algebra of dimension $n$ and rank $r$ (see \cite{FK94} for details). Denote by $\tr$ and $\Delta$ its Jordan algebra trace and determinant, and let $(x|y)=\tr(xy)$ be the trace form. Let $\Omega$ by the symmetric cone in $A$ which is the interior of the cone of squares, or equivalently the connected component of the invertible elements containing the identity element $e$.

For $\lambda\in\CC$ the K-Bessel function $\calK_\lambda$ on $\Omega$ is defined by
\begin{equation}
 \calK_\lambda(x) = \int_\Omega e^{-\tr(u^{-1})-(x|u)}\Delta(u)^{\lambda-\frac{2n}{r}}\,\td u = \int_\Omega e^{-\tr(v)-(x|v^{-1})}\Delta(v)^{-\lambda}\,\td v, \qquad x\in\Omega.\label{eq:IntFormulaKBessel1}
\end{equation}
Evaluated at a square the K-Bessel function can be expressed as
\begin{equation}
 \calK_\lambda(x^2) = \Delta(x)^{\frac{n}{r}-\lambda}\int_\Omega e^{-(x|u+u^{-1})}\Delta(u)^{\lambda-\frac{2n}{r}}\,\td u = \Delta(x)^{\frac{n}{r}-\lambda}\int_\Omega e^{-(x|v+v^{-1})}\Delta(v)^{-\lambda}\,\td v.\label{eq:IntFormulaKBessel2}
\end{equation}
All integrals converge for any $\lambda\in\CC$ and $x\in\Omega$ and we have $\calK_\lambda(x)>0$. If $\calB_\lambda$ is the Bessel operator on $A$ then $\calK_\lambda$ solves the differential equation (see e.g \cite[Proposition~7.2]{Dib90})
$$ \calB_\lambda\calK_\lambda(x) = \calK_\lambda(x)\cdot e, $$
where $e$ is the unit element of $A$. Choose a Jordan frame $e_1,\ldots,e_r\in A$ such that $e=e_1+\cdots+e_r$ and let
$$ A = \bigoplus_{1\leq i\leq j\leq r} A_{ij} $$
be the corresponding Peirce decomposition. Write $K_\lambda$ for the radial part of $\calK_\lambda$, i.e.,
\[
	K_\lambda(t_1,\ldots,t_r)=\calK_\lambda(t_1e_1+\cdots+t_re_r), \qquad t_1,\ldots,t_r>0.
\]
According to Proposition~\ref{prop:radialBessel} (cf.\ \cite[XV.2.8]{FK94}), $K_\lambda$ satisfies the differential equation
\begin{equation}\label{eq:DiffEqKBesselRadial}
	\calB_\lambda^iK_\lambda(t) = K_\lambda(t) \qquad \forall\,1\leq i\leq r,
\end{equation}
where
\[
	\calB_\lambda^i=t_i\frac{\partial^2}{\partial t_i^2}
			+\left(\lambda-(r-1)\frac{d}{2}\right)\,\frac{\partial}{\partial t_i}\\
			+\frac{d}{2}\sum_{j\neq i}\frac{1}{t_i-t_j}
				\left(t_i\frac{\partial}{\partial t_i}-t_j\frac{\partial}{\partial t_j}\right)
\]
and $d=\dim A_{ij}$ for any $i<j$.

\begin{lemma}\label{lem:L1L2KBessel}
We have
$$ \int_\Omega |\calK_\lambda(x)|^2 \Delta(x)^\mu\,\td x < \infty \qquad \Leftrightarrow \qquad \begin{cases}\mu>-1\qquad\mbox{and}\\\mu-2\lambda>-3-(r-1)d,\end{cases} $$
and
$$ \int_\Omega \calK_\lambda(x) \Delta(x)^\mu\,\td x < \infty \qquad \Leftrightarrow \qquad \begin{cases}\mu>-1\qquad\mbox{and}\\\mu-\lambda>-2-(r-1)\frac{d}{2},\end{cases} $$
where $dx$ is a Lebesgue measure on the open set $\Omega\subseteq A$. The latter integral evaluates to
$$ \int_\Omega\calK_\lambda(x)\Delta(x)^\mu\,\td x = \Gamma_{r,d}(\mu+\tfrac{n}{r})\Gamma_{r,d}(\mu-\lambda+\tfrac{2n}{r}), $$
where $\Gamma_{r,d}$ denotes the Gamma function of the symmetric cone $\Omega$.
\end{lemma}

\begin{proof}
We make use of the following integral formula (see \cite[Proposition VII.1.2]{FK94}):
$$ \int_\Omega e^{-(x|y)}\Delta(x)^{\lambda-\frac{n}{r}}\,\td x = \Gamma_{r,d}(\lambda)\Delta(y)^{-\lambda}, $$
where the integral converges (absolutely) if and only if $\lambda>(r-1)\frac{d}{2}$. We further need the identites (see \cite[Proposition III.4.2 and Lemma X.4.4]{FK94})
$$ \Delta(u^{-1}+v^{-1}) = \Delta(u+v)\Delta(u)^{-1}\Delta(v)^{-1}, \qquad \Delta(P(x)y)=\Delta(x)^2\Delta(y). $$
Then we have
\begin{align*}
 & \int_\Omega |\calK_\lambda(x)|^2 \Delta(x)^\mu\,\td x\\
 ={}& \int_\Omega\int_\Omega\int_\Omega e^{-\tr(u)-\tr(v)-(x|u^{-1})-(x|v^{-1})}\Delta(u)^{-\lambda}\Delta(v)^{-\lambda}\Delta(x)^\mu\,\td u\,\td v\,\td x\\
 ={}& \Gamma_{r,d}(\mu+\tfrac{n}{r})\int_\Omega\int_\Omega e^{-\tr(u)-\tr(v)}\Delta(u^{-1}+v^{-1})^{-\mu-\frac{n}{r}}\Delta(u)^{-\lambda}\Delta(v)^{-\lambda}\,\td u\,\td v\\
 ={}& \Gamma_{r,d}(\mu+\tfrac{n}{r})\int_\Omega\int_\Omega e^{-\tr(u)-\tr(v)}\Delta(u+v)^{-\mu-\frac{n}{r}}\Delta(u)^{\mu-\lambda+\frac{n}{r}}\Delta(v)^{\mu-\lambda+\frac{n}{r}}\,\td u\,\td v\\
 ={}& \Gamma_{r,d}(\mu+\tfrac{n}{r})\int_\Omega\int_\Omega e^{-\tr(u)-\tr(v)}\Delta(e+P(u^{\frac{1}{2}})^{-1}v)^{-\mu-\frac{n}{r}}\Delta(u)^{-\lambda}\Delta(v)^{\mu-\lambda+\frac{n}{r}}\,\td u\,\td v\\
 ={}& \Gamma_{r,d}(\mu+\tfrac{n}{r})\int_\Omega\int_\Omega e^{-\tr(u)-(u|v')}\Delta(e+v')^{-\mu-\frac{n}{r}}\Delta(u)^{\mu-2\lambda+\frac{2n}{r}}\Delta(v')^{\mu-\lambda+\frac{n}{r}}\,\td u\,\td v'\\
 ={}& \Gamma_{r,d}(\mu+\tfrac{n}{r})\Gamma_{r,d}(\mu-2\lambda+\tfrac{3n}{r})\int_\Omega \Delta(e+v')^{2\lambda-2\mu-\frac{4n}{r}}\Delta(v')^{\mu-\lambda+\frac{n}{r}}\,\td v',
\end{align*}
where we have substituted $v=P(u^{\frac{1}{2}})v'$ with $\td v=\Delta(u)^{\frac{n}{r}}\,\td v'$. Note that we have used \begin{equation}
 \mu+\tfrac{n}{r} > (r-1)\tfrac{d}{2} \qquad \mbox{and} \qquad \mu-2\lambda+\tfrac{3n}{r} > (r-1)\tfrac{d}{2}.\label{eq:L2Condition1}
\end{equation}
The final integral is finite if and only if
\begin{equation}
 \mu-\lambda+\tfrac{2n}{r}>(r-1)\tfrac{d}{2} \qquad \mbox{and} \qquad \lambda-\mu-\tfrac{3n}{r}<-1-(r-1)d.\label{eq:L2Condition2}
\end{equation}
Note that \eqref{eq:L2Condition1} is equivalent to $\mu>-1$ and $\mu-2\lambda>-3-(r-1)d$. Further, \eqref{eq:L2Condition2} is implied by \eqref{eq:L2Condition1} which proves the first claim. For the second claim we compute similarly
\begin{align*}
 \int_\Omega \calK_\lambda(x) \Delta(x)^\mu\,\td x &= \int_\Omega\int_\Omega e^{-\tr(u)-(x|u^{-1})}\Delta(u)^{-\lambda}\Delta(x)^\mu\,\td u\,\td x\\
 &= \Gamma_{r,d}(\mu+\tfrac{n}{r})\int_\Omega e^{-\tr(u)}\Delta(u)^{\mu-\lambda+\frac{n}{r}}\,\td u
\end{align*}
which is finite if and only if
$$ \mu+\tfrac{n}{r}>(r-1)\tfrac{d}{2} \qquad \mbox{and} \qquad \mu-\lambda+\tfrac{2n}{r}>(r-1)\tfrac{d}{2}. $$
These conditions are equivalent to $\mu>-1$ and $\mu-\lambda>-2-(r-1)\frac{d}{2}$ which shows the second claim. The last claim follows by using the integral formula \cite[Proposition VII.1.2]{FK94} once more.
\end{proof}

\providecommand{\bysame}{\leavevmode\hbox to3em{\hrulefill}\thinspace}
\providecommand{\MR}{\relax\ifhmode\unskip\space\fi MR }
\providecommand{\MRhref}[2]{%
  \href{http://www.ams.org/mathscinet-getitem?mr=#1}{#2}
}
\providecommand{\href}[2]{#2}

\end{document}